\documentclass[11pt]{amsart}

\usepackage[english]{babel}     
\usepackage[utf8]{inputenc}
\usepackage{amsmath}
\usepackage{amssymb}
\usepackage{mathrsfs}
\usepackage{stmaryrd}
\usepackage{hyperref}
\usepackage{xcolor}
\usepackage{todonotes}
\usepackage{comment}

\newtheorem{theorem}{Theorem}[section]
\newtheorem{lemma}[theorem]{Lemma}
\newtheorem{proposition}[theorem]{Proposition}
\newtheorem{corollary}[theorem]{Corollary}

\theoremstyle{definition}
\newtheorem{definition}[theorem]{Definition}
\newtheorem{example}[theorem]{Example}
\newtheorem{remark}[theorem]{Remark}

\theoremstyle{remark}
\numberwithin{equation}{section}

\newcommand{\Lebesgue}{\Lambda}

\let\epsilon\varepsilon

\newcommand{\att}{\big \vert}
\newcommand{\C}			{{\operatorname{\mathbb{C}}}}
\newcommand{\R}			{{\operatorname{\mathbb{R}}}}
\newcommand{\frakg}			{{\operatorname{\mathfrak{g}}}}

\newcommand{\frakh}			{{\operatorname{\mathfrak{h}}}}

\newcommand{\Vol}			{{\operatorname{\mathsf{Vol}}}}

\newcommand{\KKS}		{{\operatorname{\mathrm{KKS}}}}

\newcommand{\reg}		{{\operatorname{\mathrm{reg}}}}

\newcommand{\shl}	{{\operatorname{\mathfrak{sl}}}}

\newcommand{\cL}{\mathcal{L}}

\newcommand{\SAIN}{{\rm [SAIN]}}
\newcommand{\SIN}{{\rm [SIN]}}

\DeclareMathOperator{\ad}{ad}
\DeclareMathOperator{\supp}{supp}

\DeclareMathOperator{\id}{id}

\DeclareMathOperator{\End}{End}

\DeclareMathOperator{\Ad}{Ad}
\DeclareMathOperator{\rk}{rk}

\newcommand{\fg}{\mathfrak{g}}

\newcommand{\fh}{\mathfrak{h}}
\newcommand{\fk}{\mathfrak{k}}
\newcommand{\fa}{\mathfrak{a}}
\newcommand{\fp}{\mathfrak{p}}
\newcommand{\fz}{\mathfrak{z}}

\newcommand{\one}{\mathbf{1}}
\newcommand{\Det}{\mathrm{det}}
\newcommand{\cN}{\mathcal{N}}

\newcommand{\cO}{O}

\newcommand{\allone}{\mathbb{I}}

\newcommand{\ADS}{{\rm ADS}}
\newcommand{\SL}{\textrm{SL}}
\newcommand{\cS}{\mathcal{S}}

\title[Local and multilinear noncommutative de Leeuw theorems]{Local and multilinear noncommutative de Leeuw theorems}

\date{\noindent \today.  MSC2010 keywords: 22D25, 22E30, 46L51.  MC and AK are supported by the NWO Vidi grant VI.Vidi.192.018 `Non-commutative harmonic analysis and rigidity of operator algebras'.
BJ and LM are supported by the NWO Vidi grant 639.032.734 `Cohomology and representation theory of infinite dimensional Lie groups'. }

\author[Caspers]{Martijn Caspers}
\author[Janssens]{Bas Janssens}
\author[Krishnaswamy-Usha]{Amudhan Krishnaswamy-Usha}
\author[Miaskiwskyi]{Lukas Miaskiwskyi}

\address{TU Delft, EWI/DIAM,
	P.O.Box 5031,
	2600 GA Delft,
	The Netherlands}
\email{M.P.T.Caspers@tudelft.nl}
\email{B.Janssens@tudelft.nl}
\email{A.KrishnaswamyUsha@tudelft.nl}
\email{L.T.Miaskiwskyi@tudelft.nl}

\begin{document}

\begin{abstract}
Let $\Gamma < G$ be a discrete subgroup of a locally compact unimodular group $G$. Let
$m\in C_b(G)$ be a $p$-multiplier on $G$ with $1 \leq p < \infty$ and let
$T_{m}: L_p(\widehat{G}) \rightarrow  L_p(\widehat{G})$ be the corresponding Fourier multiplier. Similarly, let $T_{m \vert_\Gamma}: L_p(\widehat{\Gamma}) \rightarrow  L_p(\widehat{\Gamma})$ be the Fourier multiplier associated to the
restriction $m|_{\Gamma}$ of $m$ to $\Gamma$. We show that
\[
 c( \supp( m|_{\Gamma} ) )  \Vert T_{m \vert_\Gamma}: L_p(\widehat{\Gamma})   \rightarrow L_p(\widehat{\Gamma}) \Vert \leq
\Vert T_{m  }: L_p(\widehat{G}) \rightarrow L_p(\widehat{G}) \Vert,
\]
for a specific constant $0 \leq c(U) \leq 1$ that is   defined for every $U \subseteq \Gamma$.

 The function $c$ quantifies the failure of $G$ to admit small almost $\Gamma$-invariant neighbourhoods and can be  determined explicitly  in concrete cases.
In particular, $c(\Gamma) =1$ when $G$ has small almost $\Gamma$-invariant neighbourhoods. Our result thus extends the De Leeuw restriction theorem from \cite{CPPR} as well as De Leeuw's classical theorem \cite{DeLeeuw}.

For real reductive Lie groups $G$ we provide an explicit lower bound for
$c$ in terms of the maximal dimension $d$ of a nilpotent orbit in the adjoint representation. We show that $c(B_\rho^G) \geq \rho^{-d/4}$  where  $B_\rho^G$ is the ball of $g\in G$ with $\Vert \Ad_g \Vert < \rho$.

We further prove several results for multilinear Fourier multipliers. Most significantly, we prove a multilinear De Leeuw restriction theorem for pairs $\Gamma<G$ with $c(\Gamma) = 1$.  We also obtain  multilinear versions of the lattice approximation theorem, the compactification theorem and the periodization theorem. Consequently, we are able to provide the first examples of bilinear multipliers on nonabelian groups.
\end{abstract}

\maketitle

\section{Introduction}

In his seminal paper Karel de Leeuw \cite{DeLeeuw} proved several fundamental theorems about Fourier multipliers on the Euclidean space $\mathbb{R}^n$. Recall that   $m \in L_\infty(\mathbb{R}^n)$ is called a {\it $p$-multiplier} if the map $T_m$ on  $L_2(\mathbb{R}^n)$  determined by $\mathcal{F}_2(T_m f) =  m \mathcal{F}_2 (f)$ with $\mathcal{F}_2$ the unitary Fourier transform extends to a bounded map on $L_p(\mathbb{R}^n)$. Using Pontrjagin duality the definition of $p$-multipliers may be extended to any locally compact abelian group $G$ and we call $m \in L_\infty(G)$ a $p$-multiplier if the map $T_m$ on  $L_2(\widehat{G})$  determined by $\mathcal{F}_2(T_m f) =  m \mathcal{F}_2 (f)$ extends to a bounded map on $L_p(\widehat{G})$. Recall here that in case of $\mathbb{R}^n$ we have that  $ L_p(\widehat{\mathbb{R}^n})$ is isomorphic to $L_p(\mathbb{R}^n)$ canonically by conjugating with $\mathcal{F}_2$.

 One of De Leeuw's most important contributions is the {\it restriction theorem}. Let $H < \mathbb{R}^n$ be any subgroup equipped with the discrete topology. If $m \in C_b(\mathbb{R}^n)$ is a $p$-multiplier then $m\vert_H \in C_b(H)$ is a $p$-multiplier. Moreover,
 \begin{equation}\label{Eqn=RestrictIntro}
   \Vert  T_{m\vert_\Gamma}: L_p(\widehat{H}) \rightarrow L_p(\widehat{H}) \Vert \leq    \Vert  T_{m}: L_p(\widehat{\mathbb{R}^n}) \rightarrow L_p(\widehat{\mathbb{R}^n}) \Vert.
 \end{equation}
  A particular case occurs when $H = \mathbb{R}^n_{\mathrm{disc}}$ is $\mathbb{R}^n$ equipped with the discrete topology. Then $\widehat{\mathbb{R}^n_{\mathrm{disc}}}$ is the Bohr compactification of $\mathbb{R}^n$. This special case is usually referred to as the {\it compactification theorem} and is essentially the strongest theorem proved in \cite{DeLeeuw}. Its proof consists of first showing \eqref{Eqn=RestrictIntro} when $H$ is discrete in the Euclidean topology inherited from $\mathbb{R}^n$ combined with a {\it lattice approximation theorem}. These theorems became standard tools in harmonic analysis and are the first instances of so-called transference results.

 \vspace{0.3cm}

 The emergence of noncommutative integration has created a realm in which Fourier multipliers can be interpreted naturally for any locally compact group. Here function algebras on the Pontrjagin dual group are replaced by group  algebras and their C$^\ast$- and von Neumann closures.  The theory of noncommutative $L_p$-spaces \cite{Nelson}, \cite{PisierXu} and realization of Fourier transforms \cite{Kunze}, \cite{Cooney},  \cite{CaspersLpf} leads to a natural notion of $p$-multipliers. In this interpretation $G$ plays the role of the frequency side and for $m \in L_\infty(G)$ the Fourier multiplier $T_m$ acts on a noncommutative space, namely the noncommutative $L_p$-space of the group von Neumann algebra of $G$.

 For $p=\infty$ the relevance of Fourier multipliers on group von Neumann algebras was already recognized in the fundamental results of Haagerup on approximation properties \cite{HaagerupExample}, as well as the works by De Canni\`ere-Haagerup \cite{DeCanniereHaagerup} and Cowling-Haagerup \cite{CowlingHaagerup} on approximation properties of real rank one simple Lie groups.


There is a significant interest in understanding the class of $p$-multipliers beyond $p = \infty$. Already in commutative analysis this has triggered an enormous machinery of harmonic analysis including the study of singular integrals, Calder\'on-Zygmund theory, H\"ormander-Mikhlin multipliers, et cetera,  see e.g. \cite{GrafakosBook1},\cite{GrafakosBook2}, \cite{Stein}. The study of multipliers is further motivated by their applications to convergence properties of Fourier series, the structure of Banach spaces and the analysis of pseudo-differential operators.

In the noncommutative situation several efforts have been made in the understanding of $p$-multipliers. In \cite{JMPGafa} multipliers associated with cocycles on locally compact groups were constructed; such multipliers are compositions $m = \widetilde{m} \circ b$ with $b: G \rightarrow \mathcal{H}$ a cocycle into a finite dimensional orthogonal representation $(\mathcal{H}, \pi)$ with additional H\"ormander-Mikhlin conditions on $\widetilde{m}$. Further results in this direction have been obtained in \cite{JMPJEMS}, \cite{JungeGonzalezParcet}. Note that if $G$ has property (T) then a well-known theorem by Delorme-Guichardet \cite{Delorme}, \cite{Guichardet} (see \cite[Sections 2.2, 2.12]{TBook}) shows that every cocycle $b$ is bounded and therefore the multiplier $m = \widetilde{m} \circ b$ has a certain oscillatory behaviour. Note that if $G = \SL(n, \mathbb{R})$ the lack of non-trivial orthogonal representations limits this method.
However, concrete multipliers on $\SL(n, \mathbb{R})$ have been constructed in \cite{PRS} where a version of the H\"ormander-Mikhlin multiplier theorem was obtained for the Lie group $\SL(n, \mathbb{R})$, the real $n\times n$-matrices with determinant 1. Here the H\"ormander-Mikhlin conditions are formulated in the natural differentiable structure of the Lie group. In the discrete setting Cotlar's identity has successfully been applied \cite{MeiRicard}, \cite{MeiRicardXu} to obtain Hilbert transforms and H\"ormander-Mikhlin multipliers on  free groups.

\vspace{0.3cm}

It is natural to ask  which relations there are between $p$-multipliers on $G$ and $p$-multipliers on the discrete subgroups of $G$.  A cornerstone theorem was obtained in \cite{CPPR} where the De Leeuw restriction theorem was obtained for any pair $\Gamma < G$ with $\Gamma$ a discrete  subgroup  of a (for simplicity unimodular) locally compact group $G$ such that $G$ has small almost $\Gamma$-invariant neighbourhoods (see Definition \ref{Dfn=SAIN}). If $\Gamma$ is amenable then $G$ always has small almost $\Gamma$-invariant neighbourhoods \cite[Theorem 8.7]{CPPR}.   In particular the theorem is applicable to nilpotent locally compact groups.

There are several deep open problems that concern the relation of multipliers to their restrictions on subgroups beyond the cases covered by \cite{CPPR}.
For instance, the existence of explicit $p$-multipliers on $\SL(n, \mathbb{Z}), n \geq 3$ remains an open problem. In particular, there does not seem to be a genuine example of a $p$-multiplier,  $1 < p < \infty$, with compact support or with sufficient decay.  With genuine we mean that: (1)  the multiplier is not obtained from an interpolation between  $p=\infty$ and $p=2$; (2) the multiplier does not come from a subgroup, i.e. the support of the multiplier should generate  $\SL(n, \mathbb{Z})$, (3) the bound of the multiplier is sharper than the trivial estimate $\Vert T_{m}: L_p(\widehat{\Gamma})   \rightarrow L_p(\widehat{\Gamma}) \Vert \leq \Vert m \Vert_{L_1(\Gamma)}$ coming from the triangle inequality; (4) a combination of these. 
 This paper does not directly solve this problem but it provides a potential entry   as it would be sufficient to construct multipliers with sufficient bounds on $\SL(n, \mathbb{R})$; this also closely relates to the Calder\'on-Torchinsky problem mentioned in \cite[Section 5.C]{PRS}.
 Other questions are raised by the approximation and rigidity properties of noncommutative $L_p$-spaces obtained in \cite{JungeRuan}, \cite{LafforgueDeLaSalle}, \cite{DeLaat}, \cite{DeLaatDeLaSalle}, \cite[Theorem B]{PRS}.

\vspace{0.3cm}

Another direction that is initiated in this paper is the theory of multilinear Fourier multipliers on noncommutative $L_p$-spaces. In commutative harmonic analysis this became a large topic including milestone results by Lacey-Thiele on the bilinear Hilbert transform \cite{LaceyThiele} and the development of multilinear Calder\'on-Zygmund theory \cite{GrafakosTorres}. Also very recently, several results in the semi-commutative/vector-valued setting have been obtained \cite{Amenta}, \cite{DiPlinio1}, \cite{DiPlinio2}. In the noncommutative setting multilinear Schur multipliers and operator integrals have already provided very deep applications such as the resolution of Koplienko's conjecture  \cite{PSSInventiones}. The theory of Fourier multipliers -- which by transference is intimately related to Schur multipliers \cite{NeuwirthRicard}, \cite{CaspersDeLaSalle} -- seems to be undeveloped beyond abelian groups in the multilinear setting. We shall prove the multilinear versions of De Leeuw's restriction, lattice approximation and compactification theorem. These were proved in the abelian (bilinear) setting in \cite{Blasco}, \cite{rodriguez2013homomorphism}.

\vspace{0.3cm}

We now summarize the main results of this paper. We shall introduce a quantified version of having small almost invariant neighbourhoods. More precisely, we consider the following definition.

\begin{definition}\label{Dfn=DeltaNew}
Let $G$ be a locally compact group with (left) Haar measure $\mu$. Let $F \subseteq G$ be arbitrary and let $V \subseteq G$ be relatively compact of non-zero measure. We introduce the quantity
\[
\delta_F(V)  = \frac{\mu( \cap_{s \in F} \Ad_s(V) )}{ \mu(V)}.
\]
For a neighbourhood basis $\mathcal{V}$ of the identity of $G$ we set $\delta_F(\mathcal{V}) = \liminf_{V \in \mathcal{V}}  \delta_F(V)$. Then set $\delta_F$ to be the   supremum of    $\delta_F(\mathcal{V})$ over all {\it symmetric} neighbourhood bases  $\mathcal{V}$ of the identity.
\end{definition}

We  prove the following local version of the noncommutative De Leeuw restriction theorem  \cite[Theorem A]{CPPR}. We call this theorem `local' since we obtain an estimate that is controlled by a subset $U$ of $\Gamma$ and which holds for all symbols $m \in C_b(G)$ such that the restriction $m \vert_\Gamma$ is supported on $U$.

From this point let   $\Gamma$ be a discrete subgroup of any locally compact unimodular group $G$.

\vspace{0.3cm}

\noindent {\bf Theorem A.}
  Let $m \in C_b( G)$. Then for every $1 \leq  p < \infty$ we have that
\begin{equation}
c(\supp(m\vert_\Gamma) )
\cdot \Vert T_{m \vert_\Gamma}: L_p(\widehat{\Gamma})   \rightarrow L_p(\widehat{\Gamma}) \Vert \leq
\Vert T_{m  }: L_p(\widehat{G}) \rightarrow L_p(\widehat{G}) \Vert,
\end{equation}
where $c( U ) = \inf \{ \delta_F^{\frac{1}{2}} \mid F \subseteq U,  F \: {\rm finite} \} $ is defined for every $U \subseteq \Gamma$.

\vspace{0.3cm}

Theorem A holds as well if the bounds of both multipliers are replaced by their complete bounds (see proof of Proposition \ref{Prop=AlmostIsometrySymmetric}).
The strength of Theorem A is that the constant $c$ can be determined explicitly in many interesting situatons.   We have $c(\Gamma) = 1$ if and only if $G$ has small almost $\Gamma$-invariant neighbourhoods. Hence, our theorem recovers \cite[Theorem A]{CPPR}.   In Sections \ref{Sect=LowerBound}, \ref{Sect=Measures} and \ref{Sect=KeyLemma} we find very natural lower estimates on $\delta_F$ in the full generality of real reductive Lie groups.  For $\rho \geq 1$ let  $B_\rho^G = \{ g \in G \mid \Vert \Ad_g \Vert \leq \rho \}$ where $\Ad$ is the adjoint representation.

\vspace{0.3cm}

\noindent {\bf Theorem B.}
Let $G$ be a real reductive Lie group.
For $\rho > 1$ we have $\delta_{B_\rho^G}  \geq \rho^{-d/2}$ where $d$ is the maximal dimension of a nilpotent adjoint orbit.

\vspace{0.3cm}

In order to prove Theorem B we first reduce it to the connected adjoint group of $G$, which is semisimple. We then construct a symmetric neighbourhood basis $V^{\fg}_{\varepsilon, R}, \epsilon >0, R >0$ of 0 in the Lie algebra $\mathfrak{g}$ such that
\[
\lim_{R\searrow 0}\lim_{\varepsilon\searrow 0}
\frac{ \mu( \exp(V^{\fg}_{\varepsilon,  R/\rho}) )    }{  \mu( \exp(V^{\fg}_{\varepsilon, R}) )     }
=
\rho^{-d/2}.
\]
This implies  Theorem B.
The neighbourhood basis is constructed in such a way that $\bigcap_{\varepsilon > 0} V^{\fg}_{\varepsilon, R}$ is the intersection of the nilpotent cone $\cN \subseteq \fg$ with the ball of radius $R$ around the origin.
Note that the union of the nilpotent orbits $\cO_{X}$ of maximal dimension $d$ is dense in $\cN$.
Using results of Harisch-Chandra and Barbasch-Vogan on limiting orbit integrals, we show that the scaling behaviour of $\mu(  \exp(V^{\fg}_{\varepsilon, R}) )$ is governed by the scaling behaviour of the Liouville form of $\cO_{X}$. Since the KKS symplectic form $\omega^{\KKS}$ scales as $\rho$ under dilation, the Liouville form $\frac{1}{(d/2)!}\wedge^{d/2}\omega^{\KKS}$ scales as $\rho^{d/2}$.

\vspace{0.3cm}

In Example  \ref{Example=SLNR}  we consider the concrete case $\Gamma = \SL(n, \mathbb{Z}), G = \SL(n, \mathbb{R})$. We discuss how Theorems A and B give an ansatz to construct multipliers on  $\Gamma = \SL(n, \mathbb{Z})$.

\vspace{0.3cm}

Next we start the analysis of multilinear multipliers. Most efforts   are required for the restriction theorem.

\vspace{0.3cm}

\noindent {\bf Theorem C.}
 Let $G$ be a locally compact unimodular group. Let $\Gamma < G$ be a discrete subgroup such that $c(\Gamma)= 1$.  Let $m \in C_b( G^{\times n})$. Then for every $1 \leq p, p_1, \ldots, p_n < \infty$ with $p^{-1} = \sum_{j=1}^n p_j^{-1}$ we have that
\[
\Vert T_{m \vert_{\Gamma^{\times n}} }: L_{p_1}(\widehat{\Gamma}) \times \ldots \times L_{p_n}(\widehat{\Gamma})   \rightarrow L_p(\widehat{\Gamma}) \Vert \leq
\Vert T_{m  }: L_{p_1}(\widehat{G}) \times \ldots \times L_{p_n}(\widehat{G})  \rightarrow L_p(\widehat{G}) \Vert.
\]

\vspace{0.3cm}

  Theorem C hinges on the intertwining property of Lemma \ref{Lem=AuxConverge} whose proof is very delicate and which requires several new ideas compared to its linear counterpart. The global idea is to approximate $m$ with multipliers $m_k$ that  can be expressed in terms of nested compositions of linear $\infty$-multipliers. We then construct asymptotically isometric maps that  intertwine the Fourier multipliers associated with  $m_k\vert_{\Gamma^{\times n}}$ and $m_k$.
 The condition $c(\Gamma) = 1$ is used at several places in the proof, whereas in the linear case it is only needed to construct asymptotically isometric maps. This also explains why we do not obtain a `local' version of Theorem C as well.

We further obtain multilinear versions of the lattice approximation theorem (Theorem \ref{thm:LatticeApproximation}) as well as the multilinear compactification theorem and the periodization theorem  (Section \ref{Sect=OtherTheorems}).  We use these theorems to construct examples of multilinear Fourier multipliers on the Heisenberg group. These are the first genuine examples of multilinear Fourier multipliers on a nonabelian group.

\vspace{0.3cm}

\noindent {\it Structure of the paper.}  Section \ref{Sect=NCLp} contains preliminaries on noncommutative $L_p$-spaces of group von Neumann algebras. Section \ref{Sect=LocalDeLeeuw} proves Theorem A. We also establish approximate embeddings of the non-commutative $L_p$-spaces assocatied with an inclusion of groups $\Gamma < G$. Section \ref{Sect=MultilinearRestriction} defines multilinear Fourier multipliers and proves Theorem C. The proofs of the remaing De Leeuw theorems are contained in Sections \ref{Sect=Lattice} and \ref{Sect=OtherTheorems}; this concerns the lattice approximation, compactification and periodization. Section \ref{Sect=Heisenberg} contains examples on the Heisenberg group.  In Sections \ref{Sect=LowerBound}, \ref{Sect=Measures}, \ref{Sect=KeyLemma} we prove Theorem B; we have postponed the proof of Theorem B to the end of the paper so that all Lie theoretic  arguments are presented separately.

\vspace{0.3cm}

\noindent {\it Acknowledgement.}  BJ and LM are grateful to Tobias Diez for interesting discussions and for providing crucial references. MC and AK wish to thank Adri\'an Gonz\'alez-P\'erez, Javier Parcet and \'Eric Ricard for useful discussions and communication.  The authors thank Gerrit Vos and the anonymous referee for suggesting a number of corrections to an earlier version of this manuscript.

\section{\texorpdfstring{Non-commutative $L_p$-spaces of group von Neumann algebras}{Non-commutative Lp-spaces of group von Neumann algebras}}\label{Sect=NCLp}
$\mathbb{N}$ denotes the natural numbers starting from 1. We denote $\mathbb{N}_{\geq 0} = \mathbb{N} \cup \{ 0 \}$.

\subsection{Assumptions on groups} All groups $G$ in this paper are assumed to be  locally compact, second countable and unimodular. Though our results can be stated without the second countability assumption, it significantly simplifies the exposition of our proofs as we can work with neighbourhood bases of the identity instead of shrinking nets of neighbourhoods. When we say that $\Gamma$ is a discrete subgroup of $G$ we mean that it is discrete in the topology of $G$. Any discrete subgroup $\Gamma < G$ is then countable.

For $s,t \in G$ we write $\Ad_s(t) = s t s^{-1}$. We denote $\mu$ for the Haar measure on $G$.    We use the shorthand notation $ds = d\mu(s)$ for integrals with respect to $\mu$. A set $V \subseteq G$ is called symmetric if $V = V^{-1}$. A neighbourhood basis of the identity is called symmetric if it consists of symmetric sets.

\subsection{Von Neumann algebras}  A von Neumann algebra $M$ is a unital $\ast$-subalgebra of bounded operators on a Hilbert space that is closed in the strong topology. $M$ is called semi-finite if it admits a faithful normal trace $\tau: M^+ \rightarrow [0, \infty]$. $M$ is finite in case there exists such $\tau$ with $\tau(1) = 1$, i.e. $\tau$ extends to a state on $M$.For general von Neumann algebra theory and non-commutative integration we refer to \cite{TakesakiI}, \cite{StratilaZsido}.

\subsection{Group von Neumann algebras} Let $C_b(G)$ be the continuous bounded functions $G \rightarrow \mathbb{C}$ and let $C_c(G)$ be the subspace of compactly supported functions.  Let $L_p(G), 1 \leq p < \infty$ be the Banach space of functions $G \rightarrow \mathbb{C}$ that are $p$-integrable, meaning $\Vert f \Vert_p := (\int_G \vert f(s) \vert^p ds)^{1/p} < \infty$.    For $f \in L_p(G), h \in L_1(G)$ we have a convolution product $f \ast h \in L_p(G)$ determined by
\[
(f \ast h)(s) = \int_G f(t) h(t^{-1} s) dt, \qquad s \in G.
\]
We further set
\[
f^\ast(s) = \overline{ f(s^{-1})}, \qquad f^\vee(s) = f(s^{-1}), \qquad s \in G.
\]
Let
\[
s \mapsto \lambda(s), \qquad (\lambda(s)\xi)(t) = \xi(s^{-1} t), \qquad s,t \in G,  \xi \in L_2(G),
\]
be the left regular representation. Set $\lambda(f) = \int_G f(s) \lambda(s) d(s)$ where the integral is convergent in the strong topology. We have
\begin{equation}\label{Eqn=ConvolutionLambda}
\lambda(f \ast h) = \lambda(f) \lambda(h), \qquad f,h \in L_1(G).
\end{equation}
We set the group von Neumann algebra
\[
\cL(G) = \overline{{\rm span}} \{ \lambda(s) \mid s \in G \}  =  \overline{{\rm span}} \{ \lambda(f) \mid f \in L_1(G) \},
\]
where $\overline{{\rm span}}$ denotes the strong closure of the linear span.
There exists a unique normal semi-finite faithful weight $\varphi_G$ on $\cL(G)$ that is defined as follows. For $x \in \cL(G)$ we have
\begin{equation}\label{Eqn=PlancherelWeight}
\varphi_G(x^\ast x)  =
\left\{
\begin{array}{ll}
\Vert f \Vert_2^2, & \textrm{ if } \exists f \in L_2(G) \textrm{ s.t. } \forall \xi \in C_c(G):  x \xi = f \ast \xi, \\
\infty, & \textrm{ otherwise}.
\end{array}
\right.
\end{equation}
 $\varphi_G$ is tracial, meaning that for all $x\in \cL(G), \varphi_G(x^\ast x) = \varphi_G(x x^\ast)$, iff $G$ is unimodular, which we will always assume.    $\varphi_G$ is a state if and only if $G$ is discrete.

\subsection{Crossed products}
 Let $M$ be a semi-finite von Neumann algebra with a trace $\tau_M$ and let $\Gamma$ be a discrete group acting on $M\subseteq B(H)$ via a trace preserving action $\theta:\Gamma \to \mathrm{Aut}(M)$. For $x \in M$, define the operator $\iota (x) \in B(H\otimes L_2(\Gamma))$ by
 \[
 (\iota(x)\xi)(g) = (\theta(g^{-1}) x) ( \xi(g) ), \qquad \xi \in L_2(\Gamma;H), g \in \Gamma
 \]
 The crossed product $M \rtimes_\theta \Gamma \subseteq B(H\otimes L_2(\Gamma))$ is the von Neumann algebra generated by $\{ 1 \otimes \lambda(g) \mid  g\in \Gamma \}$ and $\{ \iota(x) \mid x \in M \}$. This has a natural tracial weight $\tau$ which extends $\tau_M $ on $\{\iota(x) \mid x\in M\}$ and $\tau_\Gamma$ on ${1\otimes \lambda(g), g \in \Gamma}$ (c.f. \cite{TakesakiI}).

\subsection{\texorpdfstring{Non-commutative $L_p$-spaces}{Non-commutative Lp-spaces}} The results in this section can be found in \cite{Kunze},  \cite{Nelson}, \cite{TerpThesis}, \cite{PisierXu},  \cite{CaspersLpf}. For an exponent $p \in [1, \infty]$ we will write $p' \in [1,\infty]$ for the conjugate exponent set by $\frac{1}{p} + \frac{1}{p'} = 1$.
 Let $M$ be a von Neumann algebra equipped with a normal, semi-finite, faithful trace $\tau$. For $1 \leq p < \infty$, set
\[
\Vert x \Vert_p := \tau( \vert x \vert^p  )^{\frac{1}{p}}, \qquad x \in M.
\]
We define $L_p(M, \tau)$ as the completion of $\{ x \in M \mid \Vert x \Vert_p < \infty \}$ with respect to $\Vert \: \cdot \: \Vert_p$. Alternatively,
 $L_p(M, \tau)$ may be described as the space of all closed, densely defined operators $x$ that are affiliated with $M$ and for which $\Vert x \Vert_p^p := \tau(\vert x \vert^p )  < \infty$. The latter description is more concrete but requires the introduction of affiliated operators and extension of the trace thereon; this shall not be used further in this paper. We define $L_\infty(M, \tau)$ to be  $M$ with the operator norm. Then $L_p(M, \tau), 1 \leq p \leq \infty$ is a Banach space and we have a H\"older inequality
 \[
 \Vert x y \Vert_r \leq \Vert x \Vert_p \Vert y \Vert_q, \qquad \frac{1}{r} = \frac{1}{p} + \frac{1}{q}, \qquad 1 \leq p,q,r \leq \infty.
 \]
The H\"older inequality is moreover sharp in the sense that
\begin{equation}\label{Eqn=MaxHolder}
    \Vert x \Vert_p = \sup_{y \in L_q(M, \tau), y \not = 0}   \frac{\Vert x y \Vert_r}{\Vert y \Vert_q};
\end{equation}
indeed if $x = u \vert x \vert$ is the polar decomposition of $x$ then the supremum is attained at $y = \vert x \vert^{\frac{p}{q}}$.
 The trace $\tau$ may be extended from $M \cap L_1(M, \tau)$ to $L_1(M, \tau)$ linearly and continuously in $\Vert \: \cdot \:  \Vert_1$.
 For $x \in L_p(M, \tau)$ and $y \in L_{p'}(M, \tau)$ we have  in particular that $yx \in L_1(M, \tau)$ and we have a pairing
 \begin{equation}\label{Eqn=DualPairing}
 \langle y, x \rangle_{p', p} = \tau( y x).
 \end{equation}
We have $\Vert x \Vert_p = \Vert x^\ast \Vert_p$ for any $x \in L_p(M, \tau)$.

In case $(M, \tau)$ is $(\cL(G), \varphi_G)$ with $G$ a unimodular locally compact group - so that $\varphi_G$ is tracial - we simply write $L_p(\widehat{G})$ for $L_p( \cL(G), \varphi_G)$. If $G$ is abelian the latter space is isomorphic to the usual $L_p$-space of the Pontrjagin dual group $\widehat{G}$.
For $2 \leq p \leq \infty$ and  $f \in C_c(G)$ we have $\lambda(f) \in L_{p}(\widehat{G})$ with $\Vert \lambda(f) \Vert_{p} \leq \Vert f \Vert_{p'}$ (see \cite{Kunze}, \cite{Cooney},  \cite{CaspersLpf}) and such elements are dense. We denote $C_c(G)^{\ast 2} = C_c(G) \ast C_c(G)$ for the second convolution power of $C_c(G)$.
  The elements $\lambda(f),  f  \in C_c(G)^{\ast 2}$  are dense in $L_{p}(\widehat{G}), 1 \leq p <\infty$. The same densities hold for $p = \infty$ but then in the $\sigma$-weak topology of $L_\infty(\widehat{G}) = \cL(G)$. In each case the {\it frequency support} of $\lambda(f)$ is by definition the support of $f$. For $\phi, f \in \lambda(  C_c(G)^{\ast 2} )$  we have
\begin{equation}\label{Eqn=PairingConcrete}
\langle \lambda(\phi^\vee), \lambda(f) \rangle_{p', p} = \int_G \phi(s) f(s) ds.
\end{equation}
It follows directly from \eqref{Eqn=PlancherelWeight} that the {\it Plancherel identity} holds:
\[
  \Vert \lambda(f) \Vert_2 = \Vert f \Vert_2, \qquad f \in L_1(G) \cap L_2(G).
\]
 In case $G = \Gamma$ is discrete the elements with finite frequency support are given by the group algebra
\[
\mathbb{C}[\Gamma] := \{ \lambda(f) \mid f \in c_{00}(\Gamma) \},
\]
which is dense in $L_p(\widehat{\Gamma})$ for any $1 \leq p < \infty$. Here $c_{00}(\Gamma)$ denotes the finitely supported functions $\Gamma \rightarrow \mathbb{C}$.

 \begin{remark}
 There are notions of non-commutative $L_p$ associated with an arbitrary von Neumann algebra due to Haagerup \cite{TerpThesis} and Connes-Hilsum \cite{Hilsum}. This allows for the study of de Leeuw theorems for non-unimodular groups, see for instance \cite[Section 8]{CPPR}. However, to keep the current paper broadly accessible we will work within the realm of unimodular groups and tracial von Neumann algebras.
 \end{remark}

\subsection{Fourier multipliers}  We say that a  function $m \in C_b( G)$ is a $p \rightarrow q$-multiplier with $1 \leq p,q < \infty$ if there exists a bounded linear map $T_m: L_p(\widehat{G}) \rightarrow L_q(\widehat{G})$ that is determined by
\[
T_m(\lambda(f)) = \lambda(mf), \qquad f \in C_c(G)^{\ast 2}.
\]
In particular this encompasses that $\lambda(mf) \in L_q(\widehat{G})$.
We briefly say $p$-multiplier  in case $p=q$. By the Plancherel identity the space of $2$-multipliers with continuous symbol is $C_b(G)$.

\begin{remark}\label{Rmk=Dual}
Let $2 \leq p < \infty$.
Suppose that $m \in C_b(G)$ is a $p$-multiplier. Then under the pairing  \eqref{Eqn=PairingConcrete} the dual of $T_m: L_p(\widehat{G}) \rightarrow L_p(\widehat{G})$ is the $p'$-multiplier given by $T_{m^\vee}: L_{p'}(\widehat{G}) \rightarrow L_{p'}(\widehat{G})$.
\end{remark}

\section{Theorem A: The local linear De Leeuw restriction theorem}\label{Sect=LocalDeLeeuw}
The aim of this section is to prove Theorem A. We also introduce the quantified version of local almost invariant neighbourhoods in Definition \ref{Dfn=DeltaNew} which we shall determine for real reductive Lie groups in Section \ref{Sect=LowerBound}.
We fix again a locally compact unimodular group $G$ which we assume to be second countable.   Let $\Gamma <  G$ be a discrete subgroup. Recall that $\mu$ is the Haar measure of $G$.

At this stage recall that $\delta_F$ as was defined in Definition \ref{Dfn=DeltaNew}.

\begin{definition}[See \cite{CPPR}]\label{Dfn=SAIN}
For a closed subgroup $H < G$ we say that $G$ has {\it small almost invariant neighbourhoods}  with respect to $H$  (notation $G \in \SAIN_H$) if for every $F \subseteq H$ finite we have $\delta_F = 1$. Equivalently, $c(H) = 1$ where $c$ is defined in Definition \ref{Dfn=Cfunction} or the introduction of this paper.
\end{definition}

\begin{remark}\label{Rmk=SAIN-Delta}
By \cite[Theorem 8.7]{CPPR}  if the discrete group $\Gamma$ is  amenable   then $G \in \SAIN_\Gamma$.
\end{remark}

\begin{remark}
  $G$ has small invariant neighbourhoods (notation $G \in \SIN_\Gamma$) with respect to $\Gamma$ if  there exists a neighbourhood basis $\mathcal{V}$ of the identity of $G$ such that $\Ad_s(V) = V, s \in \Gamma$. Clearly then $G \in \SAIN_\Gamma$. Indeed, we may replace $V \in \mathcal{V}$ by $V \cap V^{-1}$ to obtain a symmetric neighbourhood basis with the desired properties.
\end{remark}

\begin{remark}\label{Rmk=Shrink}
Suppose that $G \in \SAIN_\Gamma$. Since $G$ is first countable pick any neighbourhood basis $(U_i)_{i \in \mathbb{N}}$ of the identity of $G$. Since $G$ is second countable,  $\Gamma$ is countable.  Let $F_i \subseteq \Gamma, i \in \mathbb{N}$ be an increasing sequence of finite subsets whose union is $\Gamma$. For every $i \in \mathbb{N}$ we may inductively pick a relatively compact open neighbourhood $V_i$ of the identity in $G$ such that $1 - \frac{1}{i} \leq \delta_{F_i}(V_i) \leq 1$,  $V_{i} \subseteq V_{i-1}$ and $V_{i} \subseteq U_i$. Then $\mathcal{V} = (V_i)_{i \in \mathbb{N}}$ is a neighbourhood basis of the identity such that for any $F \subseteq \Gamma$ finite we have
\[
\lim_{V \in \mathcal{V}} \delta_F(V) = 1.
\]
   This construction has the advantage that we may fix a single neighbourhood basis $\mathcal{V}$ independent of the finite set $F$.
\end{remark}

\subsection{\texorpdfstring{Asymptotic embeddings for $2 \leq p < \infty$}{Asymptotic embeddings for p larger than two}}\label{Sect=Embeddings}
The aim of this section is to construct contractive maps $L_p(\widehat{\Gamma}) \rightarrow L_p(\widehat{G})$ that are asymptotic embeddings. In case $G \in \SAIN_\Gamma$ these maps are asymptotically isometric. Our results therefore generalize \cite[Claim A]{CPPR}.

\begin{lemma}\label{Lem=PositiveDef}
Let $V \subseteq G$ be a measurable set with $0 \not =  \mu(V)  < \infty$. Let $F \subseteq G$ be  finite. The matrix $A = (A_{s,t})_{s,t \in F}$ given by
\[
A_{s,t} = \frac{\mu( \Ad_{s}( V)  \cap \Ad_{t}( V)   )}{ \mu( V ) },
\]
is positive definite. Moreover $A \geq \delta_F(V) \allone$ where $\allone = (1)_{s,t \in F}$ is the matrix with all entries equal to 1.
\end{lemma}
\begin{proof}
Let $V_0 := \bigcap_{s\in F} \Ad_{s}(V)$ and define $\xi:  F \rightarrow L_2(G)$ by
\[
\xi(s) := \mu(V)^{-1/2}(1_{\Ad_s(V)} - 1_{V_0}).
\]
 Then $\langle \xi(s), \xi(t) \rangle = A_{s,t} - \frac{ \mu(V_0) }{\mu(V) }$ is a positive definite kernel on $F \times F$. Since $\delta_{F}(V) = \frac{\mu(V_0)}{\mu(V)}$ the result follows.
 \end{proof}

As usual we view $\mathbb{C}[\Gamma] \subseteq \cL(\Gamma) \subseteq \cL(G)$ as subalgebras of each other naturally.

\begin{proposition}\label{Prop=AlmostIsometry1}
  Let $x \in \mathbb{C}[\Gamma]$ with finite frequency support $F \subseteq \Gamma$. Let  $V$ be a relatively compact symmetric neighbourhood of the identity of $G$ and assume that   $s V,  s \in F$ are disjoint.
Set
$k_V =  \mu(V)^{-\frac{1}{2}}   \lambda( 1_{V} ) \in L_2(\widehat{G})$ with polar decomposition $k_V = u_V h_V$. Then set
\[
\Phi_{p,V}(x) = x  h_V^{\frac{2}{p}}, \qquad 2 \leq p \leq \infty,
\]
where we use the notation $\Phi_{\infty,V}(x) = x$.
Then for every $2 \leq p \leq \infty$ we have
\begin{equation}\label{Eqn=LocalLeeuwUp}
 \Vert \Phi_{p,V}(x) \Vert_{L_p(\widehat{G})} \leq  \Vert x \Vert_{L_p(\widehat{\Gamma})}.
\end{equation}
\end{proposition}
\begin{proof}
Write $x = \sum_{s \in F} x_s \lambda(s)$. By  Remark \ref{Rmk=kVhV} below and the Plancherel identity
\[
\Vert \Phi_{2,V}(x) \Vert_{L_2(\widehat{G})} =   \Vert  x h_V \Vert_{L_2(\widehat{G})} = \Vert  x k_V \Vert_{L_2(\widehat{G})} = \mu(V)^{-\frac{1}{2}} \Vert \sum_{s \in F}  x_s 1_{s V} \Vert_{L_2(G)}.
\]
Since $sV, s \in F$ are disjoint and using once more the  Plancherel identity,
\begin{equation}\label{Eqn=PhiPlancherel}
\Vert \Phi_{2,V}(x) \Vert_{L_2(\widehat{G})} = ( \sum_{s \in F} \vert x_s \vert ^2 )^{\frac{1}{2}} = \Vert x \Vert_{L_2(\widehat{\Gamma})}.
\end{equation}
This gives the upper estimate \eqref{Eqn=LocalLeeuwUp} for $p =2$. For $p = \infty$ the   estimate \eqref{Eqn=LocalLeeuwUp} is trivial. From the three lines lemma (similar to Stein's interpolation theorem for analytic families of maps \cite{SteinInterpolation}) we then have  \eqref{Eqn=LocalLeeuwUp} for all $2 \leq p \leq \infty$.
\end{proof}

\begin{remark}\label{Rmk=kVhV}
In Proposition \ref{Prop=AlmostIsometry1} since $V = V^{-1}$ it follows that $k_V$ is self-adjoint. Hence $h_V$ commutes with $u_V$ and $u_V$ is a self-adjoint partial isometry with $u_V^2$ being the support projection of $h_V$. We shall repeatedly make use of these observations without further reference.
\end{remark}

\begin{proposition}\label{Prop=AlmostIsometry2}
  Let $x \in \mathbb{C}[\Gamma]$ with finite frequency support $F \subseteq \Gamma$.
For $\mathcal{V}$ a symmetric neighbourhood basis of the identity of $G$  we have  for every $2 < p <\infty$
\begin{equation}\label{Eqn=LocalLeeuwDown}
\liminf_{V \in \mathcal{V}} \delta_F^{\frac{1}{2}}(V)
\Vert x \Vert_{L_p(\widehat{\Gamma} ) } \leq  \liminf_{V \in \mathcal{V}}  \Vert \Phi_{p,V}(x) \Vert_{L_p(\widehat{G})}.
\end{equation}
\end{proposition}

\begin{proof}
For $p = 2$ the proof is \eqref{Eqn=PhiPlancherel} and for $p=\infty$ the statement is obvious as $\cL(\Gamma)$ is a von Neumann subalgebra of $\cL(G)$. So we assume that $2 < p < \infty$.
Let $2 < q < \infty$ be the $L_2$-conjugate of $p$ determined by $p^{-1} +  q^{-1} = 2^{-1}$.
Let $\varepsilon >0$ and let  $y \in \mathbb{C}[\Gamma]$  be such that $\Vert y \Vert_{L_q(\widehat{\Gamma})} \leq 1$ and $\Vert x y \Vert_{L_2(\widehat{\Gamma})} > \Vert x \Vert_{L_p(\widehat{\Gamma})} - \varepsilon$, see \eqref{Eqn=MaxHolder}. Let $F_y$ be the inverse set of the frequency support of $y$ and write
\[
x = \sum_{s \in F} x_s \lambda(s), \qquad y = \sum_{s \in F_y} y_{s^{-1}} \lambda(s), \qquad x_s, y_s \in \mathbb{C}.
\]
We may consider a tail of the symmetric basis $\mathcal{V}$ so that each $V \in \mathcal{V}$ is relatively compact and so that each of the following families are disjoint: (1) $sV, s \in F$; (2) $s V, s \in F_y$. Moreover, we may assume that (3) $s_1Vt_1$ and $s_2Vt_2$ are disjoint whenever $s_1, s_2 \in F$, $t_1, t_2 \in F_y$ with  $s_1 t_1 \neq s_2 t_2$.

 For $V \in \mathcal{V}$ set $\Psi_{q,V}(y) = u_V h_V^{\frac{2}{q}} y$.
It follows that
$\Psi_{q,V}(y) =  u_V (y^\ast h_V^{\frac{2}{q}})^\ast   =  u_V \Phi_{q,V}(  y^\ast  )^\ast$.
 The frequency support of $y^\ast$ is $F_y$.  Since we assumed that $s V, s \in F_y$ are disjoint it follows from \eqref{Eqn=LocalLeeuwUp} -- but then applied to $y$ -- that,
\begin{equation}\label{Eqn=Contraction}
\Vert \Psi_{q,V}(y) \Vert_{L_q(\widehat{G})} = \Vert \Phi_{q,V}(  y^\ast  )  \Vert_{L_q(\widehat{G})}   \leq \Vert y^\ast \Vert_{L_q(\widehat{\Gamma})} =  \Vert y \Vert_{L_q(\widehat{\Gamma})}.
\end{equation}
Using \eqref{Eqn=MaxHolder} and \eqref{Eqn=Contraction} we find for $V \in \mathcal{V}$,
\begin{equation}\label{Eqn=EstimateByConjugate}
\Vert \Phi_{p,V}(x) \Vert_{L_p(\widehat{G})} = \sup_{z \in L_q(\widehat{G}), \Vert z \Vert_{L_q(\widehat{G})} \leq 1 }  \Vert \Phi_{p,V}(x) z \Vert_{L_2(\widehat{G})} \geq  \Vert \Phi_{p,V}(x) \Psi_{q,V}(y) \Vert_{L_2(\widehat{G})}.
\end{equation}
 We have
\[
\begin{split}
\Vert \Phi_{p,V}(x)  \Psi_{q,V}(y) \Vert_{L_2(\widehat{G})} = &  \Vert  x  h_V^{\frac{2}{p} + \frac{2}{q}} u_V y  \Vert_{L_2(\widehat{G})}
= \Vert x k_V y\Vert_2 = \mu(V)^{-\frac{1}{2}}
\Vert  \sum_{s\in F, t \in F_y}  x_s y_t   \lambda( 1_{ s V t   } )  \Vert_{L_2(\widehat{G})}.
\end{split}
\]
By the Plancherel identity and disjointness assumption (3) in the first paragraph of this proof,
\begin{equation}\label{Eqn=IsometryProof1}
\begin{split}
\Vert \Phi_{p,V}(x)  \Psi_{q,V}(y) \Vert_{L_2(\widehat{G})}^2 = & \mu(V)^{-1}
\int_G \sum_{s_1, s_2 \in F, t_1, t_2 \in F_y}  x_{s_1} y_{t_1}  \overline{  x_{s_2}  y_{t_2 }}    1_{ s_1 V t_1  }(g)    1_{ s_2  V t_2 } (g) dg \\
= & \sum_{r \in \Gamma} \sum_{\substack{ s_1, s_2 \in F, t_1, t_2 \in F_y \\ s_1 t_1 = r = s_2 t_2}}  x_{s_1} y_{t_1}  \overline{  x_{s_2}  y_{t_2 }}  \frac{  \mu(   s_1 V s_1^{-1} \cap s_2   V s_2^{-1}  )  }{ \mu( V) }\\
= & \sum_{r \in \Gamma} \sum_{  s_1 ,   s_2 \in F \cap r F_y^{-1} }  x_{s_1} y_{s_1^{-1} r}  \overline{  x_{s_2}  y_{s_2^{-1} r } }  \frac{  \mu( s_1 V s_1^{-1} \cap s_2   V s_2^{-1} )  }{ \mu(V)  }. \\
\end{split}
\end{equation}
By Lemma \ref{Lem=PositiveDef} we find for each summand $r \in \Gamma$ that
\begin{equation}\label{Eqn=IsometryProof2}
 \sum_{r \in \Gamma}   \sum_{  s_1 ,   s_2 \in F \cap r F_y^{-1}  }  x_{s_1} y_{s_1^{-1} r}  \overline{  x_{s_2}  y_{s_2^{-1} r } }  \frac{  \mu( s_1 V s_1^{-1} \cap s_2   V s_2^{-1} )  }{ \mu(V)  }
 \geq
\sum_{r \in \Gamma}    \sum_{s_1 ,   s_2 \in F \cap r F_y^{-1}}  x_{s_1} y_{s_1^{-1} r}  \overline{  x_{s_2}  y_{s_2^{-1} r } }    \delta_F(V).
 \end{equation}
Combining \eqref{Eqn=IsometryProof1} and  \eqref{Eqn=IsometryProof2} we get
\[
\begin{split}
\Vert \Phi_{p,V}(x)  \Psi_{q,V}(y) \Vert_{L_2(\widehat{G})}^2 \geq & \delta_F(V) \sum_{r \in \Gamma}   \sum_{   s_1 ,   s_2 \in F \cap r F_y^{-1}  }  x_{s_1} y_{s_1^{-1} r}  \overline{  x_{s_2}  y_{s_2^{-1} r } }  \\
= & \delta_F(V) \Vert \sum_{r \in \Gamma} \sum_{\substack{ s_1  \in F, t_1 \in F_y \\ s_1 t_1 = r}}  x_{s_1} y_{t_1} \lambda(r) \Vert_{L_2(\widehat{\Gamma} )}^2 \\
=& \delta_F(V)  \Vert  xy \Vert_{L_2(\widehat{\Gamma})}^2 \geq
\delta_F(V)  (\Vert x \Vert_{L_p(\widehat{\Gamma})} - \varepsilon )^2.
\end{split}
\]
  Hence by \eqref{Eqn=EstimateByConjugate} we get that
\[
\Vert \Phi_{p,V}(x) \Vert_{L_p(\widehat{G})} \geq \delta_F(V)^{\frac{1}{2}} (\Vert x \Vert_{L_p(\widehat{\Gamma})} - \varepsilon ).
\]
So certainly
$\liminf_{V \in \mathcal{V}} \Vert \Phi_{p,V}(x) \Vert_{L_p(\widehat{G})} \geq \liminf_{V \in \mathcal{V}} \delta_F(V)^{\frac{1}{2}} (\Vert x \Vert_{L_p(\widehat{\Gamma})} - \varepsilon )$. Since this holds for every $\varepsilon >0$ we get that
\[
\liminf_{V \in \mathcal{V}} \Vert \Phi_{p,V}(x) \Vert_{L_p(\widehat{G})} \geq \liminf_{V \in \mathcal{V}} \delta_F(V)^{\frac{1}{2}} \Vert x \Vert_{L_p(\widehat{\Gamma})}.
\]
\end{proof}

\subsection{Asymptotic intertwiners}

 We will need to make use of the following proposition to prove our De Leeuw restriction theorems: both in the local linear and in the multilinear setting.

\begin{proposition}[Claim B on p. 24 of \cite{CPPR}]\label{Prop=Intertwiner}
Let $\mathcal{V}$ be symmetric neighbourhood basis of the identity of $G$.
Let $2 \leq q < p < \infty$ or $1 \leq p < q \leq 2$.
  Let $m \in C_b(G)$ be a  $p$-multiplier. Then for $x \in \mathbb{C}[\Gamma]$ we have
\[
\lim_{V \in \mathcal{V}} \Vert T_{m}(  x h_V^{\frac{2}{q}} ) -    T_{m \vert_\Gamma}(x) h_V^{\frac{2}{q}} \Vert_{L_q(\widehat{G})} = 0 .
\]
\end{proposition}
\begin{proof}
For $2 \leq q < p < \infty$ the proposition is exactly Claim B of \cite{CPPR}. The case $1 \leq p < q \leq 2$ the same proof holds with trivial modifications with only one remark taken into account. Namely, at the place where \cite[Corollary 1.4]{CPPR} was used (which is stated only for $2 \leq q < \infty$) we need to use \cite[Lemma 3.1]{RicardIsrael} instead (which extends this corollary for $1 < q < 2$).  Note that  \cite[Lemma 3.1]{RicardIsrael} was only stated for finite von Neumann algebras, but its proof remains valid in the semi-finite setting exactly as it is stated.
\end{proof}

\subsection{Proof of Theorem A}
We are now ready to prove a localized  version of the De Leeuw restriction theorem generalizing \cite[Theorem A]{CPPR}. This theorem removes the $\SAIN$ condition from \cite[Theorem A]{CPPR} by taking into account the support of a multiplier.

\begin{definition}\label{Dfn=Cfunction} For $U \subseteq \Gamma$ we define
\[
	c(U):= \inf \{ \delta_{F}^{1/2} \mid  F \subseteq U \: \mathrm{finite} \}.
\]
In particular if $U$ is finite we have $c(U) = \delta_{U}^{1/2}$.
\end{definition}

\begin{theorem}[Theorem A]\label{Thm=LocalLeeuw}
  Let $m \in C_b( G)$. Then for every $1 \leq p < \infty$ we have that
\begin{equation}\label{Eqn=LocalRestriction}
c( \supp(m\vert_\Gamma)  )  \Vert T_{m \vert_\Gamma}: L_p(\widehat{\Gamma})   \rightarrow L_p(\widehat{\Gamma}) \Vert \leq
\Vert T_{m  }: L_p(\widehat{G}) \rightarrow L_p(\widehat{G}) \Vert.
\end{equation}
\end{theorem}
\begin{proof}
For $p=2$ the statement is obvious. By complex interpolation the logarithm of the norms on both sides of \eqref{Eqn=LocalRestriction} is a convex function in $p$ and hence certainly continuous. So the case $p=1$ follows from the case $1 < p < 2$ by approximation. The case $1< p < 2$ follows by duality from the case $2 < p < \infty$, see Remark \ref{Rmk=Dual}.

Now assume $2 < p < \infty$.
  Let $x \in \mathbb{C}[\Gamma]$. The frequency support $F$ of $T_{m \vert_\Gamma}(x)$ is finite and contained in $\supp(m \vert_\Gamma)$.
   Let $\epsilon > 0$.
 Let $\mathcal{V}$ be a   symmetric neighbourhood basis  shrinking to the identity of $G$ such that
 \[
 \limsup_{V \in \mathcal{V}} \delta_{F}(V)^{-\frac{1}{2}} < \delta_{F}^{-\frac{1}{2}} + \epsilon.
 \]
  By Proposition \ref{Prop=AlmostIsometry2} we find
 \[
 \begin{split}
& \Vert  T_{m \vert_\Gamma}( x) \Vert_{L_p(\widehat{\Gamma})} =   \lim_{q \nearrow p} \:  \Vert   T_{m \vert_\Gamma}( x)\Vert_{L_q(\widehat{\Gamma})} \\
  \leq &  \lim_{q \nearrow p} \: \limsup_{V \in \mathcal{V}}  \: \delta_{F}(V)^{-\frac{1}{2}}  \Vert \Phi_{q,V} ( T_{m \vert_\Gamma}( x)) \Vert_{L_q(\widehat{G})}  \\
  \leq &  \lim_{q \nearrow p} \: \limsup_{V \in \mathcal{V}}  \: (\delta_{F}^{-\frac{1}{2}} + \epsilon)   ( \Vert T_m ( \Phi_{q,V}( x ) )  \Vert_{L_q(\widehat{G})}  +  \Vert T_m ( \Phi_{q,V}( x ) ) - \Phi_{q,V} (  T_{m \vert_\Gamma}( x))  \Vert_{L_q(\widehat{G})})  \\
  \end{split}
  \]
  By Proposition \ref{Prop=Intertwiner}  the last summand goes to 0. Hence using again Proposition \ref{Prop=AlmostIsometry2},
  \[
  \begin{split}
 \Vert  T_{m \vert_\Gamma}( x) \Vert_{L_p(\widehat{\Gamma})}
  \leq & \lim_{q \nearrow p}  (\delta_{F}^{-\frac{1}{2}} + \epsilon) \Vert T_m: L_q(\widehat{G} ) \rightarrow L_q( \widehat{G}  )   \Vert  \limsup_{V \in \mathcal{V}} \Vert  \Phi_{q,V}( x )  \Vert_{L_q(\widehat{G})} \\
  \leq & \lim_{q \nearrow p}   (\delta_{F}^{-\frac{1}{2}} + \epsilon) \Vert T_m: L_q(\widehat{G}) \rightarrow L_q(\widehat{G}) \Vert \Vert  x  \Vert_{L_q(\widehat{\Gamma})} \\
  = &     (\delta_{F}^{-\frac{1}{2}} + \epsilon) \Vert T_m: L_p(\widehat{G} ) \rightarrow L_p( \widehat{G} ) \Vert \Vert  x  \Vert_{L_p(\widehat{\Gamma})}.
 \end{split}
 \]
 Letting $\epsilon \searrow 0$ concludes the proof as $\mathbb{C}[\Gamma]$ is dense in $L_p(\widehat{\Gamma})$.
\end{proof}

\begin{remark}\label{Rmk=SAINcase}
In Theorem  \ref{Thm=LocalLeeuw} if moreover $G \in \SAIN_\Gamma$ then $\delta_F = 1$ for any $F \subseteq \Gamma$ finite and we recover \cite[Theorem A]{CPPR}. In particular our theorem also generalizes the restriction theorem from \cite{DeLeeuw}.
\end{remark}

\begin{remark}\label{Rmk=CBCaseThmA}
Use the notation of Theorem \ref{Thm=LocalLeeuw}. Now let $N$ be any (semi-finite) von Neumann algebra and consider the tensor product $L_{p}(N) \otimes L_{p}(\widehat{G})$, which is equipped with the $L_p$ norm induced by the tensor product of the weights on $N$ and $\mathcal{L}(G)$. We write $L_{p}(\widehat{G}; L_p(N))$ for $L_{p}(N) \otimes L_{p}(\widehat{G})$. Consider the vector-valued extension $T_m^{N}$ of $T_m$ given by $L_{p}(\widehat{G}; L_{p}(N))  \rightarrow  L_{p}(\widehat{G}; L_{p}(N))$,
\[
  T_m^N( x \otimes  \lambda(f) )
   = \int_{G} m(s)  f(s) \: x \otimes \lambda(s) ds.
\]
Let $\mathcal{R}$ be the hyperfinite II$_1$-factor.
We will say that $T_m$ is completely bounded if $T_m^{\mathcal{R}}$ is bounded and write $\Vert T_m \Vert_{\mathrm{cb}} = \Vert T_m^{\mathcal{R}} \Vert$. We still have \eqref{Eqn=LocalRestriction} in the cb-norm since we may apply the theorem to $\Gamma \times S_\infty < G \times S_\infty$ and $\mathcal{L}(S_\infty) = \mathcal{R}$ where $S_\infty$ is the group of finite permutations of $\mathbb{N}$.
\end{remark}

\begin{example}\label{Example=SLNR}
Theorem A provides an ansatz for finding genuine examples of $p$-multipliers on $\SL(n, \mathbb{Z})$ for $1<p<\infty$.
Consider $G = \SL(n, \mathbb{R})$ with discrete subgroup $\Gamma = \SL(n, \mathbb{Z})$. Suppose we can find
a class of symbols $m \in C_b(G)$ supported on the adjoint ball $B_\rho^G := \{g\in G \mid \|\Ad_{g}\| \leq \rho \}$ with radius $\rho \rightarrow \infty$
that satisfy
\begin{equation} \label{Eqn=Discussion}
\Vert T_{m}: L_p(\widehat{G}) \rightarrow L_p(\widehat{G}) \Vert \leq   \rho^{-\frac{1}{4} n (n-1)} k(\rho) \Vert m \vert_\Gamma \Vert_1
\end{equation}
for some $k(\rho)> 0$.
As a particular case of Theorem \ref{thm:DeltaForReductive} proved below (see Remark \ref{Rmk:ConeInterpretation}), we obtain that
\[
\delta_{B_\rho^G}^{-\frac{1}{2}} \leq \rho^{\frac{1}{4} n (n-1)},
\]
so Theorem A would yield
 \[
 \Vert T_{m\vert_\Gamma}: L_p(\widehat{\Gamma}) \rightarrow L_p(\widehat{\Gamma}) \Vert \leq k(\rho)  \Vert m \vert_\Gamma \Vert_1.
 \]
 A minimal requirement for this to yield genuine multipliers on $\Gamma = \SL(n, \mathbb{Z})$ is that $k(\rho) < 1$ and it is natural to require then that $k(\rho) \rightarrow 0$ as $\rho \rightarrow \infty$. 
As a sanity check, we verify that the lower bound on $k(\rho)$ obtained from estimates of the cardinality of $B_\rho^G \cap \Gamma$ do not exclude
the possibility that $k(\rho) \rightarrow 0$.

Since $\Vert m \Vert_\infty  \leq \Vert T_{m}: L_p(\widehat{G}) \rightarrow L_p(\widehat{G}) \Vert$ and $\Vert m \vert_\Gamma \Vert_1 \leq  \#(B_\rho^G \cap \Gamma )  \Vert m \Vert_\infty$, a necessary condition for \eqref{Eqn=Discussion} to hold is that
\[
\frac{\rho^{\frac{1}{4}n(n-1)}}{\#(B_\rho^G \cap \Gamma)} \leq k(\rho).
\]
 In \cite[Corollary 1.1 and Example 4.2]{Maucourant} it is shown that
 \begin{equation}\label{Eqn=CountingPoints}
 \mu( B_\rho^G ) \approx \#(B_\rho^G \cap \Gamma )  \approx \rho^{ \lfloor \frac{1}{4} n^2  \rfloor} \log(\rho)^{ \lceil \frac{1}{2} n \rceil},
 \end{equation}
so that
 $k(\rho) \gtrsim \rho^{-\frac{1}{4}n} \log(\rho)^{- \lceil \frac{1}{2} n \rceil} \Big(\rho^{\frac{1}{4} n^2 - \lfloor \frac{1}{4} n^2  \rfloor}\Big)$.
This shows that genuine multipliers of this form are not a priori excluded by the counting estimates.

 In detail, let $\mathfrak{a}$ be a maximally noncompact Cartan subalgebra of the Lie algebra $\mathfrak{g}$ of $G$ and let $\lambda_1, \ldots, \lambda_n \in \mathfrak{a}^\ast$ be the natural weights described in  \cite[Example 4.1, 4.2]{Maucourant}. The weights occurring in the adjoint representation are $\lambda_i - \lambda_j, 1 \leq i,j \leq n$ and $\mathcal{C}$ is defined as the convex hull of these.
 For $\beta$ the sum of all positive roots (see \cite{Maucourant}; so not the usual half-sum) we find,
 \[
 \beta = \sum_{i,j=1, i<j}^n (\lambda_i - \lambda_j)
  = \sum_{i=1}^n (n+1 - 2i) \lambda_i
 =  \sum_{i=1}^{ \lfloor n/2 \rfloor}   (n+1-2i) (\lambda_i - \lambda_{n+1-i}).
 \]
  Now set $D := \sum_{i=1}^{ \lfloor n/2 \rfloor}   (n+1-2i)$ which equals $\frac{1}{4} n^2$ if $n$ is even and  $\frac{1}{4} (n^2 - 1)$ if $n$ is odd. Briefly, $D = \lfloor \frac{1}{4} n^2  \rfloor$.  If $n$ is even $\beta/D$ is on the face of $\mathcal{C}$ given by the convex hull of the weights $\lambda_i - \lambda_{n+1-i}, 1 \leq i \leq \frac{n}{2}$. If $n$ is odd then $\beta/D$ is on the face of $\mathcal{C}$ given by the convex hull of the weights $\lambda_i - \lambda_{n+1-i}, 1 \leq i \leq \frac{n-1}{2}$. Further, we have $E := \dim(\mathfrak{a}) - (\frac{n}{2} -1) = \frac{n}{2}$ if $n$ is even and  $E := \dim(\mathfrak{a}) - (\frac{n-1}{2} -1) = \frac{n+1}{2}$ if $n$ is odd. Briefly, $E =  \lceil \frac{1}{2} n \rceil$. So we have determined $D$ and $E$ as in \cite{Maucourant} and \cite[Corollary 1.1]{Maucourant} yields \eqref{Eqn=CountingPoints}. Note that $D$ and $E$ in this example are written as $d$ and $e$ in \cite{Maucourant} which differs from $d$ as in Theorem B.
\end{example}

\subsection{\texorpdfstring{Asymptotic isometries for $1 \leq p < 2$ under the $\SAIN$ condition}{Asymptotic isometries for p between 1 and 2 under the SAIN condition}}
In order to prove the multilinear theorems we prove two lemmas that in particular extend the results of Section \ref{Sect=Embeddings} under the $\SAIN$ condition. Recall that the group C$^\ast$-algebra $C_r^\ast(\Gamma)$ is defined as the norm closure of $\mathbb{C}[\Gamma]$ in $B(\ell_2(\Gamma))$.

 \begin{lemma}\label{Lem=RicardCorollary}
 Let $1 \leq p < \infty$.  Assume $G \in \SAIN_\Gamma$ and let $\mathcal{V}$ be as in Remark \ref{Rmk=Shrink}. Then for every $x \in  C_r^\ast(\Gamma)$,
 \begin{equation}\label{Lem=GoesToZero}
\begin{split}
 & \lim_{V \in \mathcal{V} } \Vert  x  h_{V}^{\frac{2}{p} }   -  h_{V}^{\frac{2}{p}} x \Vert_{L_p(\widehat{G})} = 0.
\end{split}
\end{equation}
 \end{lemma}
 \begin{proof}
By approximation and taking linear combinations it suffices to prove the lemma in case $x = \lambda(r)$ with $r \in \Gamma$.  Then  \cite[Corollary 1.4]{CPPR} shows that for $2 \leq p < \infty$,
 \begin{equation}\label{Eqn=InvarianceEstimate}
\begin{split}
 \Vert  \lambda(r)  u_V h_V^{\frac{2}{p} } \lambda(r)^\ast  -  u_V h_V^{\frac{2}{p}} \Vert_{L_p(\widehat{G})} \leq & C
 \Vert  \lambda(r) u_V h_V  \lambda(r)^\ast  -  u_V h_V  \Vert_{L_2(\widehat{G})}^{\frac{1}{2p}} \\
 = & C
 \left( \frac{  2 - 2  \mu(  \Ad_r( V )  \cap V   )}{ \mu(V) } \right)^{\frac{1}{4p}}.
 \end{split}
 \end{equation}
 Taking the limit over $V \in \mathcal{V}$ this converges to 0  by the $\SAIN$ condition, c.f. Remark \ref{Rmk=Shrink}. Now for any $1 \leq p < \infty$,
 \[
 \begin{split}
 \Vert  \lambda(r)   h_V^{\frac{2}{p} } \lambda(r)^\ast  -   h_V^{\frac{2}{p}} \Vert_{L_p(\widehat{G})} \leq &
  \Vert  \lambda(r)   h_V^{\frac{1}{p} } u_V u_V h_V^{\frac{1}{p} }  \lambda(r)^\ast  -    h_V^{\frac{1}{p} } u_V  \lambda(r)   u_V h_V^{\frac{1}{p} }  \lambda(r)^\ast  \Vert_{L_p(\widehat{G})} \\
  & +
    \Vert    h_V^{\frac{1}{p} } u_V  \lambda(r)   u_V h_V^{\frac{1}{p} }  \lambda(r)^\ast -     h_V^{\frac{1}{p} } u_V u_V  h_V^{\frac{1}{p} }  \Vert_{L_p(\widehat{G})} \\
    \leq & \Vert  \lambda(r)   h_V^{\frac{1}{p} } u_V    -    h_V^{\frac{1}{p} } u_V  \lambda(r) \Vert_{L_{2p}(\widehat{G})} \Vert  u_V h_V^{\frac{1}{p} }  \lambda(r)^\ast \Vert_{L_{2p}(\widehat{G})}    \\
  & +
    \Vert    h_V^{\frac{1}{p} } u_V \Vert_{L_{2p}(\widehat{G})}  \Vert \lambda(r)   u_V h_V^{\frac{1}{p} }  \lambda(r)^\ast -     u_V  h_V^{\frac{1}{p} }  \Vert_{L_{2p}(\widehat{G})}. \\
 \end{split}
 \]
 By Proposition \ref{Prop=AlmostIsometry1} and \eqref{Eqn=InvarianceEstimate}  this expression converges to 0.
\end{proof}

For completeness we note that the crucial inequality behind \eqref{Eqn=InvarianceEstimate} is Ando's inequality \cite[Lemma 2.2]{CPPR} from which it could also be proved directly using the methods in \cite[Section 2]{CPPR}.

\begin{proposition}\label{Prop=AlmostIsometrySymmetric}
 Let $1 \leq p < \infty$.  Assume $G \in \SAIN_\Gamma$ and let $\mathcal{V}$  be as in Remark \ref{Rmk=Shrink}. For every $x \in C_r^\ast( \Gamma )$ we have  $\lim_{V \in \mathcal{V}} \Vert h_{V}^{\frac{1}{p}}   x  h_{V}^{\frac{1}{p}}  \Vert_{L_p(\widehat{G})} = \Vert x \Vert_{L_p(\widehat{\Gamma})}$.
\end{proposition}
\begin{proof}
First, take $y \in \mathbb{C}[\Gamma]$ and assume $x = y^\ast y$. Then we have, by Lemma \ref{Lem=RicardCorollary}, Proposition \ref{Prop=AlmostIsometry2} as $G \in \SAIN_\Gamma$  that,
\[
\lim_{V \in \mathcal{V}} \Vert h_{V}^{\frac{1}{p}}   x  h_{V}^{\frac{1}{p}}  \Vert_{L_p(\widehat{G})}^p =
\lim_{V \in \mathcal{V}} \Vert   y  h_{V}^{\frac{1}{p}}  \Vert_{L_{2p}(\widehat{G})}^{2p} =   \Vert   y    \Vert_{L_{2p}(\widehat{\Gamma})}^{2p} = \Vert x \Vert_{L_p(\widehat{\Gamma})}^p.
\]
Now if   $x \in C_r^\ast(\Gamma)$ is positive and $\epsilon > 0$ then we may find $y \in \mathbb{C}[\Gamma]$ with $\Vert x - y^\ast y \Vert_{L_\infty(\widehat{\Gamma})} < \epsilon$. Then,
\[
\begin{split}
 \vert \Vert h_{V}^{\frac{1}{p}}   x  h_{V}^{\frac{1}{p}}  \Vert_{L_p(\widehat{G})} - \Vert x \Vert_{L_p(\widehat{\Gamma})} \vert \leq &
  \vert \Vert h_{V}^{\frac{1}{p}}   y^\ast y   h_{V}^{\frac{1}{p}}  \Vert_{L_p(\widehat{G})} - \Vert y^\ast y \Vert_{L_p(\widehat{\Gamma})} \vert + ( \Vert h_{V}^{\frac{1}{p}}  (x - y^\ast y)   h_{V}^{\frac{1}{p}}  \Vert_{L_p(\widehat{G})} +  \Vert x- y^\ast y \Vert_{L_p(\widehat{\Gamma})}) \\
  \leq &  \vert \Vert h_{V}^{\frac{1}{p}}   y^\ast y   h_{V}^{\frac{1}{p}}  \Vert_{L_p(\widehat{G})} - \Vert y^\ast y \Vert_{L_p(\widehat{\Gamma})}\vert + 2 \Vert x - y^\ast y \Vert_{L_\infty(\widehat{\Gamma})},
\end{split}
\]
which is smaller than $3 \epsilon$ for $V \in \mathcal{V}$ small.

Now assume that $x \in C_r^\ast(\Gamma)$ is self-adjoint and write $x = x_+ - x_-$ with $0 \leq x_\pm \in C_r^\ast(\Gamma)$ and orthogonal support.  Then $\Vert   x    \Vert_{L_p(\widehat{\Gamma})}^p =  \Vert   x_+    \Vert_{L_p(\widehat{\Gamma})}^p +  \Vert   x_-   \Vert_{L_p(\widehat{\Gamma})}^p$ and similarly, by Lemma \ref{Lem=RicardCorollary},
\[
\begin{split}
\lim_{V \in \mathcal{V}} \Vert h_V^{\frac{1}{p}}  x  h_V^{\frac{1}{p}}  \Vert_{L_p(\widehat{G})}^p =&
\lim_{V \in \mathcal{V}} \Vert x_+^{\frac{1}{2}}  h_V^{\frac{2}{p}} x_+^{\frac{1}{2}} -   x_-^{\frac{1}{2}}  h_V^{\frac{2}{p}} x_-^{\frac{1}{2}} \Vert_{L_p(\widehat{G})}^p =
\lim_{V \in \mathcal{V}} \Vert x_+^{\frac{1}{2}}  h_V^{\frac{2}{p}} x_+^{\frac{1}{2}} \Vert_{L_p(\widehat{G})}^p +  \Vert  x_-^{\frac{1}{2}}  h_V^{\frac{2}{p}} x_-^{\frac{1}{2}} \Vert_{L_p(\widehat{G})}^p\\
 = &
\lim_{V \in \mathcal{V}} \Vert h_V^{\frac{1}{p}}  x_+  h_V^{\frac{1}{p}}  \Vert_{L_p(\widehat{G})}^p + \Vert h_V^{\frac{1}{p}}  x_-  h_V^{\frac{1}{p}}  \Vert_{L_p(\widehat{G})}^p.
\end{split}
\]
So  we may conclude the proof from the self-adjoint case.

 Now for general $x \in C_r^\ast(\Gamma)$  we may replace $x$ and $h_V$ by respectively,
\[
\left(
\begin{array}{cc}
  0      &   x \\
  x^\ast &   0
\end{array}
\right), \qquad \textrm{ and } \qquad  \left(
\begin{array}{cc}
  h_V      &   0 \\
  0 &   h_V
\end{array}
\right)
\]
in $L_p(M_2(\mathbb{C}  ) \otimes L_\infty(\widehat{\Gamma}) )$ and $L_p(M_2(\mathbb{C})\otimes L_\infty(\widehat{G}))$ respectively and conclude the proof; it follows that the first part of the proof also holds in this $2 \times 2$-matrix amplification. One way to see this is using the fact that the group von Neumann algebra of the infinite permutation group $S_\infty$ is the hyperfinite $\text{II}_1$ factor, $ \mathcal{R}$.  We now replace the inclusion $\Gamma < G$ by $S_\infty \times \Gamma < S_\infty \times G$. The corresponding von Neumann algebras are then $\mathcal{R} \otimes \mathcal{L}(\Gamma)$ and $\mathcal{R} \otimes \mathcal{L}(G)$. Since the matrix algebras embed into $\mathcal{R}$, we may view the matrix amplification of $x$ as an element of $L_p(\mathcal{R}\otimes L_\infty(\widehat{\Gamma}))$.
\end{proof}

\section{Theorem C: The multilinear De Leeuw restriction theorem}\label{Sect=MultilinearRestriction}

In this section we define multilinear Fourier multipliers and prove the multilinear De Leeuw restriction theorem.

\subsection{Multilinear maps} Let $X_1, \ldots, X_n$ and $Y$ be Banach spaces. Let $T: X_1 \times \ldots \times X_n \rightarrow Y$ be a multilinear map, meaning that it is linear in each of its variables. The norm of $T$ is then defined as
\[
\Vert T \Vert = \sup_{x_i \in X_i, x_i \not = 0} \frac{ \Vert T(x_1, \ldots, x_n) \Vert_Y }{ \Vert x_1 \Vert_{X_1} \ldots \Vert x_n \Vert_{X_n} }.
\]

\subsection{Multilinear Fourier multipliers}
Let $G$ be a locally compact unimodular group. Let $m: G^{\times n} \rightarrow \mathbb{C}$ be a  bounded measurable function. Consider the tuple $(p_1, \ldots, p_n)$ with $1 \leq p_i < \infty$ and let also $1 \leq p < \infty$.     We say that $m$ is a $(p_1, \ldots, p_n) \rightarrow p$ {\it multiplier} if there exists a bounded $n$-linear map
\[
T_m: L_{p_1}(\widehat{G}) \times \ldots \times L_{p_n}(\widehat{G}) \rightarrow L_{p}(\widehat{G}),
 \]
 that is determined by
 \[
    T_m( \lambda(f_1), \ldots, \lambda(f_n) )  = \int_{G^{\times n}} m(s_1, \ldots, s_n)  f_1(s_1) \ldots f_n(s_n) \lambda(s_1 \ldots s_n) ds_1 \ldots ds_n,
 \]
 for all $f_1, \ldots, f_n \in C_c(G)^{\ast 2}$. By mild abuse of terminology we also refer to the map $T_m$ as the $(p_1, \ldots, p_n) \rightarrow p$  multiplier and $m$ is then explicitly referred to as the symbol of $T_m$. If $p$ is the conjugate determined by $p^{-1} = \sum_{j=1}^n p_j^{-1}$ then we simply speak about a $(p_1, \ldots ,p_n)$-multiplier in analogy with the linear case $n=1$. $n$ is also called the {\it order} of the multiplier.

\subsection{Reduction formulae}
We prove a number of lemmas that reduce the analysis of $n$-linear Fourier multipliers to Fourier multipliers of lower order.

 \begin{definition}
   For $1 \leq k \leq n$ we say that a tuple $(q_1, \ldots, q_k), 1 \leq q_j \leq \infty$ is a {\it consummation} of $(p_1, \ldots, p_n), 1 \leq p_j \leq \infty$ at indices  $1 = i_1  < \ldots < i_k \leq n$ if  $q_l^{-1} =   \sum_{j = i_l}^{i_{l+1} - 1}  p_{j}^{-1}$.
 \end{definition}

\begin{lemma}\label{Lem=Consumption}
    For $1 \leq k \leq n$ let  $(q_1, \ldots, q_k), 1 \leq q_j < \infty$  be a consummation of $(p_1, \ldots, p_n), 1 \leq p_j < \infty$ at indices $1 = i_1  < \ldots < i_k \leq n$.
  Suppose that $m:   G^{\times  k}   \rightarrow \mathbb{C}$ is bounded and measurable and set
    \[
  \widetilde{m}(s_1, \ldots, s_n) =  m(  s_1 \ldots s_{i_2-1}, s_{i_2} \ldots s_{i_3-1}, \ldots, s_{i_{k-1}} \ldots s_{i_k-1},  s_{i_k} \ldots s_n ).
  \]
   Then $m$ is a $(q_1, \ldots, q_{k})$-multiplier iff $\widetilde{m}$ is a $(p_1, \ldots, p_{n})$-multiplier and moreover, for $x_i \in L_{p_i}(\widehat{G})$ we have
  \[
  T_{\widetilde{m}}( x_1, \ldots, x_n  ) = T_m( x_1 \ldots x_{i_2-1}, x_{i_2} \ldots x_{i_{3}-1 }, \ldots, x_{i_{k-1}} \ldots x_{i_k-1},  x_{i_k} \ldots x_n ).
  \]
\end{lemma}
\begin{proof}
Let $f_1, \ldots f_n \in C_c(G)^{\ast 2}$.
 Then, using the invariance of the Haar integral for the second equality and \eqref{Eqn=ConvolutionLambda},
 \[
 \begin{split}
   &  T_{\widetilde{m}}(\lambda(f_1), \ldots, \lambda(f_n) ) \\
 = & \int_{G^{\times n}}  \widetilde{m}(s_1, \ldots, s_n) f_1(s_1) \ldots f_n(s_n) \lambda(s_1 \ldots s_n) ds_1 \ldots ds_n \\
 = & \int_{G^{\times k}} m(t_{1}, t_{2}, \ldots, t_{k} )
   \left(  \prod_{j=1}^k  (f_{i_j} \ast \ldots \ast  f_{i_{j+1}-1})( t_{j} ) \right)   \lambda(t_{1} \ldots t_{k}) dt_{1} \ldots t_{k} \\
    = & T_m(\lambda(f_1 \ast \ldots \ast f_{i_{2} -1}),  \ldots, \lambda(f_{i_{k}} \ast   \ldots \ast f_n) )
 =
 T_m(\lambda(f_1)  \ldots \lambda( f_{i_{2} -1}),  \ldots, \lambda(f_{i_{k}})   \ldots \lambda( f_n) ).
 \end{split}
 \]
 So if $T_m$ is bounded it follows by this expression and the H\"older inequality that also $T_{\widetilde{m}}$ is bounded.
  Conversely, assume $T_{\widetilde{m}}$ is bounded.  The products $\lambda(f_{i_j}) \ldots \lambda(f_{i_{j+1} - 1})$ with $\lambda(f_i)$ in the unit ball of $L_{p_j}(\widehat{G})$ and $f_i \in C_c(G)^{\ast 2}$ lie densely in the unit ball of $L_{q_{j}}(\widehat{G})$ by \eqref{Eqn=MaxHolder}. Since
 \[
\Vert  T_m(\lambda(f_1)  \ldots \lambda( f_{i_{2} -1}),  \ldots, \lambda(f_{i_{k}})   \ldots \lambda( f_n) ) \Vert_{L_p(\widehat{G})} =
\Vert  T_{\widetilde{m}}(\lambda(f_1), \ldots, \lambda(f_n) )  \Vert_{L_p(\widehat{G})}
\leq \Vert T_{\widetilde{m}} \Vert,
 \]
 and taking the supremum over all  such $f_i$ we conclude the proof.
\end{proof}

\begin{lemma}\label{Lemma=TranslationInvariance}
  Let $1 \leq p_i, p \leq \infty$.  Suppose that $m:   G^{\times n}   \rightarrow \mathbb{C}$ is bounded and measurable and set for $r,t,r' \in G$,
    \[
  \widetilde{m}_i( s_1, \ldots, s_n; r, t, r')     := m( r s_1, \ldots, s_i t, t^{-1}s_{i+1}, \ldots,  s_n r').
  \]
   Then $m$ is a $(p_1, \ldots, p_{k}) \rightarrow p$-multiplier if and only if so is $\widetilde{m}_i( \: \cdot \: ; r,t,r')$.  In that case for $x_j \in L_{p_j}(\widehat{G})$,
     \begin{equation}\label{Eqn=ConvolveRelation}
  T_{\widetilde{m}_i(\: \cdot \: ; r,t,r')}( x_1, \ldots,   x_n  ) = \lambda(r)^\ast T_{m}( \lambda(r) x_1, \ldots, x_i \lambda(t), \lambda(t)^\ast x_{i+1}, \ldots,  x_n \lambda(r')  )  \lambda(r')^\ast.
  \end{equation}
  Further, as maps $L_{p_1}(\widehat{G}) \times \ldots \times L_{p_n}(\widehat{G}) \rightarrow L_p(\widehat{G})$ we have $\Vert T_m \Vert = \Vert T_{\widetilde{m}_i( \: \cdot \: ; r,t,r') } \Vert$ and $(r,t,r') \mapsto T_{\widetilde{m}_i( \: \cdot \: ; r,t,r')}$  is strongly continuous.
\end{lemma}
\begin{proof}
The proof of \eqref{Eqn=ConvolveRelation} is similar to Lemma \ref{Lem=Consumption}, namely one verifies the equality for   $x_j = \lambda(f_j), f_j \in C_c(G)^{\ast 2}$ by an elementary computation and then concludes by continuity. Since multiplication with $\lambda(s), s \in G$ determines an isometry on $L_{p_j}(\widehat{G})$ that is strongly continuous in $s$ by  \cite[Lemma 2.3]{JungeSherman} the final statements follow as well.
 \end{proof}

There are several formulae available for nested Fourier multipliers of which the following lemma is a particular case.

\begin{lemma}\label{Lem=Nest}
  Let $1 \leq p_1, \ldots, p_n < \infty$. Let $q_n = p_n$ and for $1 \leq j \leq n-1$ set  $q_j^{-1} = \sum_{i=j}^{n-1} p_i^{-1}$.
  Suppose that $m_j:   G   \rightarrow \mathbb{C}$ is  a $q_j$-multiplier for all $1 \leq j \leq n$.   Set
    \[
    \begin{split}
&  \widetilde{m}(s_1, \ldots, s_n)
=   m_1( s_1 \ldots s_{n-1} ) m_2(s_2 \ldots s_{n-1}) \ldots m_{n-1}( s_{n-1}) m_n(s_n).
  \end{split}
  \]
   Then  $\widetilde{m}$ is a $(p_1, \ldots, p_n)$-multiplier and   for $x_i \in L_{p_i}(\widehat{G})$ we have
  \begin{equation}\label{Eqn=Nested}
  \begin{split}
   &T_{\widetilde{m}}( x_1, \ldots, x_n  )
  =   T_{m_1}( x_1 T_{m_2}(  x_2 \ldots T_{m_{n-1}}(x_{n-1}) \ldots  )    ) T_{m_n}(x_n).
  \end{split}
  \end{equation}
\end{lemma}
\begin{proof}
Again the identity \eqref{Eqn=Nested} can directly be verified for $x_i = \lambda(f_i), f_i \in  C_c(G)^{\ast 2}$. The boundedness of $T_{\widetilde{m}}$ then follows as all $T_{m_j}$ are bounded together with the H\"older inequality.
\end{proof}

 \subsection{Proof of Theorem C}  The idea behind the following proof is that we approximate a multiplier by multipliers that decompose as a nested iteration of linear multipliers. Then using the analysis of linear multipliers  we may prove the desired result. We first give the core of the proof of Theorem \ref{Thm=MultiDeLeeuwRestriction} and prove the intertwining property that is needed for it afterwards in Lemma \ref{Lem=AuxConverge}.

\begin{theorem}[Theorem C]\label{Thm=MultiDeLeeuwRestriction}
 Let $G$ be a locally compact unimodular group. Let $\Gamma < G$ be a discrete subgroup such that $G \in \SAIN_\Gamma$.  Let $m \in C_b( G^{\times n})$. Then for every $1 \leq p, p_1, \ldots, p_n < \infty$ with $p^{-1} = \sum_{j=1}^n p_j^{-1}$ we have that
\begin{equation}\label{Eqn=DeLeeuwMultilinear}
\Vert T_{m \vert_{\Gamma^{\times n}} }: L_{p_1}(\widehat{\Gamma}) \times \ldots \times L_{p_n}(\widehat{\Gamma})   \rightarrow L_p(\widehat{\Gamma}) \Vert \leq
\Vert T_{m  }: L_{p_1}(\widehat{G}) \times \ldots \times L_{p_n}(\widehat{G})  \rightarrow L_p(\widehat{G}) \Vert.
\end{equation}
\end{theorem}
\begin{proof}
In case $n =1$ the theorem was proved in \cite[Theorem A]{CPPR} or Remark \ref{Rmk=SAINcase}. We therefore assume that $n \geq 2$. Consequently $1 < p_1, \ldots, p_n < \infty$. We note however that our proof remains valid for $n =1$ provided that $1 < p = p_1 < \infty$.

Let $(\varphi_k)_k$ be a net of  positive definite functions in the Fourier algebra $A(G) = L_2(G) \ast L_2(G)$ with $\Vert \varphi_k \Vert_1 = 1$ and compact support shrinking to the identity of $G$. Recall that elements of $A(G)$ are continuous and $p$-multipliers for all $1 \leq p \leq \infty$ (see \cite{DeCanniereHaagerup}).   For any function $M \in C_b( G^{\times n})$ we set for $t_i, s_i \in G$,
\begin{equation}\label{Eqn=SubT}
M_{t_1, \ldots, t_n}(s_1, \ldots, s_n) :=  M(t_1^{-1} s_1 t_2, t_2^{-1} s_2 t_3, \ldots  ,  t_{n-2}^{-1} s_{n-2} t_{n-1}   ,  t_{n-1}^{-1} s_{n-1},   s_n  t_{n}^{-1} ),
\end{equation}
and
\begin{equation}\label{Eqn=Convolution}
\begin{split}
M_k(s_1, \ldots, s_n)
 := & \int_{G^{\times n}}  M_{t_1, \ldots, t_n}(s_1, \ldots s_n)  \left( \prod_{j=1}^{n} \varphi_k(t_j) \right)   dt_1 \ldots t_n \\
    = & \int_{G^{\times n}}  M(t_1^{-1}  t_2, t_2^{-1}   t_3, \ldots,  t_{n-2}^{-1}t_{n-1},       t_{n-1}^{-1},  t_n^{-1} )
    \left( \prod_{j=1}^{n-1} \varphi_k(s_j \ldots s_{n-1} t_j) \right) \varphi_k(t_n s_n)  dt_1 \ldots t_n.
\end{split}
\end{equation}
Let $x_1, \ldots, x_n \in \mathbb{C}[\Gamma]$. Since $m_k \rightarrow m$ pointwise and all $x_j$ have finite frequency support we find that
\[
\begin{split}
 \Vert T_{m\vert_{\Gamma^{\times n}}}(x_1, \ldots, x_n) \Vert_{L_p(\widehat{\Gamma})} =  \lim_k \Vert  T_{m_k\vert_{\Gamma^{\times n}}}(x_1, \ldots, x_n) \Vert_{L_p(\widehat{\Gamma})}.
\end{split}
\]
 Let $\mathcal{V}$ be a symmetric neighbourhood basis of the identity of $G$ as in Remark \ref{Rmk=Shrink}. By Proposition \ref{Prop=AlmostIsometrySymmetric},
\[
\begin{split}
 & \Vert T_{m\vert_{\Gamma^{\times n}}}(x_1, \ldots, x_n) \Vert_{L_p(\widehat{\Gamma})}\\
  =  &    \lim_k \lim_{V \in \mathcal{V} } \Vert h_V^{\frac{1}{p}}  T_{m_k\vert_{\Gamma^{\times n}}}(x_1, \ldots, x_n)   h_V^{\frac{1}{p}} \Vert_{L_p(\widehat{G})}\\
 \leq &  \limsup_k \limsup_{V \in \mathcal{V} }  \Vert    T_{m_k}(  h_V^{\frac{1}{p_1}}   x_1  h_V^{\frac{1}{p_1}}, \ldots,  h_V^{\frac{1}{p_n}}   x_n   h_V^{\frac{1}{p_n}})   \Vert_{L_p(\widehat{G})} \\
   & + \limsup_k \limsup_{V \in \mathcal{V} }  \Vert h_V^{\frac{1}{p}}   T_{m_k\vert_{\Gamma^{\times n}}}(x_1, \ldots, x_n)  h_V^{\frac{1}{p}} - T_{m_k }(h_V^{\frac{1}{p_1}}   x_1   h_V^{\frac{1}{p_1}}, \ldots, h_V^{\frac{1}{p_n}} x_n h_V^{\frac{1}{p_n}})   \Vert_{L_p(\widehat{G})}.
\end{split}
\]
Now by Lemma \ref{Lem=AuxConverge} -- which we prove below -- the limit over $k$ and $V$ in the second summand exists and is 0. Therefore,
\[
\begin{split}
& \Vert T_{m\vert_{\Gamma^{\times n}}}(x_1, \ldots, x_n) \Vert_{L_p(\widehat{\Gamma})} \\
\leq  & \limsup_k \limsup_{V \in \mathcal{V} }  \Vert T_{m_k}( h_V^{\frac{1}{p_1}}  x_1   h_V^{\frac{1}{p_1}}, \ldots,  h_V^{\frac{1}{p_n}} x_n  h_V^{\frac{1}{p_n}})   \Vert_{L_p(\widehat{G})}   \\
 \leq &  \limsup_k \Vert T_{m_k}:  L_{p_1}(\widehat{G}) \times \ldots \times L_{p_n}(\widehat{G})  \rightarrow L_p(\widehat{G}) \Vert \cdot \limsup_V  \left( \prod_{j=1}^n \Vert h_V^{\frac{1}{p_j}}   x_j  h_V^{\frac{1}{p_j}} \Vert_{ L_{p_j}(\widehat{\Gamma})  } \right).
\end{split}
\]
By Lemma \ref{Lemma=TranslationInvariance}  we find that  $\Vert T_{m_k} \Vert \leq \Vert T_m \Vert$.
Hence, together with \eqref{Eqn=LocalLeeuwUp} of Proposition \ref{Prop=AlmostIsometry1}, we conclude that
\[
 \Vert T_{m\vert_{\Gamma^{\times n}}}(x_1, \ldots, x_n) \Vert_{ L_{p}(\widehat{\Gamma})  }  \leq \Vert T_{m}:  L_{p_1}(\widehat{G}) \times \ldots \times L_{p_n}(\widehat{G})  \rightarrow L_p(\widehat{G}) \Vert
\: \cdot \: \prod_{j=1}^n \Vert x_j  \Vert_{ L_{p_j}(\widehat{\Gamma})  }.
\]
Since the elements $x_j \in \mathbb{C}[\Gamma]$ are dense in $L_p(\widehat{\Gamma})$ we conclude the proof.
\end{proof}

\begin{lemma}\label{Lem=AuxConverge}
In the proof of Theorem \ref{Thm=MultiDeLeeuwRestriction}, in particular with $n \geq 2$,  we have that
\[
\lim_k   \limsup_{V \in \mathcal{V} }   \Vert  h_V^{\frac{1}{p}}   T_{m_k\vert_{\Gamma^{\times n}}}(x_1, \ldots, x_n)   h_V^{\frac{1}{p}} - T_{m_k }(h_V^{\frac{1}{p_1}}  x_1   h_V^{\frac{1}{p_1}}, \ldots,  h_V^{\frac{1}{p_n}}   x_n  h_V^{\frac{1}{p_n}})   \Vert_{ L_{p}(\widehat{G})  } = 0.
\]
  If $n = 1$ the statement also holds in case $1 < p = p_1 < \infty$.
\end{lemma}
\begin{proof}
Recall that all $x_j$ have finite frequency support and therefore by the triangle inequality it suffices to consider the case where $x_j = \lambda(r_j), r_j \in \Gamma$.
Let $\zeta: G \rightarrow  \mathbb{R}_{\geq 0}$ be a compactly supported continuous, positive definite function in the Fourier algebra $A(G)$ with $\zeta(e) = 1$,  so that $T_\zeta$ is a contraction for all $1\leq p\leq \infty$. For $1 \leq j \leq n, s \in G$ let
\begin{equation}\label{Eqn=ZetaFunctions}
\zeta_j(s) = \zeta( r_j^{-1} s).
\end{equation}
Set the product of functions
\[
(   m  (\zeta_1, \ldots, \zeta_n)) (s_1, \ldots, s_n) = m(s_1, \ldots, s_n) \zeta_1(s_1)  \ldots \zeta_n(s_n), \qquad s_j \in G,
\]
and then let $(m  (\zeta_1, \ldots, \zeta_n))_k$ be defined by the formula \eqref{Eqn=Convolution}. Similarly consider  $(m - m   (\zeta_1, \ldots, \zeta_n))_k = m_k - (m (\zeta_1, \ldots, \zeta_n))_k$ whose value at the point $(r_1, \ldots, r_n)$ converges to 0 in $k$.

We estimate
\begin{equation}\label{Eqn=ThreeSummands}
\begin{split}
     &  \Vert   h_V^{\frac{1}{p}}  T_{ m_k\vert_{\Gamma^{\times n}}}(x_1, \ldots, x_n)   h_V^{\frac{1}{p}} - T_{m_k }( h_V^{\frac{1}{p_1}}  x_1   h_V^{\frac{1}{p_1}}, \ldots,  h_V^{\frac{1}{p_n}}  x_n   h_V^{\frac{1}{p_n}})   \Vert_{L_p(\widehat{G})}\\
\leq & \Vert  h_V^{\frac{1}{p}}  T_{m_k\vert_{\Gamma^{\times n}}}(x_1, \ldots, x_n)   h_V^{\frac{1}{p}} - h_V^{\frac{1}{p}}   T_{ (m   (\zeta_1, \ldots, \zeta_n))_k \vert_{\Gamma^{\times n}}}(x_1, \ldots, x_n)   h_V^{\frac{1}{p}}   \Vert_{L_p(\widehat{G})} \\
& +\Vert  h_V^{\frac{1}{p}}   T_{ (m  (\zeta_1, \ldots, \zeta_n))_k \vert_{\Gamma^{\times n}}}(x_1, \ldots, x_n)  h_V^{\frac{1}{p}} - T_{ (m  (\zeta_1, \ldots, \zeta_n))_k  }( h_V^{\frac{1}{p_1}}    x_1   h_V^{\frac{1}{p_1}}, \ldots, h_V^{\frac{1}{p_n}}   x_n  h_V^{\frac{1}{p_n}})  \Vert_{L_p(\widehat{G})} \\
& + \Vert T_{ (m  (\zeta_1, \ldots, \zeta_n))_k  }( h_V^{\frac{1}{p_1}}   x_1  h_V^{\frac{1}{p_1}}, \ldots,  h_V^{\frac{1}{p_n}}   x_n  h_V^{\frac{1}{p_n}}) - T_{m_k }( h_V^{\frac{1}{p_1}}   x_1   h_V^{\frac{1}{p_1}}, \ldots,  h_V^{\frac{1}{p_n}}   x_n   h_V^{\frac{1}{p_n}})  \Vert_{L_p(\widehat{G})} \\
=: & A_{k,V} + B_{k,V} + C_{k,V}.
\end{split}
\end{equation}
where in the last line we defined the three summands. We show that each of these summands converges to 0 separately. By Proposition \ref{Prop=AlmostIsometry1} we have for any $V$ that
\[
\begin{split}
A_{k,V} \leq & \Vert  T_{m_k\vert_{\Gamma^{\times n}}}(x_1, \ldots, x_n)   - T_{ (m  (\zeta_1, \ldots, \zeta_n))_k \vert_{\Gamma^{\times n}}}(x_1, \ldots, x_n)     \Vert_{L_p(\widehat{\Gamma})} \\
= &\Vert  T_{   ( m_k  -    (m  (\zeta_1, \ldots, \zeta_n))_k )\vert_{\Gamma^{\times n}}   }( \lambda(r_1), \ldots, \lambda(r_n))   \Vert_{L_p(\widehat{\Gamma})} \\
= & \vert (m_k  -   (m   (\zeta_1, \ldots, \zeta_n))_k)(r_1, \ldots, r_n) \vert.
\end{split}
\]
As observed in the first paragraph the latter expression converges to 0 in $k$,
so we find that $\lim_k \lim_{V \in \mathcal{V} }  A_{k,V} = 0$.

Next we treat the summand $C_{k,V}$. Define, using the notation \eqref{Eqn=SubT},
\[
   C_{k,V}( t_1, \ldots, t_n) := \Vert T_{ ( m - m  (\zeta_1, \ldots, \zeta_n))_{t_1, \ldots, t_n}  }(h_V^{\frac{1}{p_1}}  x_1  h_V^{\frac{1}{p_1}}, \ldots,  h_V^{\frac{1}{p_n}}   x_n   h_V^{\frac{1}{p_n}}) \Vert_{L_p(\widehat{G})},
\]
and for $j = 1, \ldots, n-2$,
\[
y_j :=  \lambda(t_j)^\ast h_V^{\frac{1}{p_{j}}}    x_{j}  h_V^{\frac{1}{p_{j}}}  \lambda(t_{j+1}),  \qquad y_{n-1} :=  \lambda(t_{n-1})^\ast h_V^{\frac{1}{p_{n-1}}}    x_{n-1}  h_V^{\frac{1}{p_{n-1}}}, \quad  y_n :=   h_V^{\frac{1}{p_{n}}}    x_{n}  h_V^{\frac{1}{p_{n}}}  \lambda(t_{n})^\ast.
\]
By Lemma \ref{Lemma=TranslationInvariance},
\[
\begin{split}
  C_{k,V}( t_1, \ldots, t_n) =  &  \Vert T_{  m - m  (\zeta_1, \ldots, \zeta_n)  }(  y_1, \ldots,   y_n ) \Vert_{L_p(\widehat{G})} \\
\leq & \sum_{j=1}^n     \Vert T_{    m (\id \times \ldots \times \id \times \zeta_{j} - \id \times \zeta_{j+1} \times \ldots \times \zeta_n)   }( y_1, \ldots, y_n  ) \Vert_{L_p(\widehat{G})}.
\end{split}
\]
Since by similar arguments to the ones used in Lemmas \ref{Lem=Consumption} and \ref{Lem=Nest},
\[
T_{    m  (\id \times \ldots \times \id \times (\zeta_{j} - \id) \times \zeta_{j+1} \times \ldots \times \zeta_n)} = T_{m} \circ ( \id, \ldots, \id, T_{\zeta_j} - \id, T_{\zeta_{j+1}}, \ldots ,T_{\zeta_n}  ).
\]
we find
\begin{equation}\label{Eqn=CEstimate}
\begin{split}
  C_{k,V}( t_1, \ldots, t_n)  \leq & \Vert T_m: L_{p_1}(\widehat{G}) \times \ldots \times L_{p_n}(\widehat{G}) \rightarrow L_p(\widehat{G}) \Vert \\
& \quad \times \quad  \sum_{j=1}^{n}  \left( \Vert (T_{\zeta_j }- \id)( y_j )  \Vert_{L_{p_j}(\widehat{G})}   \prod_{i=1, i \not = j}^{n} \Vert  y_j  \Vert_{L_{p_i}(\widehat{G})}\right).
\end{split}
\end{equation}
We have
\[
\Vert  y_i \Vert_{L_{ p_i}(\widehat{G})} =  \Vert  h_V^{\frac{1}{p_i}} x_i  h_V^{\frac{1}{p_i}} \Vert_{L_{ p_i}(\widehat{G})} \leq \Vert  h_V^{\frac{1}{p_i}}  \Vert_{L_{2p_i}(\widehat{G})}^2  \Vert x_i\Vert_{L_\infty(\widehat{\Gamma})} \leq 1,
\]
 see also Proposition  \ref{Prop=AlmostIsometry1}. Recall \eqref{Eqn=ZetaFunctions}  and the fact that $\zeta(e) = 1$ and  $x_j = \lambda(r_j)$. By Lemma \ref{Lem=RicardCorollary} and  using Proposition \ref{Prop=Intertwiner} we see that for every  $1 \leq j \leq n-2$,
 \[
 \begin{split}
 \lim_{V \in \mathcal{V} }  \Vert (T_{\zeta_j }- \id)( y_j ) \Vert_{L_{p_j}(\widehat{G})} = &
\lim_{V \in \mathcal{V} }  \Vert (T_{\zeta_j }- \id)(  \lambda(t_j)^\ast h_V^{\frac{1}{p_{j}}}  \lambda(r_{j})   h_V^{\frac{1}{p_{j}}}  \lambda(t_{j+1}) ) \Vert_{L_{p_j}(\widehat{G})} \\
= &\lim_{V \in \mathcal{V} }  \Vert (T_{\zeta_j }- \id)(  \lambda(t_j)^\ast     h_V^{\frac{2}{p_{j}}}  \lambda(r_j t_{j+1}) ) \Vert_{L_{p_j}(\widehat{G})} \\
= &
\lim_{V \in \mathcal{V} }  \Vert (T_{\zeta(  r_j^{-1 } t_j^{-1} \:  \cdot \:  r_j t_{j+1}  ) }- \id)(     h_V^{\frac{2}{p_{j}}}  ) \Vert_{L_{p_j}(\widehat{G})} \\
 = & \vert  \zeta(  r_j^{-1} t_j^{-1}    r_j t_{j+1}  )  - 1 \vert,
 \end{split}
 \]
and similarly,
\[
\begin{split}
 \lim_{V \in \mathcal{V}} \Vert (T_{\zeta_{n-1}} - \id)(y_{n-1}) \Vert_{L_{p_{n-1}}(\widehat{G})} = & \vert   \zeta(  r_{n-1}^{-1} t_{n-1}^{-1}    r_{n-1}    ) - 1 \vert, \\
  \lim_{V \in \mathcal{V}} \Vert (T_{\zeta_n} - \id)(y_n) \Vert_{L_{p_n}(\widehat{G})} = & \vert \zeta (   t_n^{-1}  ) - 1 \vert.
\end{split}
\]
 Hence from \eqref{Eqn=CEstimate},
 \begin{equation} \label{Eqn=CEstimateNew}
 \begin{split}
& \limsup_{V \in \mathcal{V}}   C_{k,V}( t_1, \ldots, t_n)\\
  \leq &
  \Vert T_m: L_{p_1}(\widehat{G}) \times \ldots \times L_{p_n}(\widehat{G}) \rightarrow L_p(\widehat{G}) \Vert \\
   & \quad \times \quad \left( \sum_{j=1}^{n-2} \vert \zeta(  r_j^{-1 } t_j^{-1}    r_j t_{j+1}  )  - 1 \vert +  \vert   \zeta(  r_{n-1}^{-1} t_{n-1}^{-1}    r_{n-1}    ) - 1 \vert  + \vert \zeta(  t_n^{-1}) -1 \vert \right).
\end{split}
 \end{equation}
Now going back to the original quantity we have to estimate, we find by Lemma \ref{Lemma=TranslationInvariance},
\[
\begin{split}
   C_{k,V}  = & \Vert  T_{ (m - m (\zeta_1, \ldots, \zeta_n))_k  }(h_V^{\frac{1}{p_1}}   x_1   h_V^{\frac{1}{p_1}}, \ldots, h_V^{\frac{1}{p_n}}   x_n   h_V^{\frac{1}{p_n}}) \Vert_{L_p(\widehat{G})}  \\
\leq &     \int_{G^{\times n}}  \!\!\!\!\!\!\!\!  C_{k,V}( t_1, \ldots, t_n)  \left( \prod_{j=1}^{n} \varphi_k(t_j) \right)  dt_1 \ldots t_n.
\end{split}
\]
Taking limits in \eqref{Eqn=CEstimateNew} yields,
\[
\begin{split}
& \limsup_{V \in \mathcal{V}}    C_{k,V} \leq    \Vert T_m: L_{p_1}(\widehat{G}) \times \ldots  \times L_{p_n}(\widehat{G}) \rightarrow L_p(\widehat{G})  \Vert \\
&  \:\:\times \:\: \int_{G^{\times n}}  \left( \sum_{j=1}^{n-2} \vert \zeta(  r_j^{-1} t_j^{-1}    r_j t_{j+1}  )  - 1 \vert +   \vert   \zeta(  r_{n-1}^{-1} t_{n-1}^{-1}    r_{n-1}    ) - 1 \vert  + \vert \zeta( t_n^{-1} ) -1 \vert \right)
 \left( \prod_{j=1}^{n} \varphi_k(t_j) \right)  dt_1 \ldots t_n.
\end{split}
\]
This expression goes to 0 when taking the limit in $k$. This concludes the estimate for $C_{k,V}$.

We now prove that for every $k$ we have  $\lim_V B_{k,V} = 0$. For what follows, we will assume, by scaling $\varphi_k$ if necessary, that $T_{\varphi_k}$ is a contraction on $L_p(\widehat{G})$; we may do this since at this point we need to prove convergence in $V$ for fixed $k$ and our arguments do not rely on the earlier assumption $\Vert \varphi_k \Vert_1 = 1$ anymore. Set the function
\[
\psi_k(s_1, \ldots, s_n; t_1, \ldots, t_n) =  \left( \prod_{j=1}^{n-1} \varphi_k(s_j \ldots s_{n-1} t_j) \right) \varphi_k(t_n s_n).
\]
When we see $\psi_k$ for fixed $t_1, \ldots, t_n$ as a function of the variables $s_1, \ldots, s_n$ we shall write $\psi_k( \: \cdot \: ; t_1, \ldots, t_n)$.
 By the last expression of \eqref{Eqn=Convolution} we see that
\begin{equation}\label{Eqn=Integrant}
\begin{split}
&B_{k,V} \leq
\int_{G^{\times n}} \vert (m  (\zeta_1, \ldots, \zeta_n) )(t_1^{-1}  t_2, \ldots,  t_{n-2}^{-1} t_{n-1}, t_n^{-1}  ) \vert  \\
& \quad \times    \Vert h_V^{\frac{1}{p}}   T_{ \psi_k(\: \cdot \: ; t_1, \ldots, t_n )\vert_{\Gamma^{\times n}  } }( x_1, \ldots, x_n  )   h_V^{\frac{1}{p}}  -  T_{ \psi_k(\: \cdot \: ; t_1, \ldots, t_n )  }( h_V^{\frac{1}{p_1}}  x_1  h_V^{\frac{1}{p_1}}, \ldots, h_V^{\frac{1}{p_n}}  x_n  h_V^{\frac{1}{p_n}}  )  \Vert_{L_p(\widehat{G})}   dt_1 \ldots t_n \\
\end{split}
\end{equation}
We justify that the integrand is dominated by an integrable function that does not depend on $V$. Indeed, note that $ T_{ \psi_k(\: \cdot \: ; t_1, \ldots, t_n )   }$ and $ T_{ \psi_k(\: \cdot \: ; t_1, \ldots, t_n )\vert_{\Gamma^{\times n}  } }$ are contractive (this follows from the expansions \eqref{Eqn=TExpansion} and \eqref{Eqn=SEmbedding} below) and in combination with Proposition \ref{Prop=AlmostIsometry1}   it follows that the integrand on the second line of \eqref{Eqn=Integrant} is at most 2 for any $t_1, \ldots, t_n \in G$. Since  $m(\zeta_1,\ldots,\zeta_n)$ is integrable, the integral \eqref{Eqn=Integrant} is majorized by
\[
\int_{G^{\times n}} 2 \vert (m (\zeta_1, \ldots, \zeta_n) )(t_1^{-1}  t_2,   \ldots,  t_{n-2}^{-1} t_{n-1}, t_n^{-1}  )\vert dt_1 \ldots t_n  = 2 \Vert m  (\zeta_1, \ldots, \zeta_n) \Vert_{L_1(G^{\times n})} < \infty.
\]
 So if we can show that the limit in $V$ of the integrand of \eqref{Eqn=Integrant} converges to 0  pointwise then by the Lebesgue dominated convergence theorem we  conclude that for all $k$ we have $\lim_V B_{k,V} = 0$.

Introduce the short hand notation for the multipliers $T_j = T_{\varphi_k( \: \cdot \: t_j )}$ and $S_j = T_{\varphi_k( \: \cdot \: t_j )\vert_\Gamma}$ for $1 \leq j \leq n-1$. We also set  $T_n = T_{\varphi_k( t_n \: \cdot \:  )}$ and $S_n = T_{\varphi_k(  t_n \: \cdot \: )\vert_\Gamma}$
For any $y_j \in L_{p_j}(\widehat{G})$ we have by Lemma \ref{Lem=Nest},
\begin{equation} \label{Eqn=TExpansion}
T_{ \psi_k(\: \cdot \: ; t_1, \ldots, t_n )  }(y_1, \ldots, y_n   )
= T_1(y_1(T_2(y_2 \ldots T_{n-1}(y_{n-1} ) \ldots ))) T_n(y_n).
\end{equation}
We also see that
\begin{equation}\label{Eqn=SEmbedding}
T_{ \psi_k(\: \cdot \: ; t_1, \ldots, t_n )\vert_{\Gamma^{\times n} }  }(x_1, \ldots, x_n   )
 = S_1(x_1(S_2(x_2 \ldots S_{n-1}(x_{n-1} ) \ldots ))) S_n(x_n).
\end{equation}
Now fix $y_{j} := y_{j,V} := h_V^{\frac{1}{p_j}}   x_j  h_V^{\frac{1}{p_j}}$; for a while $V$ shall be fixed and at the very end of the proof we will take a limit in $V$.
Set  $1 < q_j < \infty$ by
$q_j^{-1} = \sum_{i=j}^{n-1} p_i^{-1}$.
 Note that $q_1^{-1} + p_n^{-1} = p^{-1}$.  Then set for $0 \leq j \leq n-1$,
\[
R_{j, V} :=  T_1(y_{1}(T_2(y_{2} \ldots T_{j-1}(y_{j-1} T_j(y_{j}   h_V^{ \frac{1}{q_{j+1}} }  S_{j+1}( x_{j+1} S_{j+2}(x_{j+2} \ldots S_{n-1}(x_{n-1}) \ldots ) )     h_V^{ \frac{1}{q_{j+1}} }  )  ) \ldots ) )).
\]
By definition
\[
\begin{split}
R_{n-1,V} = & T_1(y_{1}(T_2(y_{2} \ldots T_{n-1}(y_{n-1} ) \ldots ))), \\
R_{0, V} = &  h_V^{\frac{1}{q_1}}  S_1(x_1(S_2(x_2 \ldots S_{n-1}(x_{n-1} ) \ldots ))) h_V^{\frac{1}{q_1}},
\end{split}
\]
and these expressions should be compared to \eqref{Eqn=TExpansion} and \eqref{Eqn=SEmbedding}.
Also let $R$ denote the term $
S_1(x_1(S_2(x_2\ldots S_{n-1}(x_{n-1})\ldots))) $.
$S_j$ and $T_j$ are contractions and therefore we record at this point that $ \Vert T_n(y_n) \Vert_{L_{p_n}(\widehat{G})},  \Vert h_V^{\frac{1}{p_n}} S_n(x_n) h_V^{\frac{1}{p_n}} \Vert_{L_{p_n}(\widehat{G})} \leq 1$.
Then -- in view of \eqref{Eqn=Integrant} -- we estimate using the triangle inequality
\begin{equation}\label{Eqn=LotsOfSummands}
\begin{split}
     & \Vert T_{ \psi_k(\: \cdot \: ; t_1, \ldots, t_n )  }(y_1, \ldots, y_n   ) -  h_V^{\frac{1}{p}}  T_{ \psi_k(\: \cdot \: ; t_1, \ldots, t_n )\vert_{\Gamma^{\times n} }  }(x_1, \ldots, x_n   )  h_V^{\frac{1}{p}} \Vert_{L_p(\widehat{G})} \\
  \leq & \Vert  R_{0,V} T_n(y_n) -  h_V^{\frac{1}{p}} R S_n(x_n) h_V^{\frac{1}{p}} \Vert_{L_p(\widehat{G})} +
  \sum_{j=1}^{n-1} \Vert (R_{j, V} -  R_{j-1, V})  T_n(y_n)  \Vert_{L_p(\widehat{G})}
  \\
  \leq & \Vert h_V^{\frac{1}{q_1}} [h_V^{1/p_n},R] S_n(x_n) h_V^{1/p_n} h_V^{1/q_1}\Vert_{L_p(\widehat{G})} + \Vert h_V^{1/q_1} R h_V^{1/p_n} [h_V^{1/q_1},S_n(x_n)]h_V^{1/p_n}\Vert_{L_p(\widehat{G})}
  \\
& + \Vert R_{0,V}  T_n(y_n) -  R_{0,V} h_V^{ \frac{1}{p_n}}  S_n(x_n) h_V^{\frac{1}{p_n}} \Vert_{L_p(\widehat{G})}
   +   \sum_{j=1}^{n-1} \Vert (R_{j, V} -  R_{j-1, V} )T_n(y_n) \Vert_{L_p(\widehat{G})}.
\end{split}
\end{equation}
Now, using H\"older's inequality, the first two summands are dominated by  $\Vert  [ h_V^{\frac{1}{p_n}},  R] \Vert_{L_{2p_n}(\widehat{G})}$ and $ \Vert  [ h_V^{\frac{1}{q_1}},  S_n(x_n)  ] \Vert_{L_{2q_1}(\widehat{G})}$ respectively, which go to $0$ by Lemma \ref{Lem=RicardCorollary}. For the third summand we recall that $y_n = h_V^{\frac{1}{p_n}} x_{n} h_V^{\frac{1}{p_n}}$ and that $T_n =  T_{\varphi_k( \: \cdot \: t_n )}$ is a multiplier both on $L_1$ and $L_\infty$. Therefore by the H\"older inequality and Proposition \ref{Prop=Intertwiner},
\[
     \Vert  R_{0,V}  T_n(y_n) -  R_{0,V}  h_V^{\frac{1}{p_n}}  S_n(x_n) h_V^{\frac{1}{p_n}} \Vert_{L_p(\widehat{G})}
\leq  \Vert   T_n(h_V^{\frac{1}{p_n}} x_{n}   h_V^{\frac{1}{p_n}}) -   h_V^{  \frac{1}{p_n}}  S_n(x_n) h_V^{\frac{1}{p_n}} \Vert_{L_{p_n}(\widehat{G})} \rightarrow 0.
\]
It remains to show that the final summand in \eqref{Eqn=LotsOfSummands} goes to 0 in $V$.
For $1 \leq j \leq n-1$, set
\[
\widetilde{S}_{j+1} =  S_{j+1}( x_{j+1} S_{j+2}(x_{j+2} \ldots S_{n-1}(x_{n-1}) \ldots ) \in \mathbb{C}[\Gamma].
\]
 We estimate,
\begin{equation}\label{Eqn=VeryFinalEstimate}
\begin{split}
\Vert (R_{j, V} - R_{j-1, V})T_n(y_n) \Vert_{L_p(\widehat{G})}
 \leq  &    \Vert   T_j(y_j h_V^{\frac{1}{q_{j+1}}}   \widetilde{S}_{j+1}     h_V^{\frac{1}{q_{j+1}}}  )     -  h_V^{\frac{1}{q_{j}}}   S_{j}( x_{j}  \widetilde{S}_{j+1}  )    h_V^{\frac{1}{q_{j}}}   \Vert_{L_{q_j}(\widehat{G})} \\
 \leq &  \Vert   T_j(  h_V^{\frac{1}{p_j}}   x_j   h_V^{\frac{1}{p_j}}   h_V^{\frac{1}{q_{j+1}}}   \widetilde{S}_{j+1}     h_V^{\frac{1}{q_{j+1}}}  )     -  T_j(  h_V^{\frac{1}{q_j}}   x_j     \widetilde{S}_{j+1}     h_V^{\frac{1}{q_{j}}}  )   \Vert_{L_{q_j}(\widehat{G})}
  \\
 & +\Vert   T_j(  h_V^{\frac{1}{q_j}}   x_j     \widetilde{S}_{j+1}     h_V^{\frac{1}{q_{j}}}  )   -  h_V^{\frac{1}{q_{j}}}   S_{j}( x_{j}  \widetilde{S}_{j+1}  )    h_V^{\frac{1}{q_{j}}}   \Vert_{L_{q_j}(\widehat{G})}.
\end{split}
\end{equation}
Since $T_j$ is a contraction the first summand goes to 0 in $V$ by Lemma \ref{Lem=RicardCorollary}. For the second summand we note that $T_j = T_{\varphi_k( \: \cdot \: t_j)}$ is a multiplier both on $L_1(\widehat{G})$ and $L_\infty(\widehat{G})$. Since $1 < q_j < \infty$  we get that also the second summand in \eqref{Eqn=VeryFinalEstimate} goes to 0 by Proposition \ref{Prop=Intertwiner}.
\end{proof}

\begin{remark}
If in Lemma \ref{Lem=AuxConverge} we have $x_1 = \ldots = x_n = 1$ then the proof holds for any symmetric neighbourhood basis $\mathcal{V}$ not necessarily witnessing the $\SAIN_\Gamma$ condition.
\end{remark}

\begin{remark}\label{Rmk=CBCaseThmC}
As in Remark \ref{Rmk=CBCaseThmA}, we note that Theorem C also has a completely bounded version.
Use the notation of Theorem \ref{Thm=MultiDeLeeuwRestriction} and let $N$ be any (semi-finite) von Neumann algebra. The vector-valued extension $T_m^{N}$ of $T_m$ is given by $L_{p_1}(\widehat{G}; L_{p_1}(N)) \times  L_{p_2}(\widehat{G}; ; L_{p_2}(N)) \times \ldots \times L_{p_n}(\widehat{G};L_{p_n}(N)) \rightarrow  L_{p}(\widehat{G}; L_{p}(N))$,
\[
\begin{split}
  & T_m( x_1 \otimes  \lambda(f_1) , \ldots, x_n \otimes   \lambda(f_n)  ) \\
   & \qquad = \int_{G^{\times n}} m(s_1, \ldots, s_n)  f_1(s_1) \ldots f_n(s_n) \: x_1 \ldots x_n \otimes \lambda(s_1 \ldots s_n) ds_1 \ldots ds_n.
\end{split}
\]
Let $\mathcal{R}$ be the hyperfinite II$_1$-factor.
As in the linear case, we will say that $T_m$ is completely bounded if $T_m^{\mathcal{R}}$ is bounded and write $\Vert T_m \Vert_{\mathrm{cb}} = \Vert T_m^{\mathcal{R}} \Vert$. Then we still have \eqref{Eqn=DeLeeuwMultilinear} in the cb-norm.
\end{remark}

\section{Multilinear lattice approximation}\label{Sect=Lattice}

The aim of this section is to prove a multilinear version of \cite[Theorem C]{CPPR}.

\begin{definition}\label{Dfn=ADS}
A locally compact group $G$ is approximable by discrete subgroups (notation $G \in \ADS$) if there exists a sequence of discrete lattices $(\Gamma_j)_{j \in \mathbb{N}}$ of $G$ with associated fundamental domains $X_j$ which form a neigbourhood basis of the identity.
\end{definition}

 By passing to a subsequence we may assume that all $X_j$ in Definition \ref{Dfn=ADS} are relatively compact. In other words $\Gamma_j$ is cocompact.  The following theorem shows that for $\ADS$ groups the norm of Fourier multipliers can be approximated by looking at the restriction of the symbol to appropriate lattices.

\begin{theorem}\label{thm:LatticeApproximation}
Let $G \in \ADS$  with approximating cocompact lattices $(\Gamma_j)_{j \in \mathbb{N}}$. Let $m \in C_b( G^{\times n})$  and let $1 \leq p, p_1, \ldots,  p_n < \infty$ be such that $p^{-1}= \sum_{i=1}^n p_i^{-1}$. Then,
\begin{equation} \label{Eqn=LatticeApprox}
\Vert T_m: L_{p_1}(\widehat{G}) \times \ldots \times  L_{p_n}( \widehat{G} )  \rightarrow   L_p(\widehat{G}) \Vert \leq \sup_{j \geq 1}  \Vert T_{  m \vert_{  \Gamma_j^{ \times n }} }: L_{p_1}(\widehat{\Gamma_j}) \times \ldots \times L_{p_n}(\widehat{\Gamma_j})  \rightarrow L_p(\widehat{\Gamma_j}) \Vert.
\end{equation}
\end{theorem}
Before proving the theorem we introduce the auxiliary maps from \cite[Theorem C]{CPPR}. Let $X_j$ be shrinking relatively compact fundamental domains for the inclusions of $\Gamma_j$ in $G$. For $s \in G$ we shall write $\gamma_j(s)$ for the unique element in $\Gamma_j$ such that $s \in \gamma_j(s) X_j$.
Let $h_j = \lambda(1_{X_j})$. For $x \in \mathbb{C}[  \Gamma_j  ]$   we set the elements in $L_\infty(\widehat{G})$ by
\[
\Phi_j(x) = h_j^\ast x h_j, \qquad \Phi_j^{(p)}(x) = \vert X_j \vert^{-2 + \frac{1}{p}} \Phi_j(x).
\]
The proof of \cite[Theorem C, p. 19-20]{CPPR} shows that $\Phi_j^{(p)}$ extends to a contraction $L_p(\widehat{\Gamma}) \rightarrow L_p(\widehat{G})$. For $x \in \lambda(C_c(G))$ we set
\begin{equation}\label{Eqn=PsiMap}
\Psi_j(x) = \sum_{\gamma \in \Gamma_j} \tau(  h_j^\ast \lambda( \gamma^{-1} )  h_j x ) \lambda(\gamma) =
 \sum_{\gamma \in \Gamma_j}  \langle  h_j x,  \lambda( \gamma )  h_j  \rangle   \lambda(\gamma), \qquad \Psi_j^{(p)}(x) = \vert X_j \vert^{-1-\frac{1}{p}} \Psi_j(x).
\end{equation}
Since $x$ has compact frequency support these summations are finite.    The proof of \cite[Theorem C, p. 19-20]{CPPR} shows that $\Psi_j^{(p)}$  extends to a contraction $L_p(\widehat{G})  \rightarrow L_p(\widehat{\Gamma})$.
Now set
\[
S_j = \Phi_j^{(p)} \circ T_{m\vert_{\Gamma_j^{\times n}  }  } \circ (\Psi_j^{(p_1)} \times \ldots \times \Psi_j^{(p_n)}) = \vert X_j \vert^{-2-n}   \Phi_j \circ T_{ m\vert_{\Gamma_j^{ \times n }} } \circ (\Psi_j \times \ldots \times \Psi_j).
 \]
Recall that we defined $\phi^\vee(s) = \phi(s^{-1})$ for $\phi$ a function on $G$. The proof of Theorem \ref{thm:LatticeApproximation} hinges on the following lemma.

\begin{lemma}\label{Lem=AuxLattice} Let $f_1, \ldots, f_n, \phi \in C_c(G)^{\ast 2}$ so that $x_i := \lambda(f_i) \in L_{p_i}(\widehat{G})$ and $y := \lambda(\phi^\vee) \in L_{p'}(\widehat{G})$. We have
\begin{equation}\label{Eqn=SConvergence}
\lim_j \langle y, S_j(x_1, \ldots, x_n) \rangle_{p', p} = \langle y,  T_m( x_1, \ldots, x_n)  \rangle_{p', p}.
\end{equation}
\end{lemma}

\begin{proof}[Proof of   Lemma \ref{Lem=AuxLattice}]
Set  for $x =\lambda(g)$ with $g \in C_c(G)$,
\[
\Psi_j'(x) =  \sum_{\gamma \in \Gamma_j}  \langle x, \lambda(\gamma)h_j  \rangle   \lambda(\gamma) =  \sum_{\gamma \in \Gamma_j}  \langle g, 1_{\gamma X_j}  \rangle   \lambda(\gamma).
\]
 Then set
\[
S_j'=  \vert X_j \vert^{-2}   \Phi_j \circ T_{m\vert_{\Gamma_j^{ \times  n }  }} \circ (\Psi_j' \times \ldots \times \Psi_j').
\]
We then have for $g_1, \ldots, g_n \in C_c(G)^{\ast 2}$,
\[
\begin{split}
S_j'(\lambda(g_1), \ldots, \lambda(g_n)) = &\vert X_j \vert^{-2}
 \sum_{\gamma_1, \ldots, \gamma_n \in \Gamma_j} m(\gamma_1, \ldots, \gamma_n )
   \left( \prod_{i=1}^n \langle g_i, 1_{\gamma_i X_j} \rangle  \right)\lambda(h_j)^\ast  \lambda(  \gamma_1 \ldots \gamma_n ) \lambda(h_j). \\
\end{split}
\]
Therefore we see by \eqref{Eqn=PairingConcrete} that
\[
\begin{split}
& \langle \lambda(\phi^\vee), S_j'(\lambda(g_1), \ldots, \lambda( g_n)) \rangle_{p', p} \\
= &
 \vert X_j \vert^{-2}  \sum_{\gamma_1, \ldots, \gamma_n \in \Gamma_j}   \int_{G} \int_G  m(\gamma_1, \ldots, \gamma_n ) \phi( t_0 \gamma_1 \ldots \gamma_n t_{1} )
   \left( \prod_{i=1}^n \langle g_i, 1_{\gamma_i X_j} \rangle  \right)    h_j^\ast(t_0)   h_j(t_{1})    dt_0 dt_1 \\
     = & \vert X_j \vert^{-2} \!\!\!\!   \int_{G^{\times n}}  \!\!  \int_{X_j} \int_{X_j^{-1}}   m(\gamma_j(s_1), \ldots, \gamma_j(s_n) ) \phi( t_0 \gamma_j(s_1) \ldots \gamma_j(s_n) t_{1} )
     g_1(s_1) \ldots g_n(s_n)      dt_0 dt_1 ds_1\ldots s_n. \\
\end{split}
\]
On the other hand we have
\[
 \langle \lambda(\phi^\vee), T_m(\lambda(g_1), \ldots, \lambda( g_n)) \rangle_{p', p} =
     \int_{G^{\times n}}  m( s_1, \ldots, s_n ) \phi( s_1 \ldots s_n  )
     g_1(s_1) \ldots g_n(s_n)       ds_1\ldots s_n.
\]
For $\epsilon > 0$ we may choose $j$ large such that for all $t_0 \in X_j^{-1}, t_1 \in X_j$ and $s_i$ in the support of $g_i$,
\[
\vert  m(\gamma_j(s_1), \ldots, \gamma_j(s_n) ) \phi( t_0 \gamma_j(s_1) \ldots \gamma_j(s_n) t_{1} ) -
 m( s_1, \ldots, s_n ) \phi( s_1 \ldots s_n  ) \vert < \epsilon.
\]
It follows then that
\begin{equation} \label{Eqn=PairingEstimate1}
\vert \langle \lambda(\phi^\vee), S_j'(\lambda(g_1), \ldots, \lambda( g_n))  - T_m(\lambda(g_1), \ldots, \lambda( g_n))   \rangle_{p', p} \vert \leq \epsilon \Vert g_1 \Vert_{L_1(G)} \ldots \Vert g_n \Vert_{L_1(G)}.
\end{equation}
As in the claim let $f_i \in C_c(G)^{\ast 2}, x_i = \lambda(f_i)$. Set $g_{i,j} = \vert X_j \vert^{-1} f_i \ast 1_{X_j}$.
 Since $C_c(G) \ast 1_{X_j} \subseteq C_c(G)$  we have still have  $g_{i,j} \in C_c(G)^{\ast 2}$. By construction,
\[
S_j( x_1, \ldots, x_n) =  S_j'(\lambda(g_{1,j}), \ldots, \lambda( g_{n,j})).
\]
As each $f_i$ is continuous we have
\begin{equation} \label{Eqn=ConvApprox}
\lim_j \Vert g_{i,j} - f_i \Vert_{L_1(G)} =   \Vert \vert X_j \vert^{-1} f_i \ast 1_{X_j} - f_i \Vert_{L_1(G)} = 0.
\end{equation}
 We crudely estimate for $1 \leq i \leq n$,
\begin{equation}
\begin{split}
    & \vert \langle \lambda(\phi^\vee), T_m(\lambda(f_1), \ldots, \lambda(f_{i-1}),    \lambda(f_{i} - g_{i,j}), \lambda(g_{i,j}),  \ldots , \lambda( g_{n,j} ) \rangle_{p', p} \vert \\
\leq  & \int_{ G^{\times n} } \!\!\!\!\! \vert  m( s_1, \ldots, s_n ) \phi( s_1 \ldots s_n  )
     f_1(s_1) \ldots f_{i-1}(s_{i-1})  \\
     & \qquad \times \qquad  (f_i-g_{i,j})(s_i) g_{i+1, j}(s_{i+1})  \ldots   g_{n,j}(s_n)      \vert  ds_1\ldots s_n  \\
\leq & \Vert m \phi \Vert_\infty \Vert f_1   \Vert_{L_1(G)} \ldots \Vert f_{i-1}   \Vert_{L_1(G)} \Vert f_i - g_{i,j} \Vert_{L_1(G)} \Vert g_{i+1, j} \Vert_{L_1(G)}  \ldots  \Vert  g_{n,j} \Vert_{L_1(G)}.
 \end{split}
\end{equation}
Therefore by the triangle inequality and \eqref{Eqn=ConvApprox},  for every $\epsilon >0$ there exists large $j$ such that,
\begin{equation}\label{Eqn=PairingEstimate2}
\begin{split}
      &  \vert  \langle \lambda(\phi^\vee), T_m(\lambda(f_1), \ldots, \lambda(f_{n} )    )  -    T_m(  \lambda(  g_{1,j})  , \ldots , \lambda(   g_{n,j}))  \rangle_{p', p} \vert \\
 \leq & \Vert m \phi \Vert_\infty  \sum_{i=1}^n  \Vert f_1   \Vert_{L_1(G)} \ldots \Vert f_{i-1}   \Vert_{L_1(G)}  \Vert f_i - g_{i,j} \Vert_{L_1(G)}   \Vert  g_{i+1,j} \Vert_{L_1(G)} \ldots  \Vert  g_{n,j} \Vert_{L_1(G)} < \epsilon.
\end{split}
\end{equation}
Therefore, combining  \eqref{Eqn=PairingEstimate1} and \eqref{Eqn=PairingEstimate2}, we have for $j$ large that
\[
\begin{split}
& \vert \langle \lambda(\phi^\vee), S_j(x_1, \ldots, x_n) -   T_m(x_1, \ldots,  x_n  )  \rangle_{p', p} \vert \\
\leq &
\vert  \langle \lambda(\phi^\vee), S_j'(\lambda(g_{1,j}), \ldots, \lambda( g_{n,j} )) -
T_m(\lambda(g_{1,j}), \ldots, \lambda(g_{n,j})  )   \rangle_{p', p} \vert \\
& + \vert   \langle \lambda(\phi^\vee), T_m(\lambda(g_{1,j} ), \ldots, \lambda( g_{n,j})  )
-   T_m(\lambda(f_1), \ldots, \lambda(f_{n})  )  \rangle_{p', p} \vert \\
\leq & \epsilon \Vert g_{1,j} \Vert_1 \ldots \Vert g_{n,j} \Vert_1  + \epsilon.
\end{split}
\]
Hence we see that this term goes to 0 proving \eqref{Eqn=SConvergence}.
\end{proof}

\begin{proof}[Proof of Theorem \ref{thm:LatticeApproximation}] Let $x_1, \ldots, x_n$ be as in Lemma \ref{Lem=AuxLattice} and assume $\Vert x_i \Vert_{L_{p_i}(\widehat{G})} \leq 1$.  Lemma \ref{Lem=AuxLattice} shows (for $p=1$ this requires Kaplansky's density theorem),
\[
\begin{split}
\Vert T_m(  x_1, \ldots, x_n) \Vert_{L_p(\widehat{G})} = & \sup_{ y \in \lambda(C_c(G)^{\ast 2}), \Vert y \Vert_{L_{p'}(\widehat{\Gamma})} \leq 1 } \vert \langle y , T_m(  x_1, \ldots, x_n) \rangle_{p', p} \vert
\\
= & \sup_{ y \in \lambda( C_c(G)^{\ast 2}), \Vert y \Vert_{L_{p'}(\widehat{\Gamma})} \leq 1 }  \limsup_j  \vert \langle y, S_j( x_1, \ldots, x_n) \rangle_{p', p} \vert.
\end{split}
\]
Then,  as $\Phi_j^{(p)}$ and $\Psi_j^{(p_i)}$ in the definition of $S_j$ are contractions,
\[
\begin{split}
\Vert T_m(  x_1, \ldots, x_n)  \Vert_{L_p(\widehat{G})} \leq &  \limsup_j  \Vert S_j(  x_1, \ldots, x_n)  \Vert_{L_p(\widehat{G})}  \\
\leq &   \sup_{j \geq 1}  \Vert T_{  m\vert_{\Gamma_j^{\times n}    }  }: L_{p_1}(\widehat{\Gamma_j }) \times \ldots \times L_{p_n}(\widehat{\Gamma_j }) \rightarrow  L_p(\widehat{\Gamma_j}) \Vert.
\end{split}
\]
This concludes the proof.
\end{proof}

We note that we may use the lattice approximation theorem to obtain a stronger version of the restriction theorem. The following follows directly by combining Theorem \ref{Thm=MultiDeLeeuwRestriction} and  Theorem \ref{thm:LatticeApproximation}.

\begin{theorem}
Let $H$ be a subgroup of a locally compact group $G$. Let $H \in \ADS$ and $G \in \SAIN_H$. Let $m:G^{\times n} \to \mathbb{C}$ be a bounded continuous symbol giving rise to a multilinear $(p_1,p_2,\ldots,p_n)$-multiplier on $G$ with $1\leq p,p_i< \infty$ and $p^{-1}=\sum_{i=1}^n p_i^{-1}$. Then

\[
\Vert T_{m \vert_{H^{\times n}} }: L_{p_1}(\widehat{H}) \times \ldots \times L_{p_n}(\widehat{H})   \rightarrow L_p(\widehat{H}) \Vert \leq
\Vert T_{m  }: L_{p_1}(\widehat{G}) \times \ldots \times L_{p_n}(\widehat{G})  \rightarrow L_p(\widehat{G}) \Vert.
\]
\end{theorem}

\section{Other multilinear de Leeuw theorems: periodization and compactification}\label{Sect=OtherTheorems}
In this section we prove multilinear versions of the periodization and compactification theorem. The proofs in this section are very similar to the linear case in \cite{CPPR} or to other proofs in this paper. Therefore, we only give a sketch of the proofs.

Given a locally compact group $G$, let $G_{\mathrm{disc}}$ denote the same group equipped with the discrete topology. The compactification theorem relates the $(p_1,p_2,\ldots,p_n)$-boundedness of a Fourier multiplier on $G$ with the $(p_1,p_2,\ldots,p_n)$-boundedness of a Fourier multiplier on $G_{\mathrm{disc}}$. We recall that if $G$ is abelian $\widehat{G_{\mathrm{disc}}}$ is known as the Bohr compactification of $\widehat{G}$.  We refer to  the discussion in \cite[Section 6]{CPPR} for why the hypotheses in the following theorem are the most natural ones.

\begin{theorem} Let $G$ be a locally compact group and let $m \in C_b(G^{\times n})$. Let $1<p,p_1, \ldots, p_n <\infty$, with $p^{-1}=\sum_{i=1}^n p_i^{-1}$.
\begin{itemize}
    \item[i)] If $G \in \ADS$, then
    \[
    \Vert T_m: L_{p_1}(\widehat{G}) \times \ldots \times  L_{p_n}( \widehat{G} )  \rightarrow   L_p(\widehat{G}) \Vert \leq \Vert T_m: L_{p_1}(\widehat{G_{\mathrm{disc}}}) \times \ldots \times  L_{p_n}( \widehat{G_{\mathrm{disc}}} )  \rightarrow   L_p(\widehat{G_{\mathrm{disc}}}) \Vert.
    \]
    \item[ii)] If $G_{\mathrm{disc}}$ is amenable, then
    \[
    \Vert T_m: L_{p_1}(\widehat{G_{\mathrm{disc}}}) \times \ldots \times  L_{p_n}( \widehat{G_{\mathrm{disc}}} )  \rightarrow   L_p(\widehat{G_{\mathrm{disc}}}) \Vert \leq \Vert T_m: L_{p_1}(\widehat{G}) \times \ldots \times  L_{p_n}( \widehat{G} )  \rightarrow   L_p(\widehat{G}) \Vert.
    \]
\end{itemize}

\begin{proof}
(i) Let $(\Gamma_j)_{j \in \mathbb{N}}$ be a sequence of approximating cocompact lattices as in the hypothesis of Theorem \ref{thm:LatticeApproximation}.
Then,
\[
\Vert T_m: L_{p_1}(\widehat{G}) \times \ldots \times  L_{p_n}( \widehat{G} )  \rightarrow   L_p(\widehat{G}) \Vert \leq \sup_{j \geq 1}  \Vert T_{  m \vert_{  \Gamma_j^{ \times n }} }: L_{p_1}(\widehat{\Gamma_j}) \times \ldots \times L_{p_n}(\widehat{\Gamma_j})  \rightarrow L_p(\widehat{\Gamma_j}) \Vert.
\]
Now $\Gamma_j < G_{\mathrm{disc}}$ is an inclusion of two discrete groups so that the natural inclusion $\mathbb{C}[\Gamma] \subseteq \mathbb{C}[G_{\mathrm{disc}}]$ extends to an isometric inclusion $L_p(\widehat{\Gamma}_j) \subseteq  L_p(\widehat{G_{\mathrm{disc}}})$.  Therefore, $T_{m \vert_{\Gamma_j^{\times n}} }$ is just a restriction of $T_m$ and hence
\[
\Vert T_{m \vert_{\Gamma_i^{\times n}} }: L_{p_1}(\widehat{\Gamma_i}) \times \ldots \times L_{p_n}(\widehat{\Gamma_i})   \rightarrow L_p(\widehat{\Gamma_i}) \Vert \leq
\Vert T_{m  }: L_{p_1}(\widehat{G_{\mathrm{disc}}}) \times \ldots \times L_{p_n}(\widehat{G_{\mathrm{disc}}})  \rightarrow L_p(\widehat{G_{\mathrm{disc}}}) \Vert.
\]

(ii)  This is essentially a restriction result for the subgroup $G_{\mathrm{disc}}$. As such, the proof is similar to Theorem \ref{Thm=MultiDeLeeuwRestriction} and we sketch its main difference here. From \cite[Theorem 8.7]{CPPR} we know that the amenability of $G_{\mathrm{disc}}$  implies that $G \in \SAIN_G$.
Now for $F \subset G$ finite let $\Gamma_F$ be the smallest (not necessarily closed) subgroup of $G$ containing $F$. Then $\Gamma_F$ is countable and we may still perform the construction in Remark \ref{Rmk=Shrink} with $\Gamma$ replaced by $\Gamma_F$. Hence we  obtain a local neighbourhood basis $\mathcal{V}$ of the identity of $G$ such that for every $F' \subseteq \Gamma_F$ finite we have $\lim_{V \in \mathcal{V}} \delta_{F'}(V) = 1$.
 Then Proposition \ref{Prop=AlmostIsometry1} still holds for any $x \in \mathbb{C}[\widehat{ \Gamma_F }]$ with $V \in \mathcal{V}$ sufficiently small.

  By \cite[Lemma 6.1]{CPPR} the amenability of $G_{\mathrm{disc}}$ implies that for any $x \in \mathbb{C}[ G_{\mathrm{disc}} ]$ we have $\Vert x \Vert_{ C_r^\ast( G_{\mathrm{disc}}  ) } = \Vert x \Vert_{ C_r^\ast(  G  ) }$. So naturally we have an isometric embedding $C_r^\ast(   G_{\mathrm{disc}}  ) \subseteq L_\infty( \widehat{ G  }  )$. Consequently, the proof of  Proposition \ref{Prop=AlmostIsometry1} still holds, where for the $p=\infty$ case the latter fact is used.
  Since $\Gamma_F$ is a discrete subgroup of $ G_{\mathrm{disc}}$ we have  $C_r^\ast(   \Gamma_F  )  \subseteq C_r^\ast( G_{\mathrm{disc}}  )$ which thus is included in $L_\infty( \widehat{ G  }  )$. Then  Lemma \ref{Lem=RicardCorollary} and Proposition \ref{Prop=AlmostIsometrySymmetric} hold for any $x \in C_r^\ast(  \Gamma_F  )$.

      This suffices then to run the proofs of Theorem \ref{Thm=MultiDeLeeuwRestriction} and Lemma \ref{Lem=AuxConverge} by putting $F$ as in the previous paragraph equal to the union of the frequency supports of $x_1, \ldots, x_n$ occurring in the proof; note that products of the $x_i$'s then have frequency support contained in $\Gamma_F$.
\end{proof}
\end{theorem}

The proof of the periodization theorem is a direct modification of the linear case (Theorem D (iii) and (iv) in \cite{CPPR}).

\begin{theorem}
Let $G$ be a locally compact group and let $H$ be a closed normal subgroup of $G$. Let $m_q \in C_b((G/H)^{\times n})$ and let $m_\pi \in C_b( G^n)$ denote its $H$-periodization $m_\pi(g_1,\ldots,g_n)=m(g_1H,g_2H,\ldots,g_nH)$. Then, the following inequalities hold for $1<p,p_i<\infty$, $p^{-1}=\sum_{i=1}^n p_i^{-1}$.

\begin{itemize}
    \item [(i)] If $G$ is abelian,
    \[ \Vert T_{m_\pi}: L_{p_1}(\widehat{G})\times \ldots \times L_{p_n}(\widehat{G}) \to L_p(\widehat{G})\Vert \leq \Vert T_{m_q}: L_{p_1}(\widehat{G/H})\times \ldots \times L_{p_n}(\widehat{G/H})\to L_p(\widehat{G/H}) \Vert
    \]
    \item [(ii)] If $H$ is compact,
     \[ \Vert T_{m_q}: L_{p_1}(\widehat{G/H})\times \ldots \times L_{p_n}(\widehat{G/H})\to L_p(\widehat{G/H}) \Vert \leq \Vert T_{m_\pi}: L_{p_1}(\widehat{G})\times \ldots \times L_{p_n}(\widehat{G}) \to L_p(\widehat{G})\Vert
    \]
\end{itemize}
\end{theorem}

\begin{proof}
(i) was proved in the linear case by Saeki \cite{saeki1970translation} and in the multilinear case by \cite{rodriguez2013homomorphism}. The proof of (ii) remains essentially unchanged from the linear case proved in \cite[Section 7]{CPPR}.
As in \cite[Section 7]{CPPR} we note that the operator defined by
\[
\Pi = \int \lambda(h) d\mu_H(h) \in \mathcal{L}(H) \subset \mathcal{L}(G)
\]
is a  central projection of $\mathcal{L}(G)$ such that $\lambda(s) \Pi = \Pi = \Pi \lambda(s)$ for all $s \in H$. Moreover, this lets us define a normal $\ast$-homomorphism
\[
\pi: \mathcal{L}(G/H) \to \mathcal{L}(G)\Pi: \lambda(gH) \mapsto \lambda(g) \Pi.
\]
 $\pi$ preserves the Plancherel weight on $\mathcal{L}(G/H)$, so it induces an isometry $\pi_p: L_p(\widehat{G/H})\to L_p(\widehat{G})\Pi$ for $1\leq p\leq \infty$. Since $\Pi$ is a central projection in $\mathcal{L}(G)$, a similar computation as in \cite{CPPR} now shows us that
\[
\begin{split}
  & \pi_p \circ T_{m_q}(x_1,\ldots,x_n)  \\
   & =\int_{(G/H)^n} m_q(g_1H,g_2H,\ldots,g_nH) \widehat{x_1}(g_1H)\ldots \widehat{x_n}(g_nH) \lambda(g_1\ldots g_n) \Pi  d\mu_{G/H}(g_1H)\ldots d\mu_{G/H}(g_nH)  \\
   & = \int_{(G/H)^n} m_q(g_1H,g_2H,\ldots,g_nH) \widehat{x_1}(g_1H)\ldots \widehat{x_n}(g_nH) \int_H \lambda(g_1\ldots g_n h) d\mu_H(h) d\mu_{(G/H)^n} \\
   & = \int_{G^n} m_\pi(g_1,\ldots,g_n) \widehat{x_1}(g_1H)\ldots \widehat{x_n}(g_nH) \lambda(g_1\ldots g_n)  d\mu_{G}(g_1)\ldots d\mu_{G}(g_n) \Pi \\
 &=   T_{m_\pi}(\pi_{p_1}(x_1),\ldots,\pi_{p_n}(x_n))
\end{split}
\]
for $x_i =\lambda(\widehat{x_i}) \in \lambda(C_c(G/H)^{\ast2})$. This concludes (ii) since $\pi_p$ is an isometry for all $1\leq p\leq \infty$.
\end{proof}

\section{Example: Multilinear multipliers on the Heisenberg group}\label{Sect=Heisenberg}

We provide an example of multilinear multipliers on a non-abelian group. Let
\[
H = \left\{
\left(
\begin{array}{ccc}
1 & x & z \\
0 & 1 & y \\
0 & 0 & 1
\end{array}
\right)
\mid x,y,z \in \mathbb{R}
 \right\}
\]
be the $(2+1)$-dimensional Heisenberg group.
We shall see $H$ as $\mathbb{R}^3$ with the group law
\[
(x,y,z).(x',y',z')=(x+x',y+y',z+z'+ xy').
\]
  $H$ is unimodular and its Haar measure is just the usual Lebesgue measure of $\mathbb{R}^3$.
There is an action $\mathbb{R} \curvearrowright \mathbb{R}^2$ by $x \cdot (y,z) = (y, z + xy)$. Then $H = \mathbb{R}^2 \rtimes \mathbb{R}$. Recall our definition of the completely bounded norm from Remark \ref{Rmk=CBCaseThmC}.

\begin{proposition}\label{Prop=Transference}
Let $\Lambda$ and $\Gamma$ be countable amenable groups.  Let $\Lambda \curvearrowright \Gamma$ be an action by group automorphisms.
Let $m \in \ell_\infty(\Lambda \times \Lambda)$ and set $M(\gamma, s, \mu, t) = m(s,t), s,t \in \Lambda, \gamma, \mu \in \Gamma$. We have for $1 \leq p, p_1, p_2 < \infty$ with $p^{-1} = p_1^{-1} + p_2^{-1}$,
\begin{equation}\label{Eqn=CBmap}
\Vert T_M: L_{p_1}(\widehat{\Gamma \rtimes \Lambda}) \times L_{p_2}(\widehat{\Gamma \rtimes \Lambda}) \rightarrow  L_{p}(\widehat{\Gamma \rtimes \Lambda}) \Vert  \leq
\Vert T_m: L_{p_1}(\widehat{\Lambda}   ) \times L_{p_2}(\widehat{\Lambda}   )     \rightarrow L_{p}(\widehat{\Lambda}   )  \Vert_{\mathrm{cb}}.
\end{equation}
\end{proposition}
\begin{proof}
The proof follows closely \cite{gonzalez2018crossed} with the difference that we need a bilinear version of transference which requires amenability in the intertwining properties of \cite[Theorem 2.2]{gonzalez2018crossed}.   For $F \subseteq \Lambda$ finite let $P_F$ be the orthogonal projection onto the linear span of the delta functions $\delta_s, s \in F$. Set $P_s = P_{\{ s \}}$.  Let $\cS_p = L_p( B(\ell_2(\Lambda)) )$ where the $L_p$-space is taken with respect to the trace on  $B(\ell_2(\Lambda))$.

Consider the bilinear Hertz-Schur multiplier:
\[
 S_{M }: L_{p_1}( \widehat{\Gamma} )  \otimes \mathcal{S}_{p_1}   \times  L_{p_2}( \widehat{\Gamma} )  \otimes \mathcal{S}_{p_2}      \rightarrow
L_{p}( \widehat{ \Gamma } )  \otimes \mathcal{S}_{p},
\]
determined on finite rank operators $(y_{1,s,t})_{s,t\in \Lambda}$ and $(y_{2,s,t})_{s,t\in \Lambda}$ by
\[
x_1 \otimes (y_{1, s,t})_{s,t\in \Lambda} \times x_2 \otimes (y_{2, s,t})_{s,t \in \Lambda} \mapsto  x_1 x_2 \otimes ( \sum_{r \in \Lambda} m(sr^{-1},rt^{-1}) y_{1, s, r} y_{2, r, t}  )_{s,t}.
\]
 Set the unitary $U =  \sum_{s \in \Lambda}   P_s \otimes \lambda(s) \in B(\ell_2(\Lambda))\otimes L_{\infty}(\widehat{\Lambda})$. Then set the trace preserving $\ast$-homomorphism,
\[
\begin{split}
\pi:  L_\infty(\widehat{\Gamma}) \otimes B(\ell_2( \Lambda)) & \rightarrow L_\infty(\widehat{\Gamma}) \otimes B(\ell_2( \Lambda)) \otimes L_\infty( \widehat{\Lambda } ), \\
  x & \mapsto (1 \otimes  U) (x \otimes 1 ) (1  \otimes U)^\ast.
\end{split}
\]
Then $\pi$ determines an isometry
\[
\pi_p: L_p(\widehat{\Gamma}) \otimes \mathcal{S}_p  \rightarrow L_p(\widehat{\Gamma}) \otimes \mathcal{S}_p \otimes L_p(\widehat{\Lambda}) \simeq   L_p(\widehat{\Lambda};
L_p(\widehat{\Gamma}) \otimes \mathcal{S}_p ).
\]
As described in Remark \ref{Rmk=CBCaseThmA}, $L_p(\widehat{\Gamma}) \otimes \mathcal{S}_p$ here refers to the tensor product equipped with the $L_p$ norm given by the tensor products of the weights on $\Gamma$ and $\mathbb{N}$.
 Since $\Gamma$ is amenable we have a trace preserving embedding $\mathcal{L}(\Gamma) \subseteq \mathcal{R}$ by Connes' theorem \cite{Connes} where $\mathcal{R}$ is the hyperfinite II$_1$ factor.
Also $\mathcal{R} \otimes M_k(\mathbb{C})$ embeds into $\mathcal{R} \otimes \mathcal{R} \simeq \mathcal{R}$ in a trace preserving way.
Therefore - approximating $\cS_p$ with $L_p(M_k(\mathbb{C}))$  - since $T_m$ is completely bounded we may extend $T_m$ to a vector valued map
\[
T_m^{\mathrm{vec}}:  L_{p_1}(\widehat{\Lambda};  L_{p_1}(\widehat{\Gamma}) \otimes \mathcal{S}_{p_1}  ) \times   L_{p_2}(\widehat{\Lambda} ; L_{p_2}(\widehat{\Gamma}) \otimes \mathcal{S}_{p_2}   )     \rightarrow  L_{p}(\widehat{\Lambda} ; L_{p}(\widehat{\Gamma}) \otimes \mathcal{S}_{p}   );
\]
and we have
\[
 T_{m  }^{\mathrm{vec}} \circ  (\pi_{p_1} \times \pi_{p_2}) =   \pi_p \circ S_{M }.
\]
Therefore $S_{M}$ is bounded  by the norm of the right hand side of \eqref{Eqn=CBmap}.

We now use the map $S_{M}$ to prove the statement. The von Neumann algebra $L_\infty(\widehat{\Gamma \rtimes \Lambda})$ is isomorphic to $L_\infty(\widehat{\Gamma})  \rtimes \Lambda$ where $\lambda_{\Gamma \rtimes \Lambda}(s,t), s \in \Gamma, t \in \Lambda$ is identified with $\sum_{r \in \Lambda}  \lambda_{\Gamma}( r^{-1} \cdot s)   \otimes e_{ r, t^{-1} r } \in B(\ell_2(\Gamma) \otimes \ell_2(\Lambda))$.
Under this identification
\[
\lambda_{\Gamma \rtimes \Lambda}(s,t) (1\otimes P_F) \lambda_{\Gamma \rtimes \Lambda}(s,t)^\ast = (1 \otimes P_{t F}) \in B(\ell_2(\Gamma) \otimes \ell_2(\Lambda)).
\]
 Let $(F_\alpha)_{\alpha \in \mathbb{N}}$ be the F\o{}lner sequence for $\Lambda$.
By  \cite[Lemma 2.1]{gonzalez2018crossed}, with $x_\alpha = \vert F_\alpha \vert^{-\frac{1}{2}} P_{F_\alpha}$, there exist contractions
\[
j_{p, \alpha }:  L_p(  L_\infty(\widehat{\Gamma}) \rtimes \Lambda ) \rightarrow  L_{p}( \widehat{\Gamma} )  \otimes \mathcal{S}_{p}:
x \mapsto \vert F_\alpha \vert^{-\frac{1}{p}} ( 1 \otimes P_{F_\alpha}) x  (1 \otimes P_{F_\alpha}).
\]
 We claim further that for $x,y,z \in \mathbb{C}[\Gamma \rtimes \Lambda]$ and any non-principal ultrafilter $\omega$ on $\mathbb{N}$,
\begin{equation} \label{Eqn=UltraIntertwiner}
\langle \prod_{\alpha, \omega} S_{M }  \circ ( j_{p_1,\alpha } \times j_{p_2,\alpha})(x,y), \prod_{\alpha, \omega} j_{p',\alpha}(z) \rangle_{p, p'} =  \langle  T_{M}(x,y), z \rangle_{p, p'}.
\end{equation}
Indeed, by linearity we may assume that $x = \lambda_{\Gamma \rtimes \Lambda}(  s_1,t_1 ), y = \lambda_{\Gamma \rtimes \Lambda}(  s_2,t_2 ), z =\lambda_{\Gamma \rtimes \Lambda}(  s_3,t_3 ), s_i \in \Gamma, t_i \in \Lambda$.  We have
\[
\begin{split}
 & S_{M }(   j_{p_1,\alpha }(x),j_{p_2,\alpha }( y ) )\\
  = & \vert F_\alpha\vert^{- \frac{1}{p}}S_{M }(    ( 1 \otimes P_{F_\alpha}) \lambda_{\Gamma \rtimes \Lambda}(  s_1,t_1 ) (1 \otimes P_{F_\alpha})  , ( 1 \otimes P_{F_\alpha}) \lambda_{\Gamma \rtimes \Lambda}(  s_2,t_2 ) (1 \otimes P_{F_\alpha}) ) \\
    = &\vert F_\alpha\vert^{- \frac{1}{p}}S_{M }(    ( 1 \otimes P_{F_\alpha \cap t_1 F_\alpha}) \lambda_{\Gamma \rtimes \Lambda}(  s_1,t_1 )    ,  \lambda_{\Gamma \rtimes \Lambda}(  s_2,t_2 )  (1 \otimes P_{F_\alpha \cap t_2^{-1} F_\alpha}) ) \\
= & \vert F_\alpha\vert^{- \frac{1}{p}}    ( 1 \otimes P_{F_\alpha \cap t_1 F_\alpha}) T_M( x    ,  y ) (1 \otimes P_{F_\alpha \cap t_2^{-1} F_\alpha}).
\end{split}
\]
So that
\[
\begin{split}
 & \langle  \prod_{\alpha, \omega} S_{M }   \circ ( j_{p_1,\alpha } \times j_{p_2,\alpha})(x,y), j_{p',\alpha}(z) \rangle_{p, p'} \\
= & \lim_\alpha \vert F_\alpha\vert^{-1} (\tau \otimes {\rm Tr}) \left(   ( 1 \otimes P_{F_\alpha \cap t_1 F_\alpha}) T_M( x    ,  y ) (1 \otimes P_{F_\alpha \cap t_2^{-1} F_\alpha}) z^\ast (1 \otimes P_{F_\alpha}) \right) \\
= &\lim_\alpha \vert F_\alpha\vert^{-1} (\tau \otimes {\rm Tr}) \left(   ( 1 \otimes P_{F_\alpha \cap t_1 F_\alpha}) T_M( x    ,  y ) z^\ast (1 \otimes P_{F_\alpha \cap t_3^{-1} F_\alpha \cap (t_2 t_3)^{-1} F_\alpha})  \right) \\
= &
\left\{
\begin{array}{ll}
\lim_\alpha M((s_1, t_1), (s_2, t_2))   \frac{ \vert  F_\alpha \cap t_1 F_\alpha \cap t_3^{-1} F_\alpha \cap (t_2 t_3)^{-1} F_\alpha \vert }{\vert F_\alpha \vert} & \textrm{ if } (s_1, t_1) (s_2, t_2) (s_3, t_3)^{-1} = (e,e), \\
  0 & \textrm{ otherwise}.
\end{array}
\right.
\end{split}
\]
By the F\o{}lner condition the limit of the fraction is 1 and so this expression equals  $\langle  T_{M}(x,y), z \rangle_{p, p'}$.
Since $j_{p,\alpha}$ are contractions (see \cite{gonzalez2018crossed}), this shows (for $p = 1$ with Kaplansky's density theorem) that the norm of $T_M$ is bounded above by the norm of $S_M$, which in turn is bounded by $\Vert T_m^{\mathrm{vec}} \Vert_{\mathrm{cb}}$.
\end{proof}

\begin{theorem}\label{Thm=Heisenberg}
Let $m \in C_b( \mathbb{R} \times \mathbb{R})$. Let $M \in C_b( H \times H)$ be defined by
\[
M((x,y,z),(x',y',z'))
= m(x,x').
\]
We have for $1 \leq p, p_1, p_2 < \infty$ with $p^{-1} = p_1^{-1} + p_2^{-1}$,
\begin{equation}\label{Eqn=CBmap2}
\Vert T_M: L_{p_1}(\widehat{H}) \times L_{p_2}(\widehat{H}) \rightarrow  L_{p}(\widehat{H}) \Vert  \leq
\Vert T_m: L_{p_1}(\widehat{ \mathbb{R} }   ) \times L_{p_2}(\widehat{ \mathbb{R} }   )     \rightarrow L_{p}(\widehat{ \mathbb{R} }   )  \Vert_{\mathrm{cb}}.
\end{equation}
\end{theorem}
\begin{proof}
Let
\[
H_j = j^{-1} \mathbb{Z} \times j^{-1}  \mathbb{Z} \times j^{-2}  \mathbb{Z}, \qquad  j \in \mathbb{N},
\]
 viewed as a cocompact discrete subgroup of $H$ with $\cup_{j \in \mathbb{N}} H_j$ dense in $H$. Set $\Gamma_j =   j^{-1} \mathbb{Z} \times j^{-2} \mathbb{Z}$  and $\Lambda_j =  j^{-1}  \mathbb{Z}$. The action $\mathbb{R} \curvearrowright \mathbb{R}^2$ described above restricts to an action $\Lambda_j \curvearrowright \Gamma_j$ and $H_j = \Gamma_j \rtimes \Lambda_j$. We now have by Theorem \ref{thm:LatticeApproximation}, Proposition \ref{Prop=Transference} and Theorem \ref{Thm=MultiDeLeeuwRestriction} (see Remark \ref{Rmk=CBCaseThmC}),
  \[
 \begin{split}
 \Vert T_M: L_{p_1}(\widehat{H}) \times L_{p_2}(\widehat{H}) \rightarrow  L_{p}(\widehat{H}) \Vert
\leq & \sup_{j \in \mathbb{N}}
\Vert T_M: L_{p_1}(\widehat{H_j}) \times L_{p_2}(\widehat{H_j}) \rightarrow  L_{p}(\widehat{H_j}) \Vert \\
\leq & \sup_j \Vert T_{m|\Lambda_j\times \Lambda_j}: L_{p_1}(\widehat{\Lambda_j}) \times L_{p_2}(\widehat{\Lambda_j}) \rightarrow  L_{p}(\widehat{\Lambda_j}) \Vert_{\mathrm{cb}} \\
=& \Vert T_m: L_{p_1}(\widehat{\mathbb{R}}) \times L_{p_2}(\widehat{\mathbb{R}}) \rightarrow  L_{p}(\widehat{\mathbb{R}}) \Vert_{\mathrm{cb}}.
 \end{split}
 \]
   \end{proof}

\begin{remark}
We note that the hypotheses of Theorem \ref{Thm=Heisenberg} are satisfied with finite norms in \eqref{Eqn=CBmap2} if $m$ obeys a suitable H\"ormander-Mikhlin type condition, see \cite{DiPlinio1}, \cite{DiPlinio2}.
\end{remark}

\begin{remark}
The methods in this section in fact work for  more general semi-direct products of the form $H = G \rtimes \mathbb{R}^n$ for an action by automorphisms of $\mathbb{R}^n$ on a locally compact amenable group $G$ provided that $G$ is ADS and the approximating groups $\Gamma_j$ of $G$ are invariant under the action of $\frac{1}{j} \mathbb{Z}^n \subseteq \mathbb{R}^n$.
\end{remark}

\section{\texorpdfstring{Theorem B: Lower bounds on $\delta$ for reductive Lie groups}{Theorem B: Lower bounds on delta for reductive Lie groups}}\label{Sect=LowerBound}

In the following sections \ref{Sect=LowerBound} through \ref{Sect=KeyLemma}, we proceed with the proof of Theorem B.
Let $G$ be a real reductive Lie group with Lie algebra $\fg$, Cartan involution $\theta$, maximal compact subgroup $K$, and
invariant bilinear form $B$, cf.~\cite[\S VII.2]{Knapp1996}.
The inner product
\begin{align*}
B_\theta(x,y) := -B(x,\theta y)
\end{align*}
on $\fg$ endows $\End(\fg)$ with the operator norm $A \mapsto \|A\|$.
For $\rho \geq 1$, we denote by
\begin{equation}\label{Eqn=AdjointBall}
B_{\rho}^{G} := \big\{ g \in G\,;\, \|\Ad_{g}\| \leq \rho \big\}
\end{equation}
the preimage under the adjoint representation $\Ad \colon G \rightarrow \End(\fg)$ of the closed ball of radius $\rho$
around the origin.
The aim of this section is to prove the following lower bound on
$\delta_{B_{\rho}^{G}}$ in terms of the radius.
\begin{theorem}[Theorem B]\label{thm:DeltaForReductive}
If $d$ is the maximal dimension of a nilpotent orbit of $G$ in $\fg$, then
\begin{equation}\label{eq:DeltaForReductive}
\delta_{B_{\rho}^{G}} \geq \rho^{-d/2}.
\end{equation}
\end{theorem}

\begin{remark}\label{Rmk:ConeInterpretation}
As
the adjoint orbit $O_{X}$ of a nilpotent element $X\in \fg$ is a symplectic manifold, the coefficient $d/2$ in \eqref{eq:DeltaForReductive} is an integer.
Since $T_{X}O_{X} = \fg/\fg_{X}$ is the quotient of $\fg$ by the centralizer $\fg_{X}$ of $X$,
the maximal dimension $d$ can be expressed as
\begin{equation}\label{eq:omgekeerd}
d = \dim(\fg) - \min_{X\in \cN} \dim(\fg_X),
\end{equation}
where $\cN \subseteq \fg$ is the nilpotent cone of $\fg$.
In particular, $d=0$ if $\fg$ is compact or abelian, $d = \dim(\fg) - \rk(\fg)$ if
$\fg$ is split (or quasisplit \cite[Thm 5.1]{rothschild1972orbits}), and
$d = 2(\dim_{\C}(\fg) - \rk_{\C}(\fg))$ if $\fg$ is complex.
In particular, $d = n(n-1)$ for $\mathrm{SL}(n,\R)$ and $\mathrm{GL}(n,\R)$, and
$d = 2n(n-1)$ for $\mathrm{SL}(n,\C)$ and $\mathrm{GL}(n,\C)$.
\end{remark}

Let $\fg = \fk \oplus \fp$ be the Cartan decomposition of $\fg$, let $\fa$ be a maximal abelian subspace of $\fp$, and let $A$ be the corresponding analytic subgroup. If $g = k_1 a k_2$ is the $KAK$-decomposition of $g\in G$ (\cite[\S VII.3]{Knapp1996}), then $\|\Ad_{g}\| = \|\Ad_{a}\|$
since both $B$ and $\theta$ are invariant under $\Ad(K)$. Let 
$\fg = \fg_0  \oplus \bigoplus_{\lambda \in \Sigma}\fg_{\lambda}$ be the restricted root space decomposition,
where the sum runs over the set
$\Sigma \subseteq \fa^*$ of restricted roots.
Since $A$ is simply connected, the set $B_{\rho}^{G}$ can be equivalently described as
\begin{equation}\label{eq:GroupBallsPolygon}
B_{\rho}^{G} = K \rho^{\mathcal{P}} K := K\exp\big(\log(\rho) \mathcal{P}\big) K,
\end{equation}
where $\mathcal{P}\subseteq \fa$ is the polygon $\mathcal{P} = \{h \in \fa \,;\, \alpha(h)\leq 1 \; \forall\,  \alpha \in \Sigma\}$.
From this description, the following result easily follows.
\begin{proposition}\label{prop:BallInGInvariant}
The sets $B_{\rho}^{G}$ are invariant under inversion, and under left and right multiplication by $K$. Furthermore,
$\bigcup_{\rho > 1} B_{\rho}^{G} = G$ and $\bigcap_{\rho>1} B^{G}_{\rho} = K$.
\end{proposition}
\begin{proof}
Invariance under left and right multiplication by $K$
is clear from \eqref{eq:GroupBallsPolygon}, and invariance under inversion follows from
the fact that $\Sigma = -\Sigma$. The formula for the union is obvious, and the formula for the intersection follows
from $\bigcap_{\rho>1} B_{\rho}^{G} = B^{G}_{1}$, and the fact that $\|\Ad_{\exp(\pm h)}\| = \|\exp(\ad_{\pm h})\| = 1$
for $h\in \fa$ if and only if $h=0$.
\end{proof}

\subsection{A neighbourhood basis}
The reductive Lie algebra $\fg$ decomposes as the direct sum ${\fg = \fg_0 \oplus \fz}$ of
the maximal semisimple ideal $\fg_0 = [\fg, \fg]$ and the centre $\fz$.  The former admits the Cartan decomposition
$\fg_0 = \fk_0 \oplus \fp_0$, with $\fk_0 = \fk \cap \fg_0$ and $\fp_0 = \fp \cap \fg_0$.

Let $B_{r}^{\fz}\subseteq \fz$ and $B_{r}^{\fg_0} \subseteq \fg_0$
be the open balls of radius $r$ with respect to the inner product $B_{\theta}$ in $\fz$ and $\fg_0$, respectively.
Both are $K$-invariant, and the global Cartan decomposition $G = K\exp(\fp)$ implies that
 $B_{r}^{\fz}\subseteq \fz$ is invariant under all of $G$.

We are interested in the neighbourhood basis
\begin{equation}\label{eq:V}
V^{\fg}_{\epsilon, R, r} :=
\big(\Ad_{G}(B_{\varepsilon}^{\fg_0}) \cap B_{R}^{\fg_{0}} \big)
\times
B^{\fz}_{r} ,
\end{equation}
obtained by intersecting the $G$-invariant neighbourhoods
\[
U_{\varepsilon} :=
\Ad_{G}(B_{\varepsilon}^{\fg_0}) \times \fz
\]
with the bounded sets $B^{\fg_0}_{R} \times B^{\fz}_{r}$.
It will be convenient to write $V^{\fg}_{\epsilon, R, r} =  V^{\fg_0}_{\varepsilon, R} \times B^{\fz}_{r}$,
where $V^{\fg_{0}}_{\varepsilon, R}$ is the bounded open subset of $\fg_0$ defined by
\begin{equation}
V^{\fg_0}_{\varepsilon, R} := \Ad_{G}(B_{\varepsilon}^{\fg_0}) \cap B_{R}^{\fg_{0}}.
\end{equation}

\begin{remark}[$K$-invariance]\label{Rk:KInvariance}
Since $B_{\varepsilon}^{\fg_0}$ is $K$-invariant, the global Cartan decomposition
$G = K\exp(\fp)$
yields
$\Ad_{G}(B_{\varepsilon}^{\fg_0}) = \Ad_{\exp(\fp)}(B_{\varepsilon}^{\fg_0})$.
Further, since $\fz$ acts trivially on $\fg_0$, we have
$\Ad_{G}(B_{\varepsilon}^{\fg_0}) = \Ad_{\exp(\fp_0)}(B_{\varepsilon}^{\fg_0})$ for $\fp_0 = \fp \cap \fg_0$.
It follows that $V^{\fg}_{\epsilon, R, r} = V^{\fg_0}_{\varepsilon, R} \times B^{\fz}_{r}$ is the product of
the $K$-invariant set $V^{\fg_0}_{\varepsilon, R} \subseteq \fg_0$ that depends only on the restriction of $B_{\theta}$ to $\fg_0$, and the $G$-invariant
set $B^{\fz}_{r} \subseteq \fz$ that depends only on the restriction of $B_{\theta}$ to $\fz$.
In particular the sets $V^{\fg}_{\epsilon, R, r}\subseteq \fg$ are $K$-invariant, and they
depend only on the triple $(\fg, \theta, B)$, not on the Lie group $G$.
\end{remark}

The lower bound \eqref{eq:DeltaForReductive} will be established by calculating
\begin{equation}
\delta^{0}_{B_{\rho}^{G}} := \limsup_{(\epsilon, R, r)\rightarrow 0}
\frac
{
	\mu\Big( \bigcap_{g\in B_{\rho}^{G}} \Ad_{g^{-1}}\exp(V^{\fg}_{\epsilon, R, r})\Big)
}
{
	\mu(\exp(V^{\fg}_{\epsilon, R, r}) )
},
\end{equation}
where $\mu$ is a Haar measure on $G$.
Since $V^{\fg}_{\epsilon, R, r} = - V^{\fg}_{\epsilon, R, r}$,
the sets $\exp(V^{\fg}_{\epsilon, R, r})$ constitute a symmetric neighbourhood basis of the identity. It follows that
$\delta_{B^{G}_{\rho}} \geq \delta^{0}_{B^{G}_{\rho}}$, so in order
to establish~\eqref{eq:DeltaForReductive}, it suffices to prove that~$\delta^{0}_{B^{G}_{\rho}} \geq \rho^{-d/2}$.

%
%

\subsection{Relation between Haar measure and Lebesgue measure.}
The first step is to reformulate this in terms of the Lebesgue measure on the Lie algebra $\fg$.

Let $\Vol_{G}$ be a left invariant volume form on $G$, so that integrating against $\Vol_{G}$ corresponds to a left Haar measure
$\mu$ on $G$.
Let $\Vol_{\fg}$ be a constant volume form on~$\fg$, corresponding to a Lebesgue measure $\Lebesgue$ on $\fg$.
We normalize these volume forms in such a way that $\Vol_{\fg}$ agrees with $\exp^*\Vol_{G}$ at the origin in~$\fg$.
Then $\exp^*\Vol_{G} = \nu \Vol_{\fg}$, where the density $\nu$ of $\exp^*\Vol_{G}$ with respect to the Lebesgue measure
satisfies $\nu(0)=1$. We show that $\nu$
can be chosen arbitrarily close to $1$ on $U_{\varepsilon} \subseteq \fg$ for small $\varepsilon$.
\begin{proposition}\label{Prop1}
The density is given by $\nu(x) = \Det(\Phi_{x})$, where
${\Phi \colon \fg \rightarrow \mathrm{End}(\fg)}$ is the left logarithmic derivative of the exponential map.
The function $\nu \colon \fg \rightarrow \R$ is smooth, $G$-invariant, and equal to $1$ on $\fz \subseteq \fg$.
Moreover, there exists a constant $c_{\fg_0}>0$ (depending only on $\fg_0$) such that
for $\varepsilon$ sufficiently small,
$\|\nu - 1\|_{\infty} \leq c_{\fg_{0}}\, \varepsilon$ uniformly on $U_{\varepsilon}$.
\end{proposition}
\begin{proof}
Let $\Phi \colon \fg \rightarrow \mathrm{End}(\fg)$ be the left logarithmic derivative of the exponential map,
defined by $\Phi_{x}(y) := (D_{\exp(x)}L_{\exp(x)^{-1}}) \circ (D_{x}\exp)(y)$, where $L_{g} \colon G \rightarrow G$
denotes left multiplication by $g$.
Then for all $x, y_1, \ldots, y_n \in \fg$,
\begin{eqnarray*}
(\exp^*\Vol_{G})_{x}(y_1, \ldots, y_n)
&=&
(\Vol_{G})_{\exp(x)}(D_{x}\exp(y_1), \ldots, D_{x}\exp(y_n) )\\
&=&
(\Vol_{G})_{\exp(x)}(D_{\one}L_{\exp(x)} \Phi_{x}(y_1), \ldots, D_{\one}L_{\exp(x)} \Phi_{x}(y_n))\\
&=&
(\Vol_{G})_{\one}(\Phi_{x}(y_1), \ldots, \Phi_{x}(y_n))\\
&=&
\mathrm{det}(\Phi_{x}) (\Vol_{\fg})_{0}(y_1, \ldots, y_n),
\end{eqnarray*}
where the last two steps use that $\Vol_{G}$ is left invariant, and
that $\exp^*\Vol_{G}$ agrees with $\Vol_{\fg}$ at the origin.
It follows that $\nu(x) = \mathrm{det}(\Phi_{x})$. Since $\exp \colon \fg \rightarrow G$ is equivariant
with respect to the adjoint action on $\fg$ and the conjugate action on $G$, its logarithmic derivative
satisfies $\Phi_{\Ad_{g}(x)} = \Ad_{g} \circ \,\Phi_{x} \circ \Ad_{g^{-1}}$.
In particular, $\nu(x) = \Det(\Phi_x)$ is invariant under the adjoint action.

In fact, the left logarithmic derivative $\Phi_{x}\colon \fg \rightarrow \fg$ of
the exponential map is explicitly given (cf. \cite[Thm.~2.1.4]{Faraut2008}) by the convergent power series
\begin{equation}\label{PowerSeriesDLogExp}
\Phi_{x} = \frac{\mathrm{Id} - \exp(-\mathrm{ad}_{x})}{\mathrm{ad}_{x}} =  \mathrm{Id} - \frac{1}{2!}\ad_{x} + \frac{1}{3!}(\ad_{x})^2 + \ldots
\end{equation}
Note that the determinant of the real linear map $\Phi_{x} \colon \fg \rightarrow \fg$
is equal to the determinant over $\C$ of the complexification $\Phi^{\C}_{x} \colon \fg_{\C} \rightarrow \fg_{\C}$.
In terms of the Jordan--Chevalley decomposition
$\ad_{x} = \ad_{x_s} + \ad_{x_{n}}$ into a semisimple and a nilpotent element of $\fg_{\C}$, we thus find
\[\nu(x) = \left|\Det_{\C}\left(\frac{\mathrm{Id} - \exp(-\mathrm{ad}_{x_{s}})}{\mathrm{ad}_{x_{s}}}\right)\right|
= \prod_{i=1}^{\mathrm{dim{\fg}}}\left|\frac{1-e^{-\mu_{i}}}{\mu_i}\right|,
\]
where $\mu_i$ are the eigenvalues of $\ad_{x}$ as a complex linear transformation of $\fg_{\C}$.

In particular, $\nu(x)$ depends only on the second factor in $\fg = \fz \oplus \fg_0$, and we can write
$\nu(x) = \nu_0(\ad_{x})$ for a smooth function $\nu_0 \colon \fg_0 \rightarrow \R$ with $\nu_0(0)=1$.
It follows that for
any $c_{\fg_0} > \|\nabla\nu_0(0)\|$, there exists an $\varepsilon_0>0$ such that $\|\nu_0 - 1\| \leq c_{\fg_0}\varepsilon$
uniformly on $B^{\fg_0}_{\varepsilon}$ for all $\varepsilon < \varepsilon_0$.
The uniform estimate for $\nu$ on $U_{\varepsilon}$ now follows from $G$-invariance.
\end{proof}

For $r$ and $R$ smaller than the injectivity radius of the exponential ${\exp \colon \fg \rightarrow G}$, we can therefore relate the Haar measure of $\exp(V^{\fg}_{\varepsilon, R, r})$ to the Lebesgue measure
of $V^{\fg}_{\varepsilon, R, r}$,
\begin{equation}\textstyle
(1-c_{\fg_{0}}\,\varepsilon) \Lebesgue(V^{\fg}_{\varepsilon, R, r})  < \mu\big(\exp(V^{\fg}_{\varepsilon, R, r})\big) < (1+c_{\fg_{0}}\,\varepsilon) \Lebesgue(V^{\fg}_{\varepsilon, R, r}).
\end{equation}
Since $\exp$ is equivariant under the adjoint action on $\fg$ and $G$, this allows us to express $\delta^{0}_{B^{G}_{\rho}}$ in terms of the Lebesgue measure on $\fg$,
\begin{equation}\label{eq:DeltaOnLieAlgebra}
\delta^{0}_{B^{G}_{\rho}} =
\limsup_{(\varepsilon, R, r) \rightarrow 0}
\frac
{
	\Lebesgue\Big( \bigcap_{g\in B_{\rho}^{G}} \Ad_{g^{-1}}(V^{\fg}_{\varepsilon, R, r})\Big)
}
{
	\Lebesgue(V^{\fg}_{\varepsilon, R, r})
}.
\end{equation}
\begin{proposition}\label{prop:reductiontosimple}
The number $\delta^{0}_{B^{G}_{\rho}}$ depends on $G$ only through the maximal semisimple ideal $\fg_0 = [\fg, \fg]$,
and the restriction of $B_{\theta}$ to $\fg_0$.
\end{proposition}
\begin{proof}
By Remark~\ref{Rk:KInvariance}
the sets $V^{\fg}_{\varepsilon, R, r}$ are $K$-invariant.
If $g = k\exp(p)$ is the global Cartan decomposition of $g\in G = K\exp(\fp)$,
we thus have
\[
\Ad_{g^{-1}}(V^{\fg}_{\varepsilon, R, r}) =  \Ad_{\exp(-p)}(V^{\fg}_{\varepsilon, R, r}).
\]
Further, since $\fz$ acts trivially on $\fg$, the decomposition
$p = p_{0} +  p_{\fz}$ of $p\in \fp$ with respect to $\fp = \fp_0 \oplus (\fz\cap \fp)$ yields
\[
\Ad_{g^{-1}}(V^{\fg}_{\varepsilon, R, r})  = \Ad_{\exp(-p_0)}(V^{\fg_0}_{\varepsilon, R}) \times B^{\fz}_{r}.
\]
Since $B^{G}_{\rho}$ is $K$-invariant (Proposition~\ref{prop:BallInGInvariant}), we have $k\exp(p) \in B^{G}_{\rho}$
if and only if $\exp(p) \in B^{G}_{\rho}$, which is the case if and only if $\exp(p_0) \in B^{G}_{\rho}$.
But since $\Ad_{\exp(p_0)}$ acts by the identity on second factor of $\fg = \fg_0 \oplus \fz$,
we have $\|\Ad^{\fg}_{\exp(p_0)}\| \leq \rho$ for the adjoint action on $\fg$ if and only
$\|\Ad^{\fg_0}_{\exp(p_0)}\| \leq \rho$ for the adjoint action on $\fg_0$.
If $G_0$ denotes the connected adjoint group of the semisimple Lie algebra $\fg_0$, we thus have
\begin{eqnarray*}
\frac
{
	\Lebesgue^{\fg}\Big( \bigcap_{g\in B_{\rho}^{G}} \Ad_{g^{-1}}(V^{\fg}_{\varepsilon, R, r})\Big)
}
{
	\Lebesgue^{\fg}(V^{\fg}_{\varepsilon, R, r})
}
&=&
\frac
{
	\Lebesgue^{\fg}\Big( \bigcap_{g_0\in B_{\rho}^{G_0}} \Ad_{g_0^{-1}}(V^{\fg_0}_{\varepsilon, R}) \times B^{\fz}_{r} \Big)
}
{
	\Lebesgue^{\fg}(V^{\fg_0}_{\varepsilon, R} \times B^{\fz}_{r} )
}\\
&=&
\frac
{
	\Lebesgue^{\fg_0}\Big(\bigcap_{g_0\in B_{\rho}^{G_0}} \Ad_{g_0^{-1}}(V^{\fg_0}_{\varepsilon, R})\Big)
}
{
	\Lebesgue^{\fg_0}(V^{\fg_0}_{\varepsilon, R})
},
\end{eqnarray*}
where $\Lebesgue^{\fg}$ and $\Lebesgue^{\fg_0}$ denote the Lebesgue measure on $\fg$ and $\fg_0$, respectively.
\end{proof}
\subsection{Proof of Theorem~\ref{thm:DeltaForReductive}}

By Proposition~\ref{prop:reductiontosimple}, it suffices to prove Theorem~\ref{thm:DeltaForReductive}
for the case where $G$ is the adjoint group of the semisimple Lie algebra $\fg_0$.
The proof hinges on the following lemma.
\begin{lemma}[Key Lemma]\label{KeyLemma}
Let $G$ be a connected, real reductive Lie group with semisimple Lie algebra $\fg$. Then
for all $R>0$ and all $\rho >1$, we have
\begin{align*}
\lim_{\epsilon \to 0} \frac{\Lebesgue(V^{\fg}_{\epsilon, \rho R})}{\Lebesgue(V^{\fg}_{\epsilon,R}) } = \rho^{d/2}.
\end{align*}
\end{lemma}
The proof of this lemma requires a rather detailed discussion of limits of orbital integrals, and will be deferred to Section~\ref{sec:MainTheorem}.
Assuming Lemma~\ref{KeyLemma}, the proof of Theorem~\ref{thm:DeltaForReductive} is quite straightforward.
\begin{proof}[Proof of Theorem~\ref{thm:DeltaForReductive}, assuming Lemma~\ref{KeyLemma}]
In view of \eqref{eq:omgekeerd},
the maximal dimension $d$ of the nilpotent orbits is the same in $\fg$ and $\fg_0$.
By Proposition~\ref{prop:reductiontosimple},
we may therefore assume without loss of generality that $G$ is a connected, real reductive Lie group with semisimple Lie algebra $\fg$.

Since ${\|\Ad_{g}\| \leq \rho}$, we have
$\Ad_{g^{-1}}(B^{\fg}_{R}) \supseteq B^{\fg}_{R/\rho}$.
Since $\Ad_{G}(B^{\fg}_{\varepsilon})$ is $\Ad_{G}$-invariant, we find for
$V_{\varepsilon, R} = \Ad_{G}(B^{\fg}_{\varepsilon}) \cap B^{\fg}_{R}$ that
\[\Ad_{g^{-1}}(V_{\varepsilon,R}) = \Ad_{G}(B^{\fg}_{\varepsilon}) \cap \Ad_{g^{-1}}B^{\fg}_{R} \supseteq V_{\varepsilon,R/\rho}.
\]
From \eqref{eq:DeltaOnLieAlgebra} (without $r$ because $\fz = \{0\}$) we thus find
\begin{align*}
\delta^0_{B^{G}_{\rho}}
=
\limsup_{\varepsilon,R\rightarrow 0}
\frac
{
	\Lebesgue\Big( \bigcap_{g\in B_{\rho}^{G}} \Ad_{g^{-1}}(V^{\fg}_{\varepsilon,R})\Big)
}
{
	\Lebesgue(V^{\fg}_{\varepsilon, R})
}
&\geq
\limsup_{\varepsilon,R\rightarrow 0}  \frac{\Lebesgue(V^{\fg}_{\varepsilon,R/\rho})}{\Lebesgue(V^{\fg}_{\varepsilon,R})}
\\&
\geq
\lim_{R \to 0} \lim_{\epsilon \to 0}
\frac{\Lebesgue(V^{\fg}_{\varepsilon,R/\rho})}{\Lebesgue(V^{\fg}_{\varepsilon,R})}
=
\rho^{-d/2}.
\end{align*}
\end{proof}

The strategy to prove Lemma~\ref{KeyLemma} is as follows.
Since the closure of an adjoint orbit $O_{x}$ through $x\in \fg$ contains $0$ if and only if
$x$ is nilpotent, the set $\bigcap_{\varepsilon > 0} V_{\varepsilon, R}$ is the intersection of the nilpotent cone $\cN$
with the unit ball $B^{\fg}_{R}(0)$. The union of the nilpotent orbits $O_{X}$ of maximal dimension is a dense open subset of
the nilpotent cone.
Using results of Harish-Chandra and Barbasch--Vogan on limiting orbit integrals,
we will show that as $\varepsilon$ approaches~$0$, the volume of $V_{\varepsilon, R}$ with respect to the Lebesgue measure on $\fg$ scales with $R$ in the same way as the Liouville volume of the symplectic manifold $O_{X} \cap B^{\fg}_{R}(0)$.
Since the Kostant-Kirillov-Souriau symplectic form $\omega^{\KKS}_{O_{X}}$ on the cone $O_{X}$ scales as $R$ under dilation,
the corresponding Liouville volume form $\Vol^{\KKS}_{O_X}$
scales as $R^{d/2}$, yielding the factor $\rho^{d/2}$ in Lemma~\ref{KeyLemma}.

\section{Limits of orbital measures}\label{Sect=Measures}
\label{Sec:LimitsOfOrbits}
In the remainder of this section, we focus on the proof of Lemma~\ref{KeyLemma}.
From now on, we assume that the Lie algebra~$\frakg$ is semisimple, and that the invariant bilinear form~$B$ is the Killing form~$\kappa$.
We will generally denote generic elements of~$\frakg$ by~$x,y,z$, elements of a Cartan subalgebra~$\frakh \subset \frakg$ by~$h$, and elements of the nilpotent cone~$\mathcal{N} \subset \frakg$ by~$X,Y,Z$.
For a subset $A\subseteq \fg$, we denote the centralizer and the normalizer in $\fg$ by
$Z_\frakg(A)$ and $N_\frakg(A)$ respectively.
On the group level, we similarly define
\[
Z_G(A) := \{g \in G : \Ad_g y = y \quad \forall y \in A\}, \quad
N_G(A) := \{g \in G : \Ad_g A \subset A \}.
\]

\subsection{Regular elements of Lie algebras}
We recall the notion of \emph{regularity} in~$\frakg$. For~$x \in \frakg$ consider the characteristic polynomial
\begin{align*}
\det(\ad x  - t) =: \sum_{k \geq 0} a_k(x) t^k, \quad t \in \mathbb{R}.
\end{align*}
Let~$k(x)$ be the minimal index so that~$a_{k(x)}(x) \neq 0$. If~$k(x) = \min_{y \in \frakg} k(y)$, we say~$x$ is \emph{regular}. If~$A \subset \frakg$ is any subset of~$\fg$, we define~$A_\reg$ to be the set of regular elements in~$A$.
If~$\frakh$ is a Cartan subalgebra (CSA) and~$C$ is a connected component of~$\frakh_\reg$, we call~$C$ an \emph{open Weyl chamber} of~$\frakh$.
\begin{remark}
Our notion of regularity is different from a common definition where~$x$ is regular if its centralizer has minimal dimension among the centralizers of all~$x' \in \frakg$. Our current definition is used, for example, in \cite{bourbaki2008lie}.
\end{remark}
We will use the following standard properties of regular elements:
\begin{lemma}
Let~$\frakg$ be a real Lie algebra.
\label{lem:PropertiesRegularElements}
\begin{itemize}
\item[i)] Regular elements in~$\frakg$ lie in a unique CSA given by their centralizer.
\item[ii)] If a single element in an adjoint orbit~$O_x \subset \frakg$ is regular, then all elements are.
\item[iii)] The set of regular elements is dense and open in~$\frakg$, and its complement in~$\frakg$ is a Lebesgue null set.
\end{itemize}
\end{lemma}
\begin{proof}
\textit{i)} The semisimplicity statement is \cite[Chapter VIII.4, Cor 2]{bourbaki2008lie}, and the statement about the CSA is  \cite[Thm VII.3.1]{bourbaki2008lie}.
\\
\textit{ii)} This follows since the characteristic polynomial is invariant under the adjoint action.
\\
\textit{iii)} From \cite[Chapter VIII.2]{bourbaki2008lie} we know that the set of regular elements is Zariski-open, which implies that it is dense and open in the standard topology.
The non-regular elements, as a complement of a Zariski-open set, are intersections of closed submanifolds of lower dimension, hence they constitute a Lebesgue null set by Sard's Theorem.
\end{proof}
\subsection{Measures and Distributions on Orbits}
Let~$x\in \fg$, and let~$O_x$ be the adjoint orbit through~$x\in \frakg$.
Since~$O_x$ can be identified with a coadjoint orbit via the invariant bilinear form, it comes equipped with a canonical symplectic form. The \emph{Kostant-Kirillov-Souriau (KKS) form}~$\omega^{\KKS}_{O_x} \in \Omega^2(O_{x})$ is given by
\begin{align}\label{eq:defKKS}
\left( \omega^{\KKS}_{O_x} \right)_{x'}(\ad_{x'}y, \ad_{x'}z) = \kappa(x', [y,z]), \quad \forall x' \in O_{x},\; y,z \in \frakg.
\end{align}
This induces a volume form called the Liouville form
\begin{align}\label{eq:defKKSVolume}
\Vol_{O_x} := \frac{1}{k!} (\wedge^k \omega^{\KKS}_{O_x}),
\end{align}
where~$k = \dim O_x /2$.
By \cite[Thm. 2]{rao1972orbital} the assignment of Borel sets~$A \subset \frakg$ to
\begin{align*}
\mu_{O_x}(A) := \int_{O_x \cap A} \Vol_{O_x}
\end{align*}
defines a Radon measure~$\mu_{O_x}$ on~$\frakg$. In particular, it is finite on compact subsets of~$\frakg$.
On nilpotent orbits, these measures are homogeneous:
\begin{lemma}
\label{lem:HomogeneityOfNilpotentVolumes}
Let~$A \subset \frakg$ be a Borel set, let~$X \in \frakg$ a nilpotent element, and let~$k = \dim O_X /2$. Then for all $\rho>0$, we have
\begin{align*}
\mu_{O_{X}}(\rho \cdot A) =
\rho^{k}
\mu_{O_{X}}(A).
\end{align*}
\end{lemma}
\begin{proof}
By the Jacobson-Morozov theorem, the nilpotent orbit~$O_X$ is a \emph{cone}, i.e.~$\rho \cdot O_X = O_X$ for all~$\rho>0$. Thus we have
\begin{align*}
(\rho \cdot A) \cap O_X = \rho \cdot (A \cap O_X).
\end{align*}
Denote by~$m_\rho : O_X \to O_X$ the multiplication by~$\rho$. By definition of the KKS form we have for all~$X' \in O_X$ and~$y,z \in \frakg$:
\begin{align*}
(m_\rho^* \omega^{\KKS}_{O_X})_{X'} (\ad_{X'} y, \ad_{X'} z)
=
(\omega^{\KKS}_{O_X} )_{\rho X'} (\ad_{\rho  X'} y, \ad_{\rho  X'} z)
=
\kappa(\rho X', [y,z])
=
\rho \cdot \kappa( X', [y,z]),
\end{align*}
hence
\begin{align*}
m_\rho^* \omega^{\KKS}_{O_X}
=
\rho \cdot \omega^{\KKS}_{O_X}.
\end{align*}
From \eqref{eq:defKKSVolume}
we then find
\begin{align*}
m_\rho^* \Vol_{O_X}
=
\rho^{k} \cdot \Vol_{O_X},
\end{align*}
so that
\begin{align*}
\mu_{O_X} (\rho \cdot A) =
\int_{m_\rho (A \cap O_X)}
\Vol_{O_X}
=
\int_{A \cap O_X}
m_\rho^* \Vol_{O_X}
= \rho^{k} \mu_{O_X}(A)
\end{align*}
as required.
\end{proof}
The measure~$\mu_{O_x}$ on the adjoint orbit~$O_{x}$ through~$x\in \fg$ yields the
distribution~$\mathcal{D}_{O_x}$ on~$\fg$ defined by
\begin{align*}
\mathcal{D}_{O_x} : C_c^\infty(\frakg) \to \mathbb{R}, \quad  \mathcal{D}_{O_x} (f) := \frac{1}{(2 \pi)^{2k}} \int_{O_x} f|_{O_x} \Vol_{O_x},
\end{align*}
again with~$k := \mathrm{dim}(O_x)/2$.
The~$2 \pi$-normalization factor ensures that the orbital distributions~$\mathcal{D}_{O_x}$ coincide with the ones in \cite{harris2012tempered}, where this normalization occurs in the volume form~$\Vol_{O_x}$.
%
Let~$\frakh \subset \frakg$ be any~$\theta$-invariant Cartan subalgebra, and~$H := Z_G(\frakh)$ the associated Cartan subgroup.
Since~$G$ and~$H$ are unimodular, we can fix a~$G$-invariant volume form~$\Vol_{G/H}$ on the quotient~$G/H$. Note that~$\Vol_{G/H}$ is unique up to a nonzero scalar.
Let~$h \in \frakg$ be any element with centralizer~$H$. Then the orbit map~$g \mapsto \Ad_{g}(h)$ descends to a diffeomorphism~$\iota \colon G/H \xrightarrow{\sim} O_{h}$, and the pullback~$\iota^*\Vol_{O_{h}}$ of the KKS volume form on~$O_{h}$ defines yet another invariant volume form on~$G/H$.
The two invariant volume forms agree up to a scalar which depends only on~$h$, 
\begin{equation}
\label{eq:TransformationOfVolumeFormOnG/H}
    \iota^*\Vol_{O_{h}} = \pi(h) \Vol_{G/H},
\end{equation}
yielding a function~$\pi : \frakh \to \mathbb{R}$.
\begin{proposition}
\label{prop:RadonNikodymDifferentOrbits}
In the above setting, there is some~$c > 0$ depending only on the choice of~$\Vol_{G/H}$, so that for all~$h \in \frakh$,
\[
|\pi(h)| = c \cdot
\prod_{\alpha \in \Delta_{+}}|\alpha(h)|.
\]
Here,~$\Delta^+ \subset \Delta(\frakg_\C, \frakh_\C)$ is a choice of positive roots for the complexified Lie algebra~$\frakg_\C$ with respect to the CSA~$\frakh_\C$.
\end{proposition}
\begin{proof}
It suffices to consider the volume forms~$\iota^* \Vol_{O_h}$ and~$\Vol_{G/H}$ at a single point of~$G/H$, hence we restrict to~$T_{[e]} (G/H) \cong \frakg/\frakh$.
We identify~$\frakg/\frakh \cong \frakh^\perp$, where~$\frakh^\perp$ denotes the orthogonal complement of~$\frakh \subset \frakg$ with respect to the inner product~$\kappa_\theta$.
There is some scalar~$ \tilde{c} \neq 0$ so that~$\tilde{c} \cdot (\Vol_{G/H})_{[e]} = \Vol_{\frakh^\perp}$, where~$\Vol_{\frakh^\perp}$ is the volume form on~$\frakh^\perp$ associated to the inner product~$\kappa_\theta$ and some choice of orientation on~$\frakh^\perp$.
Consider the map~$\theta \circ \ad_{h} \colon \fg \rightarrow \fg$ for~$h \in \frakh$.
Since it preserves~$\fh$ and
is skew-symmetric with respect to the inner product~$\kappa_\theta$,
it restricts to a skew-symmetric endomorphism of~$\fh^{\perp}$.
%
The pullback of the~$\KKS$ form at~$[e]\in G/H$
is given, for~$x, y \in \frakh^\perp$, by
\begin{align*}
\iota^*\omega^{\KKS}_{O_{h}}(x,y) =
 \kappa(h,[x,y]) =
 \kappa_{\theta}(x,\theta \circ \ad_{h}(y)).
\end{align*}
Recall that if~$(V,B)$ is an oriented inner product space of even dimension~$2k$, then the Pfaffian of a skew-symmetric linear map
$A \colon V \rightarrow V$ is defined by
\begin{align*}
\frac{1}{k!} \left(\wedge^k \omega_{A}\right)
=
\mathrm{Pf}(A)\Vol,
\end{align*}
where~$\Vol \in \wedge^{2k}V^*$ is the volume form associated to the inner product~$B$ on the oriented vector space~$V$, and~$\omega_{A}(v,w) = B(v, Aw)$ is the 2-form associated to~$A$ with respect to the inner product~$B$.
With~$V = \fh^{\perp}$ and~$A = \theta \circ \ad_{h}$, this yields
\begin{align*}
\iota^* \Vol_{O_h}
=
\frac{1}{k!}\wedge^k(\iota^*\omega^{\KKS}_{O_{h}})_{[e]}
=
\mathrm{Pf}
\left( \left(\theta \circ \ad_h \right) \att_{\frakh^\perp} \right) \Vol_{\frakh^\perp}.
\end{align*}
Recall that the Pfaffian is related to the determinant by
\begin{align*}
\mathrm{Pf}
\left( \left(\theta \circ \ad_h \right) \att_{\frakh^\perp} \right)^2 =
\det
\left(\left(\theta \circ \ad_h \right) \att_{\frakh^\perp} \right)
=
\pm
\det \left(\ad_h \att_{\frakh^\perp} \right).
\end{align*}
Since the determinant is the product of the eigenvalues over~$\C$, we can determine~$|\pi(h)|$ from the eigenvalues of~$\ad_{h}$
on the complexification~$(\frakh^\perp)_\C \cong \frakg_\C /\frakh_\C$. In view of the root space decomposition~$\frakg_\C /\frakh_\C \cong \bigoplus_{\alpha\in \Delta_{+}}\frakg_{\alpha} \oplus \frakg_{-\alpha}$, the eigenvalues of the complex linear map~$\ad_{h}$ are~$\pm\alpha(h)$. This proves the statement with~$c = |\tilde{c}|$.
 \end{proof}
\subsection{Slodowy Slices and Pointwise Orbital Limits}
Let~$\mathcal{N} \subset \frakg$ be the nilpotent cone. For any nonzero $X\in \cN$, there exists an~$\shl_2$-triple
$\{X,Y,H\}$ containing~$X$ as the nilpositive element by the Jacobson-Morosov Theorem. We denote by
\begin{align*}
S_X := X + Z_\frakg(Y)
\end{align*}
the corresponding \emph{Slodowy slice} through~$X$,
cf.~\cite[Section 7.4]{slodowysimple}.
It is transversal to the orbit~$O_X$ due to the decomposition
\begin{align}
\label{eq:DecompositionIntoSlodowy}
\frakg = \ad_X \frakg \oplus Z_\frakg(Y),
\end{align}
cf.~\cite[Chapter VIII.2]{bourbaki2008lie}. It is indeed transversal to \emph{all} orbits~$O_x$ with~$x \in S_X$, and in particular, the set~$G \cdot S_X$ is an open neighbourhood of the orbit~$O_X$, cf.~\cite[Section 7.4]{slodowysimple}.
We recall from \cite[Chapter 2]{harris2012tempered} the construction of a canonical measure~$m_{x, X}$ on the intersection~$S_X \cap O_x$:
we can consider the composition
\begin{align*}
\ad_x \frakg \hookrightarrow \frakg \to \ad_X \frakg ,
\end{align*}
where the first map is the natural embedding and the second map the projection of the direct sum \eqref{eq:DecompositionIntoSlodowy} onto the first direct summand. Since the Slodowy slice intersects~$O_x$ transversally, the composition of these two maps is surjective.
Using~$T_x O_x \cong \ad_x \frakg$ and~$T_X O_X \cong \Ad_X \frakg$, this surjective map induces the following exact sequence:
\begin{align*}
0 \to T_x(O_x \cap S_X) \to T_x O_x \to T_X O_X \to 0.
\end{align*}
We obtain a canonical volume form on~$O_x \cap S_X$ as the quotient of the KKS volume forms on~$O_x$ and~$O_X$, which in turn gives rise to the measure~$m_{x, X}$.
In \cite[Chapter 2]{harris2012tempered}, the following limit of orbits is defined for all~$x \in \frakg$:
\begin{align}
\label{eq:DefinitionOfNh}
\mathcal{N}_x := \mathcal{N} \cap \overline{ \bigcup_{\epsilon > 0} O_{ \epsilon x} }.
\end{align}
One may think of this set as the limit of the orbits~$O_{\epsilon x}$ as~$\epsilon$ approches zero, hence as the orbits approach the nilpotent cone. Let us first show that every nilpotent orbit arises, in this sense, as a limit of regular orbits:
\begin{lemma}
\label{lem:AllNilpotentOrbitsShowUpSomewhere}
Every nilpotent orbit~$O_X$ lies in the set~$\mathcal{N}_x$ for some regular~$x \in \frakg$.
\end{lemma}
\begin{proof}
The idea of this proof is essentially due to \cite[p.48]{barbasch1980local}.
If~$X$ is nilpotent, choose an~$\shl_2$-triple~$\{X,Y,H\}$ with~$X$ as the nilpositive and~$H$ as the semisimple element.
Consider the associated Slodowy slice~$S_X = X + Z_\frakg(Y)$. Since~$G \cdot S_X$ is an open neighbourhood of~$O_X$,
and since the set of regular elements is dense in~$\fg$, there exists a regular element
$x \in G \cdot S_X$. In other words, there exist~$g \in G$ and~$V \in Z_\frakg(Y)$ with
\begin{align*}
\Ad_g x = X + V.
\end{align*}
The centralizer~$Z_\frakg(Y)$ is stable under~$\ad_H$, hence we can decompose it into the eigenspaces of~$\ad_H$ on this space. Due to the structure theory of finite-dimensional~$\shl_2$-modules (cf. \cite[Sec VIII.2]{bourbaki2008lie}), the eigenvalues~$\lambda$ on these eigenspaces are all nonpositive:
\begin{align*}
Z_\frakg(Y) =
 \bigoplus_{\lambda \leq 0} (Z_\frakg(Y))_\lambda,
 \quad
V =
\sum_{\lambda \leq 0} V_\lambda.
\end{align*}
Consider then the element~$g_t := \exp \left( -\tfrac{1}{2}  \log(t) H \right) \in G$. Then we have
\begin{align*}
\Ad_{g_t g} t x
&=
t \Ad_{g_t} X + t \Ad_{g_t} V
\\
&=
t \exp(- \log(t)) X
+
\sum_{\lambda \leq 0}
t \exp \left(-\frac{\lambda}{2} \log(t) \right) V_\lambda
\\
&=
X + \sum_{\lambda \leq 0} t^{1 - \lambda/2} V_\lambda
\stackrel{t \to 0}{\to} X.
\end{align*} But this means that ~$X \in \mathcal{N} \cap \overline{\bigcup_{t > 0} \Ad_G (tx)} = \mathcal{N}_x$.
\end{proof}
\begin{remark}
\label{rmk:FresseStudiesTheSameThing}
A closer inspection of such orbital limits is given in \cite{fresse2021approximation}. Their definition of~$\mathcal{N}_x$ coincides with the one given here by \cite[Remark 3.6]{fresse2021approximation}.
\end{remark}
We will need the following asymptotic expression for the orbital integrals. It is proven in \cite[Cor 2.3]{harris2012tempered},
relying on \cite[Theorem 3.2]{barbasch1980local}.
\begin{definition}\label{def:definitionofm}
For~$x \in \frakg$, we define
\begin{align}
\label{eq:NumberM-DimensionOfOrbit}
m(x) := \min_{O_X \subset \mathcal{N}_x} {\textstyle\frac{1}{2}} \left( \dim O_x -  \dim O_X \right),
\end{align}
where the minimum is taken over all adjoint orbits~$O_{X}$ contained in~$\mathcal{N}_x$.
\end{definition}
\begin{theorem}
\label{thm:HarrisLimitOfOrbits}
Let~$\frakg$ be a real, reductive Lie algebra, and let~$x \in \frakg$.
Then for all~$f \in C_c^\infty(\frakg)$ we have
\begin{align}
\label{eq:HarrisLimitFormula}
\lim_{\epsilon \to 0}
\epsilon^{-m(x)}
\mathcal{D}_{O_{\epsilon x}}(f)
=
\sum_{O_X}
\Vol(S_X \cap O_h)
\mathcal{D}_{O_{X}}(f),
\end{align}
where the sum is taken over the set of nilpotent orbits~$O_X \subset \mathcal{N}_x$ which are of maximal dimension among all orbits contained in~$\mathcal{N}_x$, and the volume of~$S_X \cap O_x$ is calculated with respect to the measure~$m_{x,X}$.
\end{theorem}
\begin{remark}
\label{rmk:NumberMIsConstantAlongConnectedComponent}
If~$C$ is an open Weyl chamber of some CSA~$\frakh$, one actually has~$\mathcal{N}_h = \mathcal{N}_{h'}$ for all~$h,h' \in C$ by \cite[Cor 2.4]{harris2012tempered} (originally attributed to \cite{barbasch1980local}).
In particular, the number~$m$ in Theorem~\ref{thm:HarrisLimitOfOrbits} is equal for all~$h$ in one such open Weyl chamber. In this case we also write~$m(C)$ for the~$m$ associated to any~$h \in C$.
\end{remark}
\subsection{Uniform Orbital Limits}
Using results of Harish-Chandra and Varadarajan, we will show that the convergence in Theorem~\ref{thm:HarrisLimitOfOrbits}
is uniform on certain subsets of~$\frakh$.
For an open subset~$C\subseteq V$ of a vector space~$V$, we denote by~$C^{k}(\overline{C}, \R)$ the space
of functions~$u \colon C \rightarrow \R$ that are~$k$ times continuously differentiable on~$C$, and whose derivatives of order at most~$k$ extend continuously to the closure.
The following result follows from \cite[Thm 3]{harish1957fourier} (see also \cite[Thm I.3.23]{varadarajan2006harmonic}).
\begin{theorem}
\label{thm:VaradarajanContinuouslyExtensible}
Let~$\frakh \subset \frakg$ be a Cartan subalgebra,~$C \subset \frakh$ an open Weyl chamber,
and let~$f \in C^\infty_c(\frakg, \R)$. Then the function~$u_f : C \to \mathbb{R}$ defined by
\begin{align*}
u_f(h) := \mathcal{D}_{O_h}(f)
\end{align*}
is in~$C^{\infty}(\overline{C},\R)$.
\end{theorem}
\begin{remark}
In fact,~$u_{f}$ extends to a Schwartz function on the connected component of the non-zero sets
of the singular imaginary roots.
\end{remark}
\begin{remark}
The definition of the invariant integral in \cite{harish1957fourier} is, up to a positive scalar, equivalent to ours by Proposition~\ref{prop:RadonNikodymDifferentOrbits}.
\end{remark}
\begin{proposition}
Let~$C$ be an open cone in~$V$, and let
$u \in C^{m+1}(\overline{C}, \R)$ with~${D^{m-1}u(0) = 0}$.
Then for any compact~$K\subseteq \overline{C}$,
\[
    \lim_{\varepsilon \rightarrow 0^{+}} \varepsilon^{-m}u(\varepsilon v) = \frac{1}{m!}\partial^{m}_{v}u(0)
\]
uniformly for~$v\in K$.
\end{proposition}
\begin{proof}
Let~$p(t):=u(tv)$. Then since~$p \in C^{m+1}([0,1],\R)$, Taylor's Theorem yields
$p(\varepsilon) = \frac{1}{m!}\varepsilon^m\partial_{v}^mu(0) + R$, where~$R = \frac{1}{(m+1)!}\varepsilon^{m+1}\partial_{v}^{m+1}u(\theta v)$
for some~$\theta \in [0,\varepsilon]$. Since~$(v,w) \mapsto  \partial_{v}^{m+1}u(w)$
is uniformly bounded on~$K \times K$, the result follows.
\end{proof}
Fix again an open Weyl chamber~$C$ of some Cartan subalgebra~$\frakh \subset \frakg$. From Theorem~\ref{thm:HarrisLimitOfOrbits}, Theorem~\ref{thm:VaradarajanContinuouslyExtensible}, and the fact that integrals over compact domains commute with uniform limits, it immediately follows that
\begin{align*}
\lim_{\epsilon \to 0}
\int_{K \cap C}
\epsilon^{-m} \mathcal{D}_{O_{\epsilon h}}(f)
w(h)
d \Lebesgue_\frakh(h)
=
\sum_{O_X}
\int_{K \cap C}
\Vol(S_X \cap O_h)
\mathcal{D}_{O_{X}}(f)
w(h)
d \Lebesgue_\frakh(h),
\end{align*}
for all compact~$K \subset \frakh$, all~$f \in C^\infty_c(\frakg)$ and all continuous functions~$w \colon \fh \rightarrow \R$. In particular, the integral on the right-hand side is well-defined.
A similar statement holds with~$\mathcal{D}_{O_X}(f)$ replaced by~$\mu_{O_X}(B_{R}(0))$.
\begin{corollary}
\label{cor:LimitCommutesWithIntegral}
Let~$\frakh \subset \frakg$ be a Cartan subalgebra,~$K \subset \frakh$ a compact subset, and~$C \subset \frakh$ an open Weyl chamber. Let~$w \colon \fh \rightarrow \R$ be a continuous function, and let~$R > 0$. Then
\begin{equation*}
\begin{aligned}
\lim_{\epsilon \to 0}
\int_{K \cap C}
\epsilon^{-m}
&
\mu_{O_{\epsilon h}}(B_R(0))
w(h)
d \Lebesgue_\frakh(h)
\\
&=
\sum_{O_X}
\int_{K \cap C}
 \Vol(S_X \cap O_h)
 \mu_{O_{X}}(B_R(0))
 w(h)
 d \Lebesgue_\frakh(h),
 \end{aligned}
\end{equation*}
with the sum over the~$O_X$ as in Theorem \ref{thm:HarrisLimitOfOrbits}, and the right hand side is integrable.
\end{corollary}
\begin{proof}
The following proof is essentially taken from \cite[Lem 4.1]{kallenberg2017random}.
Choose sequences of functions~$\{f_n \in C^\infty_c(\frakg) \}_{n \geq 1}, \{g_n \in C^\infty_c(\frakg) \}_{n \geq 1}$ with monotone, pointwise convergence
\begin{align*}
f_n \nearrow 1_{B_R(0)}, \quad g_n \searrow 1_{\overline{B_R(0)}}.
\end{align*}
Let~$k = \dim O_X / 2$ for any of the orbits~$O_X$ in the sum. Since~$w$ can be written as the difference of two nonnegative functions, we may assume without loss of generality that
$w$ is nonnegative.
Then, for all~$n \in \mathbb{N}$, we have
\begin{equation}
\label{eq:InequalityForIndicators}
\begin{aligned}
(2 \pi)^{2k}
\sum_{O_X}
 &\int_{K \cap C}
 \Vol(S_X \cap O_h)
 \mathcal{D}_{O_{X}}(f_n)
 w(h)
 d \Lebesgue_\frakh(h)
 \\
 \leq
 \liminf_{\epsilon \to 0}
  &\int_{K \cap C}
 \epsilon^{-m}
 \mu_{O_{\epsilon h}}(B_R(0))
 w(h)
 d \Lebesgue_\frakh(h)
 \\
 \leq
  \limsup_{\epsilon \to 0}
  &\int_{K \cap C}
  \epsilon^{-m}
 \mu_{O_{\epsilon h}}(\overline{B_R}(0))
 w(h)
 d \Lebesgue_\frakh(h)
 \\
 \leq
(2 \pi)^{2k}
\sum_{O_X}
 &\int_{K \cap C}
 \Vol(S_X \cap O_h)
 \mathcal{D}_{O_{X}}(g_n)
 w(h)
 d \Lebesgue_\frakh(h).
 \end{aligned}
\end{equation}
Note that~$\mu_{O_X}(B_R(0)) = \mu_{O_X}(\overline{B_R(0)})$ for all nilpotent elements~$X$.
Indeed, since~$O_X$ is a cone, the boundary~$\partial B_{R}(0)$ intersects~$O_{X}$ transversally, and
the intersection~$O_{X} \cap \partial B_{R}(0) \subset O_{X}$
is either empty (when~$X = 0$) or a submanifold of codimension at least one (when~$X \neq 0$). Its Liouville measure is thus zero by Sard's Theorem.
Finally, the statement follows by taking the monotone limits~$n \to \infty$ in \eqref{eq:InequalityForIndicators}.
\end{proof}
\section{Proof of the key Lemma \ref{KeyLemma} concluding Theorem B}\label{Sect=KeyLemma}
\label{sec:MainTheorem}
In all that follows, we will fix a maximal set of mutually nonconjugate,~$\theta$-stable CSAs~$\frakh_1,\dots,\frakh_n$ of~$\frakg$, with associated Cartan subgroups~$H_1, \dots, H_n \subset G$ defined by~$H_i := Z_G(\frakh_i)$. We denote by
$W_j := N_G(\frakh_j)/H_j$ the Weyl group associated to~$\fh_j$.
We will need the following Lie algebraic version of the well-known Harish-Chandra integral formula \cite[Lemma 41]{harish1965invariant},
see e.g.~\cite[Part I, Section 3, Lemma 2]{varadarajan2006harmonic}.
\begin{lemma}
\label{lem:HarishChandraFormula}
Let~$\frakg$ be a reductive Lie algebra with connected adjoint group~$G$, and let~$f \in L^1(\frakg, \Lebesgue_\frakg)$ be an integrable function on~$\frakg$.
Then, for~$j = 1,\dots,n$, there are~$G$-invariant volume forms~$\Vol_{G/H_j}$ on~$G/H_j$, so that
\begin{align*}
\int_{\frakg} f(x) d\Lebesgue_\frakg(x) =
\sum_{j=1}^n
\frac{1}{|W_j|}
\int_{\frakh_j}
\left(
\int_{G/H_j} f(\Ad_g h) \Vol_{G/H_j}([g]) \right)
|\pi_{j}(h)|^2 d \Lebesgue_{\frakh_j}(h).
\end{align*}
Here,~$|\pi_j|$ is a product of positive roots of~$(\frakg_\C,(\frakh_j)_\C)$ as in Proposition~\ref{prop:RadonNikodymDifferentOrbits}.
\end{lemma}
\begin{lemma}
Let~$\frakh \subset \frakg$ be a CSA, and define
\begin{align*}
V_\epsilon^{\frakh} := \Ad_G B_\epsilon(0) \cap \frakh.
\end{align*}
Then for all~$\epsilon, R > 0$ we have
\begin{align}
\label{eq:RepresentIntegrationAsDisjointUnion}
(V_{\epsilon, R})_\reg = \bigsqcup_{j=1}^n \{\Ad_g h : h \in (V_\epsilon^{\frakh_j})_\reg, \, \Ad_{g}(h) \in B_R(0) \}.
\end{align}
\end{lemma}
\begin{proof}
To see that the right hand side is indeed a disjoint union, suppose that~$\Ad_g h = \Ad_{g'} h'$ for regular elements~$h\in \frakh_i$,~$h' \in \frakh_{j}$ and~$g,g' \in G$. Since~$h$ and~$h'$ are regular,~$\frakh_i = Z_\frakg(h)$ and~$\frakh_j = Z_\frakg(h')$. Since
$h$ is conjugate to~$h'$,~$\frakh_{i}$ is conjugate to~$\frakh_{j}$.
As the various CSA's are mutually nonconjugate, we conclude that~$\frakh_{i} = \frakh_{j}$.
\\
$\subset$:
Let~$\Ad_g x \in (V_{\epsilon, R})_\reg$ with~$g \in G$ and~$x \in B_\epsilon(0)$. Since~$x$ is regular, it lies in a unique CSA, and is conjugate to an element~$h \in \frakh_j$ for some~$j$, i.e.
\begin{align*}
\exists g' \in G : \Ad_{g'} x = h \in \frakh_j.
\end{align*}
By Lemma~\ref{lem:PropertiesRegularElements}, the orbit of a regular element consists only of regular elements. Hence~$h \in (V_\epsilon^{\frakh_j})_\reg$ and~$\Ad_{g (g')^{-1}} h = \Ad_g x \in B_R(0)$. Hence~$\Ad_g x$ lies in the right hand side.
\\
$\supset$: Let~$\Ad_g h$ lie in the right hand side. Because~$h \in \Ad_G B_\epsilon(0)$, we can write~$h = \Ad_{g'} x$ for some~$g' \in G$ and~$x \in B_\epsilon(0)$. But then~$\Ad_g h = \Ad_{g g'} x \in \Ad_G B_\epsilon(0)$, and since~$h$ was regular, so is every element in its orbit, and we have~$\Ad_g h \in (V_{\epsilon, R})_\reg$.
\end{proof}
\begin{lemma}
\label{lem:VepsIsCompact}
Let~$\frakh$ be a~$\theta$-invariant CSA.
For all~$\epsilon > 0$, the sets~$V_\epsilon^{\frakh} = \Ad_G B_\epsilon(0) \cap \frakh$ are bounded.
\end{lemma}
\begin{proof}
With respect to the inner product~$\kappa_{\theta} \colon \frakg \times \frakg \rightarrow \R$, we have
\begin{align*}
\kappa_{\theta} (\ad_x y, z) = - \kappa ([x,y], \theta z)
&=
\kappa(y, [x,\theta z])
\\&=
\kappa(y, \theta([\theta x ,z])) = - \kappa_{\theta} (y, \ad_{\theta x} z).
\end{align*}
It follows that~$\ad_{h}^{*} = -\ad_{\theta h}$. Since~$[h, \theta h] = 0$,  the operator~$\ad_{h} \colon \frakg \rightarrow \frakg$ is normal, and its operator norm~$\|\ad_{h}\|$ with respect to the inner product~$\kappa_{\theta}$ satisfies
\[
\|\ad_{h}\| = \max_{\lambda \in \mathrm{Spec}(\ad_{h})}|\lambda| \,.
\]
Let~$g\in G$ such that~$\Ad_{g}(h) \in B_{\varepsilon}(0)$. Since~$\ad_{\Ad_{g}h}$ is conjugate to~$\ad_{h}$, it is a normal operator with the same eigenvalues, so in particular
$\|\ad_{h}\| = \|\ad_{\Ad_{g}h}\|$. So since the operator norm is bounded on~$B_{\varepsilon}(0)$, it is bounded on
$\Ad_{G}(B_{\varepsilon}(0))\cap \frakh$ as well, and the latter is a bounded subset of~$\frakh$.
\end{proof}
Recall that~$V_{\epsilon, R}$ differs from~$(V_{\epsilon, R})_\reg$ only in a Lebesgue null set.
 Also, we can write every~$ (\frakh_j)_\reg$ as the union of its open Weyl chambers~$C_{j,r}$.
Hence, we can use Lemma~\ref{lem:HarishChandraFormula} and the decomposition \eqref{eq:RepresentIntegrationAsDisjointUnion} to find
\begin{equation}
\begin{aligned}
\label{eq:IntegralDecomposition}
&\Lebesgue(V_{\epsilon, R}) =
\sum_{j=1}^n
\frac{1}{|W_j|}
\int_{(V_\epsilon^{\frakh_j})_\reg}
\left(
\int_{
\{[g] \in G/H_j : \| \Ad_g h\| \leq R \}
}
\Vol_{G/H_j}
\right)
|\pi_j(h)|^2
d \Lebesgue_{\frakh_j}(h)
\\&=
\sum_{j=1}^n
\sum_{C_{j,r} \subset \frakh_j}
\frac{1}{|W_j|}
\int_{V_\epsilon^{\frakh_j} \cap C_{j,r}}
\left(
\int_{
\{[g] \in G/H_j : \| \Ad_g h\| \leq R \}
}
\Vol_{G/H_j}
\right)
|\pi_j(h)|^2
d \Lebesgue_{\frakh_j}(h).
\end{aligned}
\end{equation}
To simplify, we will fix a single CSA~$\frakh_j$ and a single open Weyl chamber~$C_{j,r}$, and suppress the indices:
\begin{align}
\label{eq:NotationLeftOutIndex}
H := H_j, \quad
\pi := \pi_j,\quad
\frakh := \frakh_j, \quad
\Vol_{G/H} := \Vol_{G/H_j}, \quad
C := C_{j,r}.
\end{align}
For now, let us look at the single summand of \eqref{eq:IntegralDecomposition} corresponding to~$\frakh$ and~$C$.
\begin{lemma}
\label{lem:BigCalculation}
Fix the notation as in \eqref{eq:NotationLeftOutIndex}, let~$m := m(h)$ as in Definition~\ref{def:definitionofm} for an arbitrary~$h \in C$, and~$R > 0$ arbitrary. Then there is some~$c \neq 0$ which does not depend on~$R$ so that
\begin{align*}
\lim_{\epsilon \to 0^+}
\epsilon^{-m- \dim \frakh - |\Delta^+|}
&\int_{V_\epsilon^{\frakh} \cap C}
\left(
\int_{
\{[g] \in G/H : \| \Ad_g h\| \leq R \}
}
\Vol_{G/H}
\right)
|\pi(h)|^2
d \Lebesgue_{\frakh}(h)
\\
&=
c
\sum_{O_X}
\int_{V_1^{\frakh} \cap C}
\Vol(S_{X} \cap O_h)
\mu_{O_{X}}(B_{R}(0)) \pi(h)
d \Lebesgue_{\frakh}(h),
\end{align*}
where the sum is taken over all nilpotent orbits $O_X$ contained in $\mathcal{N}_h$ for an arbitrary $h \in C$, cf. Remark~\ref{rmk:NumberMIsConstantAlongConnectedComponent}.
\end{lemma}
\begin{proof}
Recall the notation~$\mu_{O_{h}}$ for the measure defined in Section~\ref{Sec:LimitsOfOrbits}, and that~$\pi(h)$ was defined in Equation \eqref{eq:TransformationOfVolumeFormOnG/H} as the volume density function of the orbit-stabilizer-diffeomorphism~$G/H \to O_h$ with respect to a fixed invariant volume form on~$G/H$ and the KKS volume form on~$O_h$.
Using this property of~$\pi(h)$, we find that there exists some~$c \neq 0$ depending only on the choice of invariant measure~$\Vol_{G/H}$ on~$G/H$ with:
\begin{align*}
\int_{V_\epsilon^{\frakh} \cap C}
\Bigg(
\int_{
\{[g] \in G/H : \| \Ad_g h\| \leq R \}
}
&\Vol_{G/H}
\Bigg)
|\pi(h)|^2 d \Lebesgue_{\frakh}(h)
\\
&=
c
\cdot
\int_{V_\epsilon^{\frakh} \cap C}
\left(
\int_{B_R(0) \cap O_h } \Vol_{O_h} \right)
\pi(h)
d \Lebesgue_{\frakh}(h)
\\
&=
c
\cdot
\int_{V_\epsilon^{\frakh} \cap C}
\mu_{O_{h}}(B_R(0))
\pi(h)
d \Lebesgue_{\frakh}(h)
\\
&=
\epsilon^{\dim \frakh + |\Delta^+|} c
\cdot
\int_{V_1^{\frakh} \cap C}
\mu_{O_{ \epsilon h}}(B_R(0))
\pi(h)
d \Lebesgue_{\frakh}(h).
\end{align*}
In the last step, we used that~$V_{\varepsilon}^{\fh} = \epsilon V_{1}^{\fh}$, that
$d\Lebesgue_{\fh}(\varepsilon h) = \varepsilon^{\dim \fh} d\Lebesgue_{\fh}(h)$, and that~$\pi(\varepsilon h) = \varepsilon^{|\Delta_+|}\pi(h)$.
By Corollary~\ref{cor:LimitCommutesWithIntegral} and compactness of~$V_1^{\frakh} \cap C$ (Lemma~\ref{lem:VepsIsCompact}), we have:
\begin{align*}
\lim_{\epsilon \to 0}
\int_{V_1^{\frakh} \cap C}
\epsilon^{-m}
\mu_{O_{\epsilon h}}
&(B_{R}(0))
\pi(h)  d \Lebesgue_{\frakh}(h)
\\
&=
 \sum_{O_X}
\int_{V_1^{\frakh} \cap C}
\Vol(S_{X} \cap O_h)
\mu_{O_{X}}(B_{R}(0))
\pi(h)
d \Lebesgue_{\frakh}(h).
\end{align*}
This shows the statement.
\end{proof}
Finally, we use this to prove the key lemma:
\begin{proof}[Proof of key Lemma \ref{KeyLemma}]
Let~$O_X \subset \frakg$ be a nilpotent orbit of dimension~$d$.
By Lemma~\ref{lem:AllNilpotentOrbitsShowUpSomewhere} and Remark~\ref{rmk:NumberMIsConstantAlongConnectedComponent}, there is some CSA~$\frakh \subset \frakg$ and some open Weyl chamber~$C \subset \frakh_\reg$ so that~$O_X \subset \mathcal{N}_h$ for all~$h \in C$. Since~$O_X$ is of maximal dimension, the number~$m := m(x)$ from Definition~\ref{def:definitionofm}
is minimal among~$m(x')$ for all~$x'\in \fg_{\reg}$.
Then, by \eqref{eq:IntegralDecomposition} and Lemma~\ref{lem:BigCalculation}, there are numbers~$c_{j,r} \neq 0$,  independent of~$R$, such that
\begin{align*}
\lim_{\epsilon \to 0^+}
&\epsilon^{-m - \dim \frakh - |\Delta^+|}
\Lebesgue(V_{\epsilon, R})
\\
&=
\sum_{j=1}^n
\sum_{C_{j,r} \subset \frakh_j}
\sum_{O_X}
\frac{1}{|W_j|}
c_{j,r} \cdot
\int_{V_1^{\frakh} \cap C}
\Vol(S_{X} \cap O_h)
\mu_{O_{X}}(B_{R}(0))
\pi_j(h)
d \Lebesgue_{\frakh_j}(h),
\end{align*}
where the sum over the~$O_X$ is carried out over all nilpotent orbits~$O_X \subset \mathcal{N}_h$ of dimension~$d$.
Note that this sum is independent of the choice of~$h \in C_{j,r}$ for fixed~$j$ and~$r$, cf.~Remark~\ref{rmk:NumberMIsConstantAlongConnectedComponent}.
\\
All summands in the above sum are positive as they arise from volumes of subsets of $\frakg$, and thus the total sum is nonzero. Lastly, the sum is homogeneous of degree~$d/2$ in~$R$ by Lemma~\ref{lem:HomogeneityOfNilpotentVolumes}. Hence we may consider their quotient and conclude the proof:
\begin{align*}
\lim_{\epsilon \to 0^+}
\frac
{\Lebesgue(V_{\epsilon,\rho R})}
{\Lebesgue(V_{\epsilon, R})}
=
\frac
{\lim_{\epsilon \to 0^+}  \epsilon^{-m - \dim \frakh - |\Delta^+|} \Lebesgue(V_{\epsilon,\rho R})}
{\lim_{\epsilon \to 0^+}  \epsilon^{-m - \dim \frakh - |\Delta^+|} \Lebesgue(V_{\epsilon, R})}
= \rho^{d/2}.
\end{align*}
\end{proof}


\noindent {\it Conflict of interest and data availability statement.} On behalf of all authors, the corresponding author states that there is no conflict of interest. Data sharing is not applicable to this article as no datasets were generated or analysed during the current study.

\end{document}